\documentclass[12pt]{amsart}

\usepackage{upref,amssymb,bm,mathrsfs}
\usepackage{enumerate}
\usepackage{mathtools}
\usepackage[shortlabels]{enumitem}

\usepackage[centering,width=6in,truedimen]{geometry}
\usepackage[colorlinks,pagebackref,pdfstartview=FitH]{hyperref}

\hypersetup{
    colorlinks,
    citecolor=blue,
    filecolor=black,
    linkcolor=red,
    urlcolor=black
}

\theoremstyle{plain}
\newtheorem{thm}{Theorem}[section]
\newtheorem{prop}[thm]{Proposition}
\newtheorem{lem}[thm]{Lemma}
\newtheorem{cor}[thm]{Corollary}

\newtheorem*{ques*}{Question}

\theoremstyle{definition}

\theoremstyle{remark}
\newtheorem{rem}[thm]{Remark}

\newcommand{\N}{\mathbb N}
\newcommand{\Q}{\mathbb Q}
\newcommand{\R}{\mathbb R}
\newcommand{\Z}{\mathbb Z}
\numberwithin{equation}{section}
\renewcommand{\epsilon}{\varepsilon}

\newcommand{\bb}{\mathbf{b}}
\newcommand{\cC}{\mathcal{C}}

\newcommand{\cB}{\mathcal{B}}
\newcommand{\cL}{\mathcal{L}}
\newcommand{\bc}{\mathbf{c}}
\newcommand{\ba}{\mathbf{a}}

\title[A quantitative framework for sets of exact approximation order]{A quantitative framework for sets of exact approximation order by rational numbers}

\author{Simon Baker}

\address{
Mathematical Sciences\\
Loughborough University\\
Loughborough\\
LE11 3TU, UK
}
\email{simonbaker412@gmail.com}

\author{Benjamin Ward}

\address{
 Department of Mathematics\\
 University of York\\
 Heslington\\
 YO10 5DD, UK
}
\email{benjamin.ward@york.ac.uk, ward.ben1994@gmail.com}



\begin{document}

\begin{abstract}
	In this paper we study a quantitative notion of exactness within Diophantine approximation. Given $\Psi:(0,\infty)\to (0,\infty)$ and $\omega:(0,\infty)\to (0,1)$ satisfying $\lim_{q\to\infty}\omega(q)=0$, we study the set of points, which we call $E(\Psi,\omega)$, that are $\Psi$-well approximable but not $\Psi(1-\omega)$-well approximable. We prove results on the cardinality and dimension of $E(\Psi,\omega)$. In particular we obtain the following general statements:
    \begin{itemize}[leftmargin=2ex]
        \item For any $\omega:(0,\infty)\to (0,1)$ and $\tau>2$ there exists $\Psi:(0,\infty)\to (0,\infty)$ such that $\lim_{q\to\infty}\frac{-\log \Psi(q)}{\log q}=\tau$ and $E(\Psi,\omega)\neq\emptyset.$
        \item Under natural monotonicity assumptions on $\Psi$ and $\omega,$ we prove that if $\omega$ decays to zero sufficiently slowly (in a way that depends upon $\Psi$) then $E(\Psi,\omega)$ is uncountable. Moreover, under further natural assumptions on $\Psi$ we can calculate the Hausdorff dimension of $E(\Psi,\omega)$.
    \end{itemize}
Our main result demonstrates a new threshold for the behaviour of $E(\Psi,\omega)$. A particular instance of this threshold is illustrated by considering functions of the form $\Psi_{\tau}(q)=q^{-\tau}$ when $\tau\in \N_{\geq 3}$. For these functions we prove the following:
    \begin{itemize}[leftmargin=2ex]
        \item If $\omega(q)= Cq^{-\tau(\tau-1)}$ for some sufficiently large $C$ or $\omega(q)=q^{-\tau'}$ for some $\tau'<\tau(\tau-1),$ then $E(\Psi_{\tau},\omega)$ is uncountable and we calculate its Hausdorff dimension.
        \item If $\omega(q)< cq^{-\tau(\tau-1)}$ for some $c\in (0,1)$ for all $q$ sufficiently large then $E(\Psi_{\tau},\omega)=\emptyset.$ 
    \end{itemize}  
\end{abstract}

\maketitle
\section{Introduction}
Diophantine approximation in $\R$ is concerned primarily with how well real numbers can be approximated by rational numbers. 
This can be made precise within the following general framework:
Given a non-increasing function $\Psi:(0,\infty)\to (0,\infty)$ we say that $x\in \R$ is $\Psi$-approximable if 
$$\left|x-\frac{p}{q}\right|\leq \Psi(q)$$ for i.m. $(p,q)\in \mathbb{Z}\times \mathbb{N}$. Here and throughout we use i.m. as a shorthand for infinitely many. We denote the set of $\Psi$-approximable numbers by $W(\Psi)$. The sets $W(\Psi)$ are of central interest in Diophantine approximation and their properties are well understood. 


Suppose $x\in W(\Psi)$ for some $\Psi$, it is natural to ask whether $\Psi$ is in a sense the optimal approximation function for $x$. This problem was first considered from a metric perspective by G\"{u}ting in \cite{Gut}. Given $\tau\geq 2$ let $$W(\tau)=\left\{x\in \R:\left|x-\frac{p}{q}\right|\leq q^{-\tau} \textrm{ for i.m. }(p,q)\in \mathbb{Z}\times \N\right\},$$ 
and for $x\in \R$ let $V(x)=\sup\{\tau: x\in W(\tau)\}$. G\"{u}ting proved in \cite{Gut} that for $\tau\geq 2$ we have $$\dim_{H}(\{x:V(x)=\tau\})=\dim_{H}(W(\tau))=\frac{2}{\tau}\, ,$$
where $\dim_{H}$ denotes the Hausdorff dimension. That is, the set of points for which $\tau$ is the optimal exponent of approximation has the same Hausdorff dimension as the set of points that are $\tau$-approximable. Note that the formula $\dim_{H}(W(\tau))=\frac{2}{\tau}$ was established prior to this independently by Jarnik \cite{Jar} and Besicovitch \cite{Bes}. In \cite{BDV} Beresnevich, Dickinson and Velani proved an analogue of G\"{u}ting's result where the quality of approximation is viewed at a finer logarithmic scale. The authors of \cite{BDV} also introduced the following analogue of the set of badly approximable numbers: Given a non-increasing $\Psi:(0,\infty)\to (0,\infty)$ we say that $x\in \R$ is $\Psi$-badly approximable if $x\in W(\Psi)$ and there exists $c>0$ such that for all sufficiently large $q\in \N$ we have
$$\left|x-\frac{p}{q}\right|\geq  c\Psi(q)$$ for all $p\in \Z$. We let $\textrm{Bad}(\Psi)$ denote the set of $\Psi$-badly approximable points. If $x\in \textrm{Bad}(\Psi)$ then $\Psi$ can be viewed as an optimal approximation function for $x$ because $x\notin W(\Psi')$ for any $\Psi'$ satisfying $\Psi'=o(\Psi)$. In \cite{BDV} the authors posed the problem of determining the Hausdorff dimension of $\textrm{Bad}(\Psi)$ for a general $\Psi$.  Note that for any function $\Psi$ tending to zero, by Cassels' Scaling Lemma \cite{Cas}, we have that $\textrm{Bad}(\Psi)$ is a Lebesgue nullset. In this paper, the sets we will be interested in will be a subset of $\textrm{Bad}(\Psi)$ for some $\Psi$. Hence they are nullsets and so when considering the size of these sets we will use Hausdorff dimension.

A more precise framework for understanding those points for which $\Psi$ is the optimal approximation function is provided by the notion of $\Psi$-exact approximation order: We say that $x$ has $\Psi$-exact approximation order if $x\in W(\Psi)$ and there exists a function $c_{x}:\N\to (0,1)$ depending on $x$ which satisfies $\lim_{q\to \infty}c_{x}(q)=1$ and
$$\left|x-\frac{p}{q}\right|\geq  c_{x}(q)\Psi(q)$$
for all $(p,q)\in \mathbb{Z}\times \N.$ 
We let $E(\Psi)$ denote the set of points with $\Psi$-exact approximation order. If $x\in E(\Psi)$ then $\Psi$ can be viewed as the optimal approximation function for $x$ because $x\in W(\Psi)$ but $x\notin W(c\Psi)$ for any $c\in (0,1).$ By definition, for any $\Psi$ we have the inclusions
$$E(\Psi)\subset \textrm{Bad}(\Psi)\subset W(\Psi).$$ In the special case where $\Psi(q)=q^{-2}$ then $W(\Psi)$ has full Lebesgue measure and $\textrm{Bad}(\Psi)$, which is now the classical set of badly approximable points, has zero Lebesgue measure and full Hausdorff dimension. The set $E(\Psi)$ is empty. It was shown by Moreria that if one considers $\Psi_{c}(q)=cq^{-2},$ then $E(\Psi_{c})$ has positive Hausdorff dimension for $c$ sufficiently small \cite[Theorem 2]{Mor}. In the case where $\Psi(q)=o(q^{-2})$ these sets exhibit different behaviour. Jarnik showed in \cite[Satz. 6 (p.539)]{Jar} that when $\Psi(q)=o(q^{-2})$ and non-increasing then $E(\Psi)$ is non-empty. In \cite{Bug1} Bugeaud answered the problem posed by Beresnevich, Dickinson and Velani in \cite{BDV} in a strong sense and proved that when $\Psi$ is non-increasing and satisfies some additional decay assumptions then $\dim_{H}(E(\Psi))=\dim_{H}(\textrm{Bad}(\Psi))=\dim_{H}(W(\Psi))$. These additional assumptions were subsequently removed in a follow up paper by Bugeaud and Moreira \cite[Theorem 4]{BugMor} who showed that when $\Psi$ is non-increasing and $\Psi(q)=o(q^{-2}),$ then $\dim_{H}(E(\Psi))=\dim_{H}(\textrm{Bad}(\Psi))=\dim_{H}(W(\Psi)).$
For more on this topic we refer the reader to \cite{Bug2,BugMor,Mor} and the references therein. For more on the higher dimensional version of this problem we refer the reader to \cite{BadeSax,Fre,KLWZ,Joh}, and for this problem in other settings see \cite{BanGhoNan, BDGW, HeXio,Pan, Zha, ZhaZho}. One can also consider other notions of the size of $E(\Psi)$ such as the packing dimension \cite{Joh2}, or Fourier dimension \cite{FraWhe, FraWhe2}.

\subsection{A quantitative framework for exactness}
We are concerned with providing a more quantitative framework for understanding those points for which $\Psi$ is the optimal approximation function. Given a non-increasing function $\Psi:(0,\infty)\to (0,\infty)$ and a function $\omega:(0,\infty)\to (0,1)$ satisfying $\lim_{q\to\infty}\omega(q)=0,$ we associate the set of $x\in \R$ satisfying $x\in W(\Psi)$ and for all $q\in \N$ sufficiently large we have $$ \left|x-\frac{p}{q}\right| \geq\left(1-\omega(q)\right) \Psi(q)$$ for all $p\in \mathbb{Z}.$ We say that a point is of exact $(\Psi,\omega)$-order of approximation if it satisfies these two properties, and denote the set of points by $E(\Psi,\omega)$. If $x\in E(\Psi,\omega)$ then $\Psi$ is the optimal approximation function for $\Psi$ in the sense that $x\in W(\Psi)$ but $x\notin W\left((1-\omega)\Psi\right)$. Note that unlike in the definition of $E(\Psi)$ the function $\omega$ is uniform over $x\in \R$ and so provides a more quantitative framework by which we can assert that $\Psi$ is the optimal approximation function. We note that it is a consequence of our assumption $\lim_{q\to\infty}\omega(q)=0$ that for any $\Psi$ and $\omega$ we have the inclusions 
$$E(\Psi,\omega)\subset E(\Psi)\subset \textrm{Bad}(\Psi)\subset W(\Psi).$$ \par


 In \cite{AkhMos} Akhunzhanov and Moschevitin introduced the related set of points $x\in \R$ satisfying $x\in W(\Psi(1+\omega))$ and for all $q\in \N$ sufficiently large we have
\begin{equation*}
    \left|x-\frac{p}{q}\right|\geq \Psi(q)(1-\omega(q))\, 
\end{equation*}
for all $p\in \Z$. This can be considered as a symmetric version of our set $E(\Psi,\omega)$.
They proved that if $q^{2}\Psi(q)$ is non-increasing and $\omega(q)\geq C q^{2}\Psi(q)$ for all $q$ sufficiently large, for some determined fixed constant $C>0$, then the corresponding set is non-empty. A later paper by Akhunzhanov \cite{Akh} improved this further by showing that the same conclusion holds under the weaker assumption that $\omega(q)\geq C \left(q^{2}\Psi(q)\right)^{2}$ for all $q$ sufficiently large. Note that this work also studied a higher dimensional analogue of this problem, and in particular implies \cite[Theorem 2]{BadeSax} in the symmetric setting.\par 
The aforementioned results motivate the following question:
\begin{itemize}
    \item[Q1.] For which $\Psi$ and $\omega$ is $E(\Psi,\omega)$ non-empty?
\end{itemize}
At first glance it might seem reasonable to expect that $E(\Psi,\omega)$ is non-empty for every pair $(\Psi,\omega)$. For each $\Psi$ the set $W(\Psi)$ is reasonably well distributed within the unit interval (for example, each $W(\Psi)$ has full packing dimension and satisfies the large intersection property \cite{Fal1994}), so one could expect that changing the order of $\Psi$, by however small an amount, would lead to an increase in points in the set. Indeed, the non-emptiness of the sets $E(\Psi)$ supports this idea. However, this intuitive guess turns out to be incorrect. The obstruction to $E(\Psi,\omega)$ being non-empty comes from the fact $E(\Psi,\omega)$ is not just a $\limsup$ set, it is the intersection of a $\limsup$ set with a $\liminf$ set. For all $q$ sufficiently large, we have to be careful to rule out rational approximations that satisfy $ |x-\frac{p}{q}|<\Psi(q)\left(1-\omega(q)\right)$. In Theorem~\ref{thm:empty} below we show that there are large families of pairs of functions $(\Psi,\omega)$ such that $E(\Psi,\omega)=\emptyset$. In other words, we show that there exists functions $\Psi,\Phi:(0,\infty)\to(0,\infty)$ such that
\begin{equation*}
   \Psi(q)<\Phi(q) \quad \text{ for all }\, \, q\in(0,\infty)\, \quad \text{ and } \quad  W(\Psi)=W(\Phi).
\end{equation*}
See also Theorem~\ref{thm: non-empty2} below. During the preparation of this paper a preprint of Moshchevitin and Pitcyn \cite{MoshPit} was released which shows a similar result. We reserve further comments to Remark~\ref{rem: MoshPit result} below. As highlighted in \cite{MoshPit} the only previously known approximation functions where this phenomenon occurs is when one considers approximation functions related to the gaps in the Markoff-Lagrange spectrum \cite{CusFla, Mat}, for example $\Psi(q)=\tfrac{1}{\sqrt{5}q^{2}}$ and $\Phi(q)=\tfrac{1}{(\sqrt{5}+\varepsilon)q^{2}}$ for $\varepsilon>0$. See also \cite{Han}.\par

If the set $E(\Psi,\omega)$ is non-empty, one can ask about the finer metric properties of $E(\Psi,\omega)$. We can easily see via Cassels' Scaling Lemma that $E(\Psi,\omega)$ is always a Lebesgue nullset, so we consider its Hausdorff dimension. It is an immediate consequence of the definition of $E(\Psi,\omega)$ that if $x\in E(\Psi,\omega)\cap[0,1]$ then $x$ will belong to one of $2(q+1)$ intervals of size $\Psi(q)\omega(q)$ for infinitely many $q\in \N.$ The following upper bound for the Hausdorff dimension of $E(\Psi,\omega)$ follows from this observation:
\begin{equation}
\label{eq:Hausdorff upper bound}
\dim_{H}(E(\Psi,\omega))\leq \inf\left\{s\geq 0: \sum_{q=1}^{\infty}q(\Psi(q)\omega(q))^{s}<\infty\right\}.
\end{equation}This leads us to our next question:
\begin{itemize}
     \item[Q2.] For which $\Psi$ and $\omega$ do we have equality in \eqref{eq:Hausdorff upper bound}?
\end{itemize}

In Theorems~\ref{thm:main}-\ref{thm:main2} below we give sufficient conditions on $\omega$ (depending on $\Psi$) ensuring that we have equality in \eqref{eq:Hausdorff upper bound}. Jarnik's theorem \cite{Jar}, together with \cite{BugMor}, imply that for a general family of $\Psi,$ the Hausdorff dimension of $E(\Psi), \textrm{Bad}(\Psi)$ and $W(\Psi)$ coincide with the right hand side of \eqref{eq:Hausdorff upper bound} when $\omega$ is taken to be a strictly positive constant function. Thus, if $\omega$ decays to zero sufficiently quickly, we should not expect the Hausdorff dimension of $E(\Psi,\omega)$ to coincide with that of $E(\Psi), \textrm{Bad}(\Psi)$ and $W(\Psi)$. Conversely, if $\omega$ decays to zero sufficiently slowly then the dimensions should coincide. This is shown to be the case in Corollary~\ref{cor:dimension coincidence} below. \par 

In this paper we address both Q1 and Q2. In the special case when $\Psi$ satisfies an additional divisibility criteria, then we obtain complete answers to both of these questions (see Theorems \ref{thm: non-empty2} and \ref{thm:main2}).

\subsection{Main results: Q1}
We prove the following results on $E(\Psi,\omega)$.

\begin{thm} \label{thm: non-empty}
    Suppose that $\lim_{q\to\infty}\limits q^{2}\Psi(q)=\lim_{q\to\infty}\limits\omega(q)=0$ and $q\to q^{2}\Psi(q)$ is non-increasing. Furthermore, suppose that 
    \begin{equation} \label{lower bound for corollary}
        \omega(q)\geq \frac{45}{\Psi(q)}\Psi\left(\frac{1}{2q\Psi(q)}\right)
    \end{equation}
    for all $q$ sufficiently large. Then $E(\Psi,\omega)$ is uncountable.
\end{thm}
Theorem \ref{thm: non-empty} allows us to conclude that $E(\Psi,\omega)\neq\emptyset$ and uncountable under a weaker assumption on $\omega$ then that imposed by Akhunzhanov in \cite{Akh}. The following corollary makes our improvement on their result clear in the case of power functions.
\begin{cor} \label{cor: non-empty}
    Suppose that $\Psi(q)=q^{-\tau}$ for some $\tau>2$ and that $\omega(q)=Cq^{-\tau(\tau-2)}$ for some $C>0$. Then $E(\Psi,\omega)$ is uncountable if $C$ is sufficiently large.
\end{cor}

\begin{rem}
    If $\omega,\omega'$ satisfy $\omega(q)\leq \omega'(q)$ for all $q$ sufficiently large then $E(\Psi,\omega)\subseteq E(\Psi,\omega')$ for any $\Psi$.  Thus Theorem~\ref{thm: non-empty} is equivalent to the following statement: If $\Psi$ satisfies the conditions of Theorem~\ref{thm: non-empty} and $\omega:(0,\infty)\to(0,\infty)$ is given by $$\omega(q)= \frac{45}{\Psi(q)}\Psi\left(\frac{1}{2q\Psi(q)}\right),$$
    then $E(\Psi,\omega)$ is uncountable. This may be viewed as the correct way to present Theorem \ref{thm: non-empty} However, to avoid this theorem being misinterpreted as applying to only one function $\omega$, we chose to formulate it with condition \eqref{lower bound for corollary}. We will adopt this convention throughout.
\end{rem}

\begin{rem}
Equation \eqref{cor: non-empty} may not appear natural. To explain where this bound comes from and to aid the reader with our later proofs, we provide a brief heuristic idea for where this condition appears. It is good to keep in mind the ideas discussed in the following when reading the technical proofs that appear in Section~\ref{sec:Technical results}. Let $\Psi$ satisfy the assumptions of Theorem~\ref{thm: non-empty} and let us try to see what sort of bound we should expect on $\omega$. Suppose $E(\Psi,\omega)\neq \emptyset$ and let $x\in E(\Psi,\omega)$. Then there exists $(p,q)\in \Z\times \N$, with $q$ sufficiently large so that $\Psi(q)<\frac{1}{2q^{2}}$, for which 
\begin{equation*}
    (1-\omega(q))\Psi(q) \leq \left|x-\frac{p}{q}\right|\leq\Psi(q)\, .
\end{equation*}
Since $\Psi(q)<\frac{1}{2q^{2}}$ we must have that $\frac{p}{q}=\frac{p_{n}}{q_{n}}$ is a convergent of $x$ by Legendre's theorem (see Lemma \ref{lem:Legendre}) and the next partial quotient $a_{n+1}$ must be of the order $\frac{1}{q^{2}\Psi(q)}$ (see Section~\ref{sec:preliminaries} for definitions and relevant background on continued fractions). Note that $a_{n+1}$ can not be too much bigger than this, otherwise we would have $\left|x-\frac{p}{q}\right|<(1-\omega(q))\Psi(q),$ which is not allowed. Using properties of continued fractions and our bound on $a_{n+1},$ we have that the denominator $q_{n+1}$ of the next convergent of $x$ must be of the order $\frac{1}{q\Psi(q)}$. Now let us suppose we are in the worst case scenario where $\frac{p_{n+1}}{q_{n+1}}$ lies in the interval $[\frac{p}{q}+(1-\omega(q))\Psi(q), \frac{p}{q}+\Psi(q)]$. Hence, since $x$ also belongs to this interval, we have that
\begin{equation} \label{assumption for contradiction}
    \left|x-\frac{p_{n+1}}{q_{n+1}}\right| < \omega(q)\Psi(q)\, .
\end{equation}
Now, considering the next partial quotient $a_{n+2},$ we must suppose again that $a_{n+2}$ is bounded from above by some constant multiple of $\frac{1}{q_{n+1}^{2}\Psi(q_{n+1})}$, otherwise $\frac{p_{n+1}}{q_{n+1}}$ will approximate $x$ too well. Using properties of convergents of continued fractions and our bound on $a_{n+2}$, we have that
\begin{equation*}
    \left|x-\frac{p_{n+1}}{q_{n+1}}\right|\geq \frac{1}{3a_{n+2}q_{n+1}^{2}}\gg  \Psi(q_{n+1}) \gg \Psi\left(\frac{1}{q\Psi(q)}\right)\, .
\end{equation*}
Combining this with \eqref{assumption for contradiction} we must have that
\begin{equation*}
    \omega(q)\Psi(q) \gg \Psi\left(\frac{1}{q\Psi(q)}\right)
\end{equation*}
which is essentially our condition \eqref{lower bound for corollary}. 
\end{rem}
\vspace{0.5cm}

Note that a crucial part in the above heuristics was that the next convergent of our point of interest lies inside the interval to be avoided. If we impose suitable divisibility conditions on our approximation function $\Psi$ we can do much better than assuming this worst case scenario. To that end, we say a function $\Psi$ satisfies the \textit{square divisibility condition} if for every $q\in \N$ we have $\Psi(q)^{-1}=0 \mod q^{2}$.
Examples of functions satisfying the square divisibility condition include $\Psi(q)=q^{-\tau}$ for some $\tau\in \N_{\geq 2}$. Moreover, for any $\tau\geq 2$ there exists a function $\Psi$ satisfying the square divisibility condition and $\lim_{q\to \infty}\frac{-\log \Psi(q)}{\log q}=\tau$. Restricting to approximation functions satisfying the square divisibility condition we are able to prove the following result.
\begin{thm}
\label{thm: non-empty2}
Suppose that $\lim_{q\to\infty}\limits q^{2}\Psi(q)=\lim_{q\to\infty}\limits\omega(q)=0$ and $q\to q^{2}\Psi(q)$ is non-increasing. Furthermore, suppose that $\Psi$ satisfies the square divisibility condition. Then the following statements hold:
\begin{enumerate}
    \item $E(\Psi,\omega)=\emptyset$ if $$\omega(q) < \frac{1-\omega\left(\frac{1}{\Psi(q)}\right)}{\Psi(q)}\Psi\left(\frac{1}{\Psi(q)}\right)$$ for all $q$ sufficiently large. In particular, if $$\omega(q) < \frac{c}{\Psi(q)}\Psi\left(\frac{1}{\Psi(q)}\right)$$ for all $q$ sufficiently large for some $c\in (0,1),$ then $E(\Psi,\omega)=\emptyset.$
    \item $E(\Psi,\omega)$ is non-empty and uncountable if $$\omega(q)\geq \frac{45}{16\Psi(q)}\Psi\left( \frac{1}{8\Psi(q)} \right)$$ for all $q$ sufficiently large.
\end{enumerate}
\end{thm}

The second statement in Theorem \ref{thm: non-empty2} allows us to conclude that $E(\Psi,\omega)$ is uncountable under a weaker assumption on $\omega$ than that appearing in Theorem \ref{thm: non-empty}. The first statement of this theorem shows that for certain approximation functions $\Psi$ the set $E(\Psi,\omega)$ is empty if $\omega$ decays to zero sufficiently quickly. Note that for many natural choices of $\Psi$ the thresholds for non-empty and empty appearing in Theorem~\ref{thm: non-empty2} differ by a multiplicative constant. 
We refer the reader to Corollary \ref{cor:polynomial corollary} where we specialise to functions of the form $q\to q^{-\tau}.$ In this special case, the critical threshold identified by Theorem \ref{thm: non-empty2} is particularly clear. \par 

Statement 1 from Theorem \ref{thm: non-empty2} can be deduced from the following more general result. This result shows that the family of approximation functions satisfying the square divisibility condition are not the only ones that require a bound on the decay rate of $\omega$ to ensure that $E(\Psi,\omega)$ is non-empty.

\begin{thm} \label{thm:empty}
    Assume that $\Psi$ is such that there exists functions $d:(0,\infty)\to (0,\infty)$ and $e:(0,\infty)\to (0,\infty)$ such that $\Psi=\frac{d}{e}$ and $d(q),e(q)\in \N$ for all $q\in \N$. Then the following statements are true:
    \begin{enumerate}
               \item If $\omega:(0,\infty)\to (0,\infty)$ is such that 
               $$\omega(q)<\frac{1-\omega(qe(q))}{\Psi(q)}\Psi(qe(q))$$
               for all $q$ sufficiently large, then $E(\Psi,\omega)=\emptyset$.
       \item Assume $e(q)=0 \mod q$ for all $q\in \N$. If $\omega:(0,\infty)\to (0,\infty)$ is such that 
       $$\omega(q)<\frac{1-\omega(e(q))}{\Psi(q)}\Psi(e(q))$$
       for all $q$ sufficiently large, then $E(\Psi,\omega)=\emptyset.$ 
    \end{enumerate}
\end{thm}
\begin{rem} \rm
If $\Psi$ satisfies the square divisibility condition then we define $d:(0,\infty)\to (0,\infty) $ and $e:(0,\infty)\to(0,\infty)$ according to the rules $d(q)=1$ and $e(q)=\frac{1}{\Psi(q)}$ for all $q$. With this formulation it is now clear that the second statement of Theorem \ref{thm:empty} implies the first statement of Theorem \ref{thm: non-empty2}.
\end{rem}

\begin{rem}\rm \label{rem: MoshPit result}
In the preprint of Moshchevitin and Pitcyn \cite{MoshPit} it is shown that $E(\Psi,\omega)=\emptyset$ for a range of approximation functions $\Psi$ and $\omega$. In particular \cite[Theorem 2]{MoshPit} considers approximation functions of the same form as those appearing in Theorem~\ref{thm:empty}. Indeed, there is a significant overlap between these two statements. For example, when $\Psi_{\tau}(q)=q^{-\tau}$ for some $\tau\in \N_{\geq 3},$ both of these statements imply that for any $\epsilon>0$ we have $E(\Psi_{\tau},\omega)=\emptyset$  if $\omega(q)\leq (1-\epsilon)q^{-\tau(\tau-1)}$ for all $q$ sufficiently large. That being said, our proofs are different. 
\end{rem}

The proof of Theorem \ref{thm:empty} is straightforward and is given in Section \ref{sec: proof of empty}. We will briefly highlight here the mechanism that leads to $E(\Psi,\omega)$ being empty. If $x\in E(\Psi,\omega)$ then $x$ is contained in an interval of length $\omega(q)\Psi(q)$ and this interval has either $\frac{p}{q}+\Psi(q)$ or $\frac{p}{q}-\Psi(q)$ as an endpoint for infinitely many $(p,q)\in \mathbb{Z}\times \N$. If $\frac{p}{q}+\Psi(q)$ or $\frac{p}{q}-\Psi(q)$ is a rational number then this observation provides a rational approximation to $x$ with error $\omega(q)\Psi(q).$ If $\omega\cdot \Psi$ is too small relative to the denominators of these newly constructed rationals then $x$ cannot be in $E(\Psi,\omega)$. It is this observation that we exploit in our proof of Theorem \ref{thm:empty}.

Theorem~\ref{thm: non-empty} provides a lower bound for $\omega$ (depending upon $\Psi$) which guarantees that $E(\Psi,\omega)$ is uncountable. 
Contrasting this, Theorem~\ref{thm:empty} shows that for $\Psi$ satisfying $\Psi(q)\in \mathbb{Q}$ for all $q\in \N$ there exists an upper bound for $\omega$ (again, depending upon $\Psi$) that guarantees $E(\Psi,\omega)=\emptyset$. In addition, the preprint \cite{MoshPit} provides other families of functions $\Psi$ for which there exists an $\omega$ such that $E(\Psi,\omega)=\emptyset$. These observations may lead one to ask whether there exists some upper bound for $\omega$ (depending only on the decay rate of $\Psi$) that guarantees $E(\Psi,\omega)=\emptyset$ without any additional assumptions on $\Psi$. The following theorem demonstrates that this is not the case.  
\begin{thm}\label{thm: non-empty3}
    For any $\omega:(0,\infty)\to(0,\infty)$ and  $\tau>2$ there exists a non-increasing function $\Psi:(0,\infty)\to (0,\infty)$ such that $\lim_{q\to\infty} \frac{-\log \Psi(q)}{\log q}=\tau$ and $E(\Psi,\omega)\neq \emptyset$. 
\end{thm}

The following question that naturally arises from Theorems \ref{thm: non-empty2} and \ref{thm:empty}.

\begin{ques*}
Does there exist a non-increasing $\Psi:(0,\infty)\to (0,\infty)$ and $\omega:(0,\infty)\to(0,\infty)$ such that
\begin{equation*}
    \omega(q)\geq  \frac{1}{\Psi(q)}\Psi\left( \frac{1}{\Psi(q)}\right) 
\end{equation*}
for all sufficiently large $q$, yet $E(\Psi,\omega)=\emptyset$?
\end{ques*}

\subsection{Main results: Q2}
In response to Q2 we are able to prove the following statements. Essentially, they show that if we impose the additional 
assumption $\sum_{q=1}^{\infty}q\Psi(q)<\infty,$ then we can replace the uncountability conclusion in Theorems \ref{thm: non-empty} and \ref{thm: non-empty2} with a stronger dimension conclusion.

\begin{thm} 
    \label{thm:main}
    Assume that $\Psi$ and $\omega$ are such that $q\to q^{2}\Psi(q)$ and $q\to q^{2}\Psi(q)\omega(q)$ are non-increasing, $\sum_{q=1}^{\infty}q\Psi(q)<\infty$ and $\lim_{q\to \infty}\omega(q)=0.$ Let us also assume that $$\omega(q)\geq \frac{45}{\Psi(q)}\Psi\left(\frac{1}{2q\Psi(q)}\right)$$ for all $q$ sufficiently large.  Then $$\dim_{H}(E(\Psi,\omega))=\inf\left\{s\geq 0: \sum_{q=1}^{\infty}q(\Psi(q)\omega(q))^{s}<\infty\right\}.$$
\end{thm}
Note that the conclusion of Theorem \ref{thm:main} can also be written in terms of the lower order at infinity of the function $\Psi\cdot\omega$. That is, if $\Psi$ and $\omega$ satisfy the assumptions of Theorem~\ref{thm:main} and we let $\lambda(\Psi,\omega)$ denote the lower order at infinity of $\Psi\cdot\omega$, then $\dim_{H} (E(\Psi,\omega))=\frac{2}{\lambda(\Psi,\omega)}$. See Proposition~\ref{prop:lower order and dimension} and ~\S~\ref{sec: dim statements proof} for the definitions and details. \par
In the case of approximation functions satisfying the square divisibility condition we can again allow $\omega$ to have a faster rate of decay.
\begin{thm} 
    \label{thm:main2}
    Assume that $\Psi$ and $\omega$ are such that $q\to q^{2}\Psi(q)$ and $q\to q^{2}\Psi(q)\omega(q)$ are non-increasing, $\Psi$ satisfies the square divisibility condition, $\sum_{q=1}^{\infty}q\Psi(q)<\infty$, and $\lim_{q\to \infty}\omega(q)=0.$ Let us also assume that 
    \begin{equation*}
        \omega(q)\geq \frac{45}{16\Psi(q)}\Psi\left(\frac{1}{8\Psi(q)}\right) 
    \end{equation*}
    for all $q$ sufficiently large.  Then $$\dim_{H}(E(\Psi,\omega))=\inf\left\{s\geq 0: \sum_{q=1}^{\infty}q(\Psi(q)\omega(q))^{s}<\infty\right\}.$$
\end{thm}

Focusing on the special case of functions of the form $q\to q^{-\tau}$ we see that Theorems \ref{thm: non-empty2}, \ref{thm:main} and \ref{thm:main2} imply the following corollary. The second part of this statement simplifies the critical threshold identified in Theorem \ref{thm: non-empty2}.

\begin{cor}
    \label{cor:polynomial corollary}
    Let $\Psi:(0,\infty)\to(0,\infty)$ be given by $\Psi(q)=q^{-\tau_{1}}$ for some $\tau_{1}>2$.  Then the following statements are true:
    \begin{enumerate}
        \item If $\omega:(0,\infty)\to(0,\infty)$ is given by $\omega(q)=q^{-\tau_{2}}$ for some $\tau_{2}<\tau_{1}(\tau_{1}-2)$, then $\dim_{H}(E(\Psi,\omega))=\frac{2}{\tau_{1}+\tau_{2}}$.
        \item If $\tau_{1}\in \N_{\geq3}$ and $\omega:(0,\infty)\to(0,\infty)$ is given by $\omega(q)=q^{-\tau_{2}}$ for some $\tau_{2}<\tau_{1}(\tau_{1}-1)$, then $\dim_{H}(E(\Psi,\omega))=\frac{2}{\tau_{1}+\tau_{2}}$.
        \item If $\tau_{1}\in \N_{\geq3}$ and $\omega:(0,\infty)\to(0,\infty)$ is given by $\omega(q)=Cq^{-\tau_{1}(\tau_{1}-1)}$ for some $C>\frac{45}{16}\cdot 8^{\tau_{1}},$ then $\dim_{H}(E(\Psi,\omega))=\frac{2}{\tau_{1}^{2}}.$
        \item If $\tau_{1}\in \N_{\geq3}$ and $\omega:(0,\infty)\to(0,\infty)$ is given by $\omega(q)=cq^{-\tau_{1}(\tau_{1}-1)}$ for some $c\in (0,1),$ then $E(\Psi,\omega)=\emptyset.$
        \end{enumerate}
  
            
\end{cor}
 
Taking $\omega$ to be a function that decays to zero sufficiently slowly, e.g. logarithmically, it is possible to obtain the following quantitative improvement on \cite[Theorem 1]{Bug1}. This is an immediate consequence of Theorem \ref{thm:main} and Proposition \ref{prop:lower order and dimension} proved below.

\begin{cor}
\label{cor:dimension coincidence}
    Assume that $\Psi$ and $\omega$ are such that $q\to q^{2}\Psi(q)$ and $q\to q^{2}\Psi(q)\omega(q)$ are non-increasing, $\sum_{q=1}^{\infty}q\Psi(q)<\infty$ and $\lim_{q\to \infty}\omega(q)=0.$ Suppose in addition that $\lim_{q\to\infty}\frac{-\log \omega(q)}{\log q}=0.$ Then 
    \begin{equation*}
        \dim_{H}(E(\Psi,\omega))=\dim_{H}(E(\Psi))=\dim_{H}(\textrm{\rm Bad}(\Psi))=\dim_{H}(W(\Psi))\, .
    \end{equation*}
\end{cor}

From a technical perspective the main innovation of this paper is a detailed analysis of the set of $x$ satisfying 
\begin{equation}
\label{eq:Extra Diophantine inequality}
\Psi(q)(1-\omega(q))\leq\left|x-\frac{p}{q}\right|\leq \Psi(q) 
\end{equation} for some $\Psi,\omega$ and $p/q\in \Q.$ This analysis is given in Section \ref{sec:Technical results}. We show that under suitable assumptions the set of $x$ satisfying \eqref{eq:Extra Diophantine inequality} can be approximated from within by a set that has a simple symbolic description in terms of the continued fraction expansion. Crucially, this symbolic description can be used to rule out unwanted rational solutions to $|x-p'/q'|< \Psi(q')(1-\omega(q')).$ Under suitable assumptions we can guarantee that this approximating set has comparable Lebesgue measure to the set of $x$ satisfying \eqref{eq:Extra Diophantine inequality}. This property will be important when it comes to proving Hausdorff dimension statements.

The rest of the paper is structured as follows. In Section \ref{sec:preliminaries} we recall some standard results from continued fraction theory that we will use in our proofs. We also formulate a result that gives an equivalent definition of the set $E(\Psi,\omega)$ in terms of continued fractions, see Proposition~\ref{prop: equivalence}. In Section \ref{sec:Technical results} we will prove several key technical results, Lemmas~\ref{lem: building U(p/q)}, ~\ref{lem: non-empty S(p/q) divisibility cond}, and ~\ref{lem: building U(p/q) with divisibility}, showing that under suitable assumptions on $\Psi$ and $\omega$, we can build suitable subsets of $E(\Psi,\omega)$ that have a simple description in terms of the partial quotients of the continued fraction expansion. In Section \ref{sec: proof non-empty} we prove Theorems \ref{thm: non-empty}, \ref{thm: non-empty2} and \ref{thm: non-empty3}, and in Section~\ref{sec: dim statements proof} we prove Theorems~\ref{thm:main} and \ref{thm:main2}. Theorem~\ref{thm:empty} is proven in a different manner to our other results. Its proof can be found in Section~\ref{sec: proof of empty}.\\

\noindent \textbf{Notation.} Throughout this paper we will adopt the following notation conventions. Given two positive real-valued functions $f,g:X\to \R$ defined on some set $X$, we write $f\ll g$ if there exists $C>0$ such that $f(x)\leq Cg(x)$ for all $x\in X$. Moreover, given a digit $a$ belonging to some alphabet set $\mathcal{A}$ and $n\in \N$, we denote by $a^{n}=(\overbrace{a,\ldots,a}^{\times n})$ the word given by the $n$-fold concatenation of $a$ with itself. We similarly define $a^{\infty}$ to be the infinite word whose digits consist only of the digit $a.$

\section{Preliminaries}
\label{sec:preliminaries}

\subsection{Continued fractions and fundamental intervals}
In this section, we recall some classical results on continued fractions, many of which can be found in Khintchine's book \cite{Khi2}. The theory of continued fractions is deeply related to many sets appearing in Diophantine approximation. For instance, one can relate the set of badly approximable points, $\Psi$-well approximable points, and $\Psi$-Dirichlet improvable points to properties of the continued fraction expansion of a point, see for example \cite{Khi2,KleinWad18}. We prove a key proposition below that gives an equivalent condition for belonging to $E(\Psi,\omega)$ in terms of the continued fraction expansion. \par 
We can write any $x\in \R\setminus \Q$ in continued fraction expansion form:
\begin{equation*}
    x=a_{0}(x)+\frac{1}{a_{1}(x)+\frac{1}{a_{2}(x)+\frac{1}{\ddots}}}=[a_{0}(x);a_{1}(x),a_{2}(x),\ldots]\, ,
\end{equation*}
for unique $a_{0}(x)\in \Z$ and $a_{i}(x)\in \N$ for all $i\geq 1.$ We call the entries in $(a_i)$ the partial quotients of $x$. When we want to emphasise that some partial quotients are the partial quotients of a real number $x$, we may write $a_{i}(x)=a_{i}$. For $x\in \Q$ a similar continued fraction can be defined, however it is not necessarily unique. We will make use of the following lemma. For a proof see Chapter 1 of Schmidt's book \cite{Sch}.

\begin{lem}
\label{lem:rationalcfexpansion}
    Every rational number $\frac{p}{q}\in \Q$ has a unique continued fraction expansion $\frac{p}{q}=[a_{0};a_{1},\ldots, a_{n}]$ where $n\in \N$ is even. Furthermore, if $c>b>0$ we have
    \begin{equation*}
        [a_{0};a_{1},\ldots, a_{n}]<[a_{0};a_{1},\ldots, a_{n},c]<[a_{0};a_{1},\ldots, a_{n},b]\, .
    \end{equation*}
    Similarly, every $\frac{p}{q}\in \Q$ also has a unique continued fraction expansion $\frac{p}{q}=[a_{0};a_{1},\ldots, a_{m}]$ where $m\in \N$ is odd, and for $c>b>0$ we have
    \begin{equation*}
        [a_{0};a_{1},\ldots, a_{m}]>[a_{0};a_{1},\ldots, a_{m},b]>[a_{0};a_{1},\ldots, a_{m},c]\, .
    \end{equation*}
\end{lem}
Unless the expansion is explicitly given, when we speak of the continued fraction expansion of a rational number we will often mean the even continued fraction expansion given by the first part of Lemma \ref{lem:rationalcfexpansion}. Given $x\in \R$ and $n\in \N$ we call 
\begin{equation*}
    \frac{p_{n}}{q_{n}}:=[a_{0}(x);a_{1}(x),\ldots, a_{n}(x)]
\end{equation*}
the $n$-th convergent of $x$. Given the partial quotients of a real number $x\in (0,1)$ we can inductively define the numerator and denominator of the convergents by setting $p_{0}=q_{-1}=0$, $q_{0}=p_{-1}=1$ and then
\begin{equation}
\label{eq:recursive formulas}
    q_{n+1}=a_{n+1}q_{n}+q_{n-1}\, , \quad p_{n+1}=a_{n+1}p_{n}+p_{n-1}\, .
\end{equation}
When manipulating a continued fraction expansion it is often useful to write $q_{n}(\ba)$, for an $n$-tuple of positive integers $\ba\in \N^{n},$ to make explicit what sequence of partial quotients we are using to define $q_{n}$. In addition to this, if $q_{n}$ is the denominator of the $n$th convergent of $x$ we may write $q_{n}(x)$ to emphasise the dependence on $x$. While we assume all partial quotients are strictly positive, we will on occasion undertake some arithmetic operations on the partial quotients with the result being that some terms may equal zero. In these instances, we use the formula $[\ldots, b,0,c,\ldots]=[\ldots,b+c,\ldots]$, which can easily be verified, to revert back to the case of strictly positive entries. 

 The following lemma shows how the quality of approximation provided by a convergent is determined by the partial quotients. For a proof of this lemma see \cite{Khi}.

\begin{lem}
    \label{lem:approximation quality}
    Let $x\in \R.$ Then for any $n\geq 1,$ if the $n$th convergent $\frac{p_{n}}{q_{n}}$ is distinct from $x$, then $$\frac{1}{3a_{n+1}q_{n}^{2}}<\frac{1}{(a_{n+1}+2)q_{n}^{2}}<\left|x-\frac{p_{n}}{q_{n}}\right|< \frac{1}{a_{n+1}q_{n}^{2}}.$$
\end{lem}
For the $\Psi$ we will consider, the following result due to Legendre implies that all but finitely many of the rational $\Psi$-approximations of a point will be given by its convergents. For a proof see \cite{HardyWright}.
\begin{lem}
\label{lem:Legendre}
    Let $x\in \R$ and suppose $(p,q)\in \mathbb{Z}\times \N$ is such that \begin{equation*}
    \left|x-\frac{p}{q}\right|<\frac{1}{2q^{2}}.
\end{equation*}
Then $\frac{p}{q}$ is a convergent of $x$.
\end{lem}
When the constant $1/2$ appearing on the right hand side of the inequality in Lemma \ref{lem:Legendre} is replaced with $1$ we have the following result due to Grace \cite{Grace}. This result tells us that if a rational approximation is sufficiently good, then it is either a convergent, or its continued fraction expansion closely resembles that of the point it is approximating. 
\begin{lem}[\cite{Grace}]
\label{lem:Grace}
Let $x=[0;a_1,a_2,\ldots]$ be a real number and $(p,q)\in\mathbb{Z}\times \N$ be such that
\begin{equation*}
    \left|x-\frac{p}{q}\right|<\frac{1}{q^{2}}\, .
\end{equation*}
Then, there exists an integer $n$ such that 
$$\frac{p}{q}\in \left\{[0;a_1,\ldots,a_n], [0;a_1,\ldots,a_n,1], [0;a_1,\ldots,a_n,a_{n+1}-1]\right\}\, .$$

\end{lem}

Armed with the result of Legendre and several key properties of continued fractions, we are able to prove the following key proposition. This result provides a useful equivalent definition for a set $E(\Psi,\omega)$.

\begin{prop} \label{prop: equivalence}
Let $\Psi:(0,\infty)\to (0,\infty)$ be such that $\Psi(q)<\frac{1}{2q^{2}}$ for all sufficiently large $q$. Let $x=[a_{0};a_{1},a_{2},\ldots]$ and $\frac{p_{n}}{q_{n}}$ denote the $n$-th convergent of $x$ for each $n\in \N$. Then $x\in E(\Psi,\omega)$ if and only if
\begin{enumerate}[(i)]
    \item \label{E-bad} for all $n\in \N$ sufficiently large
    \begin{equation*}
        [a_{n+1};a_{n+2},\ldots] \leq \frac{1}{q_{n}^{2}\Psi(q_{n})(1-\omega(q_{n}))}-\frac{q_{n-1}}{q_{n}}\, ,
    \end{equation*}
    and
    \item \label{S(p/q) set} for infinitely many $n\in \N$
    \begin{equation*}
        [a_{n+1};a_{n+2},\ldots] \in \left[ \frac{1}{q_{n}^{2}\Psi(q_{n})}-\frac{q_{n-1}}{q_{n}}, \frac{1}{q_{n}^{2}\Psi(q_{n})(1-\omega(q_{n}))}-\frac{q_{n-1}}{q_{n}} \right]\, .
    \end{equation*}
\end{enumerate}
\end{prop}
\begin{rem} \label{remark to prop}
Condition \ref{E-bad} corresponds to the $\liminf$ property in the definition of $E(\Psi,\omega)$, while \ref{S(p/q) set} corresponds to the $\limsup$ property. Under the assumption $\lim_{q\to\infty} q^{2}\Psi(q)=\lim_{q\to\infty}\omega(q)=0$, Lemma~\ref{lem:approximation quality} implies that if $x$ satisfies \ref{S(p/q) set} and
\begin{enumerate}[(i')]
    \item \label{i' statement} for every $n \in \N$ sufficiently large either \ref{S(p/q) set} holds or $a_{n} \leq \left\lfloor \frac{1}{3q_{n-1}^{2}\Psi(q_{n-1})}\right\rfloor$\, ,
\end{enumerate}
then $x\in E(\Psi,\omega)$. We will often find it easier to work with condition \ref{i' statement} instead of \ref{E-bad}. Similarly, if $x$ does not satisfy \ref{S(p/q) set} or
\begin{equation*}
    a_{n}(x)>\frac{2}{q_{n-1}^{2}\Psi(q_{n-1})} \quad \text{ for i.m.\, } \, n\in \N\, ,
\end{equation*}
then $x\not\in E(\Psi,\omega)$.
\end{rem}
\begin{proof}[Proof of Proposition~\ref{prop: equivalence}]
    Clearly our statements are unaffected by a finite number of exceptions. Thus we can suppose $Q\in \N$ is large enough so that $\Psi(q)<\frac{1}{2q^{2}}$ for all $q>Q$. Furthermore, we can assume $n$ is sufficiently large so that $q_{n}>Q$. Hence, by Lemma~\ref{lem:Legendre}, when determining whether $x$ belongs to $E(\Psi,\omega)$ we only need to consider the convergents of $x$. Write $x=[a_{0}(x);a_{1}(x),\ldots]$ and note that, for every $n\in \N$, 
\begin{equation*}
    x=\frac{p_{n}[a_{n+1}(x);a_{n+2}(x),\ldots] +p_{n-1}}{q_{n}[a_{n+1}(x);a_{n+2}(x),\ldots]+q_{n-1}},
\end{equation*}
see for example \cite[Chapter II. \S 5]{Khi2}. Using the well known formula $p_{n-1}q_{n}-q_{n-1}p_{n}=(-1)^{n}$ (see for example \cite[Theorem 2]{Khi2}), it follows from the above that we have
\begin{equation*}
    \left|x-\frac{p_{n}}{q_{n}}\right|=\frac{|(-1)^{n}|}{q_{n}\left(q_{n}[a_{n+1}(x);a_{n+2}(x),\ldots] +q_{n-1}\right)}=\frac{1}{q_{n}^{2}\left([a_{n+1}(x);a_{n+2}(x),\ldots] +\tfrac{q_{n-1}}{q_{n}}\right)}\, .
\end{equation*}
It is then clear that 
\begin{equation} \label{eq: equivalence 1}
    \left|x-\frac{p_{n}}{q_{n}}\right|\geq (1-\omega(q_{n}))\Psi(q_{n}) \quad \iff \quad [a_{n+1};a_{n+2},\ldots] \leq \tfrac{1}{q_{n}^{2}\Psi(q_{n})(1-\omega(q_{n}))}-\tfrac{q_{n-1}}{q_{n}}
\end{equation}
and 
\begin{align} \label{eq: equivalence 2}
 &(1-\omega(q_{n}))\Psi(q_{n})\leq \left|x-\frac{p_{n}}{q_{n}}\right|\leq \Psi(q_{n})  \nonumber \\  \iff \quad &[a_{n+1};a_{n+2},\ldots] \in \left[ \tfrac{1}{q_{n}^{2}\Psi(q_{n})}-\tfrac{q_{n-1}}{q_{n}}, \tfrac{1}{q_{n}^{2}\Psi(q_{n})(1-\omega(q_{n}))}-\tfrac{q_{n-1}}{q_{n}} \right]\, .
\end{align}
Our result now follows from \eqref{eq: equivalence 1}, \eqref{eq: equivalence 2} and our earlier observation that belonging to $E(\Psi,\omega)$ is determined by the convergents of $x$.
\end{proof}

\subsubsection{Fundamental intervals}
 The natural next step is to partition the unit interval into suitable subsets that satisfy the conditions appearing in Proposition~\ref{prop: equivalence}. To do this we use the notion of fundamental intervals. Given $\bb=(b_{0},b_{1},\ldots, b_{n})\in \N^{n+1}$ we associate the fundamental interval
\begin{equation*}
    I_{n}(b_{0};b_{1},\ldots, b_{n}):=\{x\in \R: a_{0}(x)=b_{0},\, a_{1}(x)=b_{1}, \, \ldots , \, a_{n}(x)=b_{n} \}\, .
\end{equation*}
We emphasise that unlike many articles on Diophantine approximation our fundamental intervals are not confined to $[0,1]$ because we allow for a choice of $a_{0}$. This added flexbility will simplify our proofs in Section \ref{sec:Technical results}. When it is notationally more convenient, we may denote a fundamental interval in several different ways; if $\bb\in \N^{n+1}$ let $I(\bb)=I_{n}(b_{0};b_{1},\ldots,b_{n})$ and, for $j\leq n$, let $I_{j}(\bb)=I_{j}(b_{0};b_{1},\ldots, b_{j})$.\par 
Note that fundamental intervals are indeed intervals. They will act as building blocks in the proofs of each of our theorems.\par 
A useful fact when working with fundamental intervals follows from Lemma~\ref{lem:rationalcfexpansion}: A fundamental interval $I_{n}(\bb)$ is partitioned into countably many $(n+1)$-th level fundamental intervals $I_{n+1}(\bb k)$ where $k\in \N$. If $n$ is even then the $k$ appearing in $I_{n+1}(\bb k)$ is increasing as we move from right to left within $I_{n}(\bb)$. Similarly, if $n$ is odd then $k$ increases as we move from left to right within $I_{n}(\bb)$. \par 
The following lemma provides a useful description of fundamental intervals. Here and throughout this paper, we let $\mathcal{L}$ denote the Lebesgue measure on $\R$.

\begin{lem}
\label{lem:Fundamental interval}
Given $\bb=(b_{0},b_{1},\ldots, b_{n})\in \N^{n}$ let $\frac{p_n}{q_{n}}=[0;b_1,\ldots,b_n]$ and $\frac{p_{n-1}}{q_{n-1}}=[0;b_1,\ldots,b_{n-1}]$. Then
$$I_{n}(b_{0};b_{1},\ldots, b_{n})=\left[b_{0}+\frac{p_{n}}{q_{n}},b_{0}+\frac{p_{n}+p_{n-1}}{q_{n}+q_{n-1}}\right]$$ if $n$ is even and 
$$I_{n}(b_{0};b_{1},\ldots, b_{n})=\left[b_{0}+\frac{p_{n}+p_{n-1}}{q_{n}+q_{n-1}},b_{0}+\frac{p_{n}}{q_{n}}\right]$$ if $n$ is odd. Moreover, we have $$\frac{1}{2q_{n}^{2}}\leq \mathcal{L}(I_{n}(b_{0};b_{1},\ldots, b_{n}))\leq \frac{1}{q_{n}^{2}}$$
 \end{lem}
 \begin{proof}
      Proofs of the $I_{n}(b_{0};b_{1},\ldots, b_{n})$ identities can be found in \cite[Chapter III, \S~12.]{Khi}. Our bounds for $\mathcal{L}(I_{n}(b_{0};b_{1},\ldots, b_{n}))$ follows from these identities and the well known fact that $|p_{n-1}q_{n}-q_{n-1}p_{n}|=1$ for all $n\in \N$.  
 \end{proof}


The following lemma gives bounds on the Lebesgue measure of naturally occurring subsets of fundamental intervals. For a proof see \cite[Lemma 3]{Bug1}.

\begin{lem}
\label{lem:Subset measure bound}
    Let $\bb=(b_{1},\ldots, b_{n})\in \N^{n}$ and $A\geq 3.$ Then for any $a\in \Z$ $$\mathcal{L}\left(\left\{x\in I_{n}(a;\bb): a_{n+1}(x)\leq A\right\}\right)\geq \left(1-\frac{3}{A}\right) \mathcal{L}(I_{n}(a;\bb))$$
\end{lem}
Lastly, we have the following lemma which shows that the length of a fundamental interval and the denominator of a convergent are almost multiplicative functions. 

\begin{lem}
    \label{lem:Quasi-multiplicative}
    Let $\ba=(a_{1},\ldots,a_{m}),\bb=(b_{1},\ldots,b_{n})$ be finite words. Then 
    $$\frac{q_{m}(\ba)q_{n}(\bb)}{2}\leq q_{m+n}(\ba\bb)\leq 4q_{m}(\ba)q_{n}(\bb).$$ Moreover, for any $a\in \mathbb{Z}$ we have $$\frac{\mathcal{L}(I_{m}(0,\ba))\mathcal{L}(I_{n}(0,\bb))}{16}\leq \mathcal{L}(I_{m+n}(a,\ba\bb))\leq 64\mathcal{L}(I_{m}(0,\ba))\mathcal{L}(I_{n}(0,\bb)).$$
    In the special case when $\bb=(b)$ is a word of length $1$ we have the following bound for any $a\in \Z$
    $$\mathcal{L}(I_{m+1}(a;\ba b))\leq 4\mathcal{L}\left(I_{m+1}(a;\ba (b+1))\right).$$
\end{lem}

\begin{proof}
The first two statements follow from \cite[Lemmas 2.5 and 2.6]{JorSah}. For the final statement when $\bb=(b)$ is a word of length $1,$ we observe that for any $a\in \Z$ we have $$\mathcal{L}(I_{m+1}(a;\ba b))=\frac{1}{(q_{m}b+q_{m-1})(q_{m}(b+1)+q_{m-1})}$$
where $\frac{p_{m}}{q_{m}}=[0;a_{1},\ldots, a_{m}].$ This follows from Lemma~\ref{lem:Fundamental interval} and \eqref{eq:recursive formulas}. Thus 
$$\mathcal{L}(I_{m+1}(a;\ba (b+1)))=\frac{(q_{m}b+q_{m-1})}{(q_{m}(b+2)+q_{m-1})}\mathcal{L}(I_{m+1}(a;\ba b)).$$ Now using the fact $0\leq \frac{q_{m-1}}{q_{m}}\leq 1$ for all $m$ and $b\geq 1,$ we have 
$$\mathcal{L}(I_{m+1}(a;\ba (b+1)))\geq \frac{b}{b+3}\mathcal{L}(I_{m+1}(a;\ba b))\geq \frac{\mathcal{L}(I_{m+1}(a;\ba b))}{4}.$$ This completes our proof.

\end{proof}

There are some instances when we will need to know about the relative size of fundamental intervals that are close in a geometric sense, but are not nested within the same previous level fundamental interval.

\begin{lem} \label{lem: c-1,1 shorter than c}
Let $\ba=(a_{1},\ldots, a_{n})$ be a finite word and let $c\in \N_{\geq 2}$. Then for any $b\in \N$ we have that
\begin{equation*}
    \cL\left(I_{n+2}(0;\ba, c, b)\right)\geq \cL\left(I_{n+3}(0;\ba,c-1,1,b)\right)\, .
\end{equation*}
\end{lem}

\begin{proof}
    By Lemma~\ref{lem:Fundamental interval}, \eqref{eq:recursive formulas}, and the well-known property that $|p_{n}q_{n+1}-p_{n+1}q_{n}|=1$ for every $n\in \N$, we have that
    \begin{align} \label{eq: big FI}
        \cL\left(I_{n+2}(0;\ba, c,b)\right)\nonumber
        =&\frac{1}{q_{n+2}(\ba cb)\left(q_{n+2}(\ba cb)+q_{n+1}(\ba c)\right)}\nonumber \\
        =&\frac{1}{\left((bc+1)q_{n}(\ba)+bq_{n-1}(\ba)\right)\left(((b+1)c+1)q_{n}(\ba)+(b+1)q_{n-1}(\ba) \right)}\, ,
    \end{align}
    and
    \begin{align} \label{eq: small FI}
        &\cL(I_{n+3}(0;\ba, c-1,1,b))\nonumber\\
        =&\frac{1}{q_{n+3}(\ba (c-1)1b)\left( q_{n+3}(\ba(c-1)1b)+q_{n+2}(\ba(c-1)1)\right)}\nonumber \\
        =&\frac{1}{\left(((b+1)c-1)q_{n}(\ba)+(b+1)q_{n-1}(\ba)\right)\left( ((b+2)c-1)q_{n}(\ba)+(b+2)q_{n-1}(\ba)\right)}\, .
    \end{align}
    Considering the first bracket in the denominator of \eqref{eq: small FI} we see that
    \begin{align*}
        ((b+1)c-1)q_{n}(\ba)+(b+1)q_{n-1}(\ba)&=(bc+1)q_{n}(\ba)+bq_{n-1}(\ba)+(c-2)q_{n}(\ba) +q_{n-1}(\ba)\\
        &\geq(bc+1)q_{n}(\ba)+bq_{n-1}(\ba)
    \end{align*}
    since $c\geq 2$ and $q_{n-1}(\ba)\geq 0$ for any $n\in\N$. This is the first bracket in the denominator of \eqref{eq: big FI}.  The same inequality holds when comparing the second brackets appearing in each of the denominators. Thus \eqref{eq: small FI} is less than or equal to \eqref{eq: big FI} as claimed.
\end{proof}

\subsection{Properties of convergent series and Cantor sets} \label{sec: cantor set stuff}
The following lemma is well known. For a proof see \cite[Lemma 5]{Bug1}.

\begin{lem}
\label{lem:Geometric convergence}
    Let $\Psi:(0,\infty)\to (0,\infty)$ be non-increasing. If $\sum_{q=1}^{\infty}\Psi(q)$ converges then for any $c>0$ and $M>1$ the series $\sum_{x=1}^{\infty}M^{x}\Psi(cM^{x})$ converges. 
\end{lem} 

We now recall a standard lower bound for the Hausdorff dimension of Cantor sets. Let $[0,1]=E_{0}\subset E_{1}\subset E_{2}\subset \cdots$ be a nested sequence of sets such that each $E_{k}$ consists of finitely many disjoint intervals. Moreover, assume that each interval in $E_{k}$ contains at least two intervals in $E_{k+1}$, and that the maximal length of an interval in $E_{k}$ converges to zero as $k$ tends to infinity. Then the set $$E=\bigcap_{k=0}^{\infty}E_{k}$$ is a Cantor set. The following lower bound for the Hausdorff dimension of $E$ follows from Falconer \cite[Example 4.6]{Fal} and the remarks following \cite[Proposition 1]{Bug1}.

\begin{prop}
    \label{prop:dimension lower bound}
    Let $\{E_{k}\}$ be as above. Suppose that each interval in $E_{k-1}$ contains at least $m_{k}$ intervals in $E_{k}$ which are separated by gaps of size at least $\epsilon_{k},$ where $0<\epsilon_{k+1}<\epsilon_{k}.$ Moreover, suppose that the intervals of $E_{k}$ are of length at least $\delta_{k}$ and that $m_{k+1}\epsilon_{k+1}>c\delta_{k}$ for some $c>0$. Then 
    $$\dim_{H}(E)\geq \liminf_{k\to\infty}\frac{\log m_{k}}{-\log \delta_{k}}.$$
\end{prop}

\section{Proof of empty statements} \label{sec: proof of empty}

\begin{proof}[Proof of Theorem~\ref{thm:empty}] We will prove statement $(1)$ in detail and then explain the minor changes needed to prove statement $(2).$

Let $\Psi$ and $\omega$ satisfy the assumptions of statement $(1)$. For the purpose of obtaining a contradiction, assume $E(\Psi,\omega)$ is non-empty. Let $x\in E(\Psi,\omega),$ then for infinitely many $(p,q)\in\mathbb{Z}\times \N$ we have 
\begin{equation}
    \label{eq:first approximationA}
    \Psi(q)\left(1-\omega(q)\right)\leq \left|x-\frac{p}{q}\right|\leq \Psi(q).
\end{equation}
Let us temporarily assume that $(p,q)\in \mathbb{Z}\times \N$ satisfies \eqref{eq:first approximationA} and we have $x-\frac{p}{q}>0$. For such a $(p,q)$ we observe the following equivalences
\begin{align*}
\Psi(q)\left(1-\omega(q)\right)\leq \left|x-\frac{p}{q}\right|\leq \Psi(q)\iff&-\Psi(q)\omega(q)\leq x-\frac{p}{q}-\Psi(q)\leq 0\\
\iff &-\Psi(q)\omega(q)\leq x-\frac{p}{q}-\frac{d(q)}{e(q)}\leq 0\\
\iff &-\Psi(q)\omega(q)\leq x-\frac{pe(q)+d(q)q}{qe(q)}\leq 0.
\end{align*}
Letting $p'=pe(q)+d(q)q,$ the above implies that 
\begin{equation}
\label{eq:secondary approximationA}
\left|x-\frac{p'}{qe(q)}\right|\leq \Psi(q)\omega(q).
\end{equation}
It can similarly be shown that \eqref{eq:secondary approximationA} holds for some $p'\in \Z$ when \eqref{eq:first approximationA} is satisfied by some $(p,q)\in \Z\times \N$ satisfying $x-\frac{p}{q}<0.$ Thus if $x\in E(\Psi,\omega)$ then \eqref{eq:secondary approximationA} is satisfied by i.m. $(p',q)\in \mathbb{Z}\times \N$.
The key observation is that \eqref{eq:secondary approximationA} is a rational approximation to $x$. Therefore, since $x\in E(\Psi,\omega)$, it follows from \eqref{eq:secondary approximationA} that we must have 
\begin{equation}
\label{eq:about to be contradicted}
\Psi(qe(q))(1-\omega(qe(q))\leq \left|x-\frac{p'}{qe(q)}\right|\leq\Psi(q)\omega(q)
\end{equation}for infinitely many values of $q$. Comparing the left hand term and right hand term of \eqref{eq:about to be contradicted}, we see that this contradicts our assumption.

To prove statement $(2)$ we will use the fact that $\frac{p}{q}+\frac{d(q)}{e(q)}$ can be written as a fraction whose denominator is $e(q)$ for any $(p,q)\in \Z\times \N.$ This is a consequence of our assumption that $e(q)=0 \mod q$ for all $q\in \N$. As a result of this, if we proceed as in the the proof of statement $(1)$ and let $x\in E(\Psi,\omega),$ then for infinitely many $(p',q)\in \Z\times \N$ we would have 
$$\left|x-\frac{p'}{e(q)}\right|\leq \Psi(q)\omega(q).$$ This is our analogue of \eqref{eq:secondary approximationA}. The rest of the argument is now identical. 
\end{proof}
\section{Technical results }
\label{sec:Technical results}
In this section we establish the main technical results that will allow us to prove our theorems. Roughly speaking, we will show that for a given $(p,q)\in \Z\times \N,$ under the assumptions of our theorems, it is possible to construct a well-behaved subset of 
\begin{equation}
\label{eq:set we want to approximate}
\left\{x:(1-\omega(q))\Psi(q)\leq x-\frac{p}{q}<\Psi(q)\right\}.
\end{equation}
That is, using fundamental intervals, we build subsets whose partial quotients up to some $n$ satisfy the conditions appearing in Proposition~\ref{prop: equivalence}.
The main results of this section are Lemmas~\ref{lem: building U(p/q)},~\ref{lem: non-empty S(p/q) divisibility cond} and~\ref{lem: building U(p/q) with divisibility}. \par Throughout this section we let $\Psi:(0,\infty)\to (0,\infty)$ and $\omega:(0,\infty)\to (0,1)$ be such that 
\begin{equation} \label{psi omega cond}
    q\to q^{2}\Psi(q) \, \text{  is non-increasing and  }\,  \lim_{q\to\infty}\limits q^{2}\Psi(q)=\lim_{q\to\infty}\limits \omega(q)=0.
\end{equation}

\subsection{Technical results for Theorems \ref{thm: non-empty} and \ref{thm:main}} \label{sec:Technical results part 1}
Before we state the main result of this subsection we introduce some terminology. We say that a finite set of words $\{\bb_{l}\}_{l=1}^{L}$ is consecutive if they are all of the same length and either $\{\bb_{l}\}_{l=1}^{L}=\{b\}_{b=M_{1}}^{M_{2}}$ for some $M_{1},M_{2}\in \N$ or if there exists a word $\ba$ such that $\{\bb_{l}\}_{l=1}^{L}=\{\ba b\}_{b=M_{1}}^{M_{2}}$ for some $M_{1},M_{2}\in \N$.
The main result of this subsection is the following lemma.
\begin{lem} \label{lem: building U(p/q)}
Let $\Psi$ and $\omega$ satisfy \eqref{psi omega cond}. Assume that $$\omega(q)\geq \frac{45}{\Psi(q)}\Psi\left(\frac{1}{2q\Psi(q)}\right)$$ for all $q$ sufficiently large. Let $(p,q)\in \Z\times \N$ and $\frac{p}{q}=[0;a_{1},\ldots,a_{n}]$ for $n$ even.  Then if $q$ is sufficiently large, there exists a finite set of consecutive words $\{\bb_{l}\}_{l=1}^{L}$ such that 
\begin{equation} \label{eq: inclusion statement}
    \bigcup_{l=1}^{L}I(0;a_{1},\ldots,a_{n},\bb_{l})\subset \left\{x:\Psi(q)(1-\omega(q))\leq x-\frac{p}{q}< \Psi(q)\right\}
\end{equation}
and 
\begin{equation} \label{eq: measure bound}
    \mathcal{L}\left(\bigcup_{l=1}^{L}I(0;a_{1},\ldots,a_{n},\bb_{l}))\right)\geq \frac{\omega(q)\Psi(q)}{6400}\, .
\end{equation} 
Moreover, if each $\bb_{l}$ has length $m\geq 2$, then
      \begin{equation} \label{b digit restrition2}
    b_{i}\in \left\{ 1,2, \ldots,\left\lfloor \frac{1}{3q_{n+i}^{2}\Psi(q_{n+i})}\right\rfloor \right\} \quad 1\leq i \leq m-1\, ,
    \end{equation} where $\frac{p_{n+i}}{q_{n+i}}=[0;a_{1},\ldots, a_{n},b_{0},b_{1},\ldots,b_{i-1}]$ for $1\leq i\leq m-1.$ 
\end{lem}

Instead of working directly with the set in \eqref{eq:set we want to approximate}, we will work with the set appearing in Proposition~\ref{prop: equivalence} \ref{S(p/q) set}. The following lemma can be viewed as a finite version of Proposition~\ref{prop: equivalence} with a fixed prefix word.

\begin{lem} \label{lem: belongs to u(p/q)}
Let $\Psi$ and $\omega$ satisfy \eqref{psi omega cond}. Let $(p,q)\in \Z\times \N$ be such that $\frac{p}{q}\in (0,1)$ and $\frac{p}{q}=[0;a_{1},\ldots, a_{n}]$ for $n$ even. Then if $q$ is sufficiently large, we have
    \begin{equation}\label{eq1}
        (1-\omega(q))\Psi(q)\leq x-\frac{p}{q}\leq \Psi(q)
    \end{equation}
    if and only if
    \begin{equation} \label{eq2}
        x=[0;a_{1},\ldots, a_{n}, a_{n+1}(x), a_{n+2}(x),\ldots]
    \end{equation}
    with
    \begin{equation} \label{condition}
        [a_{n+1}(x);a_{n+2}(x),\ldots] \in \left[ \frac{1}{q^{2}\Psi(q)}-\frac{q_{n-1}}{q}, \frac{1}{q^{2}\Psi(q)(1-\omega(q))}-\frac{q_{n-1}}{q} \right].
    \end{equation}
\end{lem}

\begin{proof} The proof is similar to the proof of Proposition~\ref{prop: equivalence}. Note that since $\lim_{q\to\infty}q^{2}\Psi(q)= 0$ we have $\Psi(q)\leq \tfrac{1}{2q^{2}}$ for all $q$ sufficiently large, and so by Lemma \ref{lem:Legendre} if $0<x-p/q\leq \Psi(q)$ then $p/q$ is a convergent of $x$. Furthermore, since $n$ is even and we are looking at approximations of $x$ exclusively from rationals to the left of $x$, it is a consequence of Lemma~\ref{lem:rationalcfexpansion} and Lemma~\ref{lem:Legendre} that the first $n$ partial quotients of $x$ must agree with the partial quotients of $\frac{p}{q}$, thus giving us \eqref{eq2}. The equivalence between \eqref{eq1} and \eqref{condition} now follows from the same argument as given in the proof of Proposition~\ref{prop: equivalence}.
\end{proof}
Instead of working with the set in \eqref{eq:set we want to approximate}, we will work with the set provided by Lemma \ref{lem: belongs to u(p/q)}.  For this reason, we introduce some additional notation. Given $\Psi$, $\omega$ and $p/q=[0;a_{1},\ldots,a_{n}]$ for $n$ even, let
$$S(p/q)=S(\Psi,\omega, p/q):=\left[\frac{1}{q^{2}\Psi(q)}-\frac{q_{n-1}}{q}, \frac{1}{q^{2}\Psi(q)(1-\omega(q))}-\frac{q_{n-1}}{q} \right].$$
Note that $S(p/q)$ is the set appearing in Proposition~\ref{prop: equivalence} \ref{S(p/q) set}. For convenience we will often suppress the dependence on $\Psi$ and $\omega$ from our notation and just write $S(p/q)$. We construct a well-behaved subset of $S(p/q)$ in Lemma \ref{lem: bound on S(p/q)}. Lemma \ref{lem: building U(p/q)} will essentially follow from Lemma \ref{lem: belongs to u(p/q)} and Lemma~\ref{lem: bound on S(p/q)}. Before we prove Lemma \ref{lem: bound on S(p/q)} we establish some basic properties of $S(p/q)$. This is the content of the next two lemmas.
\begin{lem}
\label{lem:b size}
Let $\Psi$ and $\omega$ satisfy \eqref{psi omega cond}. Let $(p,q)\in \Z\times \N$ and $\frac{p}{q}=[0;a_{1},\ldots,a_{n}]$ for $n$ even. Then if $q$ is sufficiently large, for any $b\in \N$ such that $[b,b+1]\cap S(p/q)\neq \emptyset$ we have
$$\frac{1}{2q^{2}\Psi(q)}\leq b\leq \frac{2}{q^{2}\Psi(q)}.$$
\end{lem}
\begin{proof}
These inequalities immediately follow our assumptions and $0\leq \tfrac{q_{n-1}}{q}\leq 1.$    
\end{proof}

Note that belonging to $S(p/q)$ infinitely often is not enough to guarantee an element belongs to $E(\Psi,\omega)$, we also need to ensure Proposition~\ref{prop: equivalence} \ref{E-bad} is satisfied. As stated in Remark~\ref{remark to prop} it is easier to work with condition \ref{i' statement}. We impose upper bounds on the partial quotients via the following sets: Given $(p,q)\in \Z\times \N$ where $\frac{p}{q}=[0;a_{1},\ldots, a_{n}]$ for $n$ even, an integer $b_{0}\in \N,$ and a possibly empty word $\bb=(b_{1},\ldots, b_{m})$, we define \begin{equation*}
        B(p/q,b_{0},\bb):=\left\{ I_{m+1}(b_{0};\bb c_{1}): c_{1}\in \left\{ \left\lceil \frac{1}{3q_{n+m+1}^{2}\Psi(q_{n+m+1})}\right\rceil, \ldots \right\}  \right\}\, ,
    \end{equation*}
    where $q_{n+m+1}$ is given by $\frac{p_{n+m+1}}{q_{n+m+1}}=[0;a_{1},\ldots,a_{n},b_{0},b_{1},\ldots, b_{m}]$. 
    We also let
    \begin{equation} \label{bad collection}
        \cB(p/q,b_{0},\bb):=\bigcup_{I\in B(p/q,b_{0},\bb)}I \,.
    \end{equation}
We emphasise that $\cB(p/q,b_{0},\bb)$ is an interval. The significance of this set is that if $$x\in I_{n+m+2}(0;a_{1},\ldots ,a_{n},b_{0},\bb c_{1})\text{ for some }c_{1}<\left\lceil \frac{1}{3q_{n+m+1}^{2}\Psi(q_{n+m+1})}\right\rceil,$$
then $$\left|x-\frac{p_{n+m+1}}{q_{n+m+1}}\right|> \Psi(q_{n+m+1}).$$ This follows from Lemma \ref{lem:approximation quality}. We will on occasion informally refer to the sets of the form $\cB(p/q,b_{0},\bb)$ as bad sets. This is because in our proofs they will consist of those fundamental intervals that we are trying to avoid. Given a word $\bb=(b_{0},b_{1},\cdots,b_{m})$ we will also adopt the notational convention that $B(p/q,\bb)=B(p/q,b_{0},b_{1}\cdots b_m)$ and $\cB(p/q,\bb)=\cB(p/q,b_{0},b_{1}\cdots b_{m}),$ i.e. the first entry in $\bb$ determines the integer and the remainder determines the finite word appearing in the definition of these sets. 

The following lemma shows that under the assumptions of Theorems \ref{thm: non-empty} and \ref{thm:main}, the set $S(p/q)$ will always be much larger than $\cB(p/q,b_{0},\bb)$ for any $\bb$ and $b_{0}$ such that $[b_{0},b_{0}+1]\cap S(p/q)\neq\emptyset$.
\begin{lem}
\label{lem:S bigger than bad}
Let $\Psi$ and $\omega$ satisfy \eqref{psi omega cond}. Assume that $$\omega(q)\geq \frac{45}{\Psi(q)}\Psi\left(\frac{1}{2q\Psi(q)}\right)$$ for all $q$ sufficiently large. Let $(p,q)\in \Z\times \N$ and $p/q=[0;a_{1},\ldots,a_{n}]$ for $n$ even. Then if $q$ is sufficiently large, for $b_{0}\in \N$ satisfying $[b_{0},b_{0}+1]\cap S(p/q)\neq\emptyset$ and a word $\bb$ we have $$\mathcal{L}(S(p/q))\geq 20\mathcal{L}(\cB(p/q,b_{0},\bb)).$$
\end{lem}
\begin{proof}
Let $(p,q)\in \Z\times \N$ be given. Let $b_{0}$ satisfy $[b_0,b_{0}+1]\cap S(p/q)\neq\emptyset$  and $\bb=(b_{1},\ldots,b_{m})$ be a possibly empty word. It is a consequence of Lemma \ref{lem:Subset measure bound} that 
\begin{equation} \label{eq: bad size}
\mathcal{L}(\cB(p/q,b_{0},\bb))\leq 9q_{n+m+1}^{2}\Psi(q_{n+m+1})\mathcal{L}(I_{m}(b_{0};\bb)).
\end{equation}
Using the fact $q\to q^{2}\Psi(q)$ is non-increasing we see that the right hand side of the inequality above is maximised when $\bb$ is the empty word. Thus 
\begin{equation}
\label{eq:empty word bound}
\mathcal{L}(\cB(p/q,b_{0},\bb))\leq 9q_{n+1}^{2}\Psi(q_{n+1})
\end{equation}for any word $\bb$.
By Lemma \ref{lem:b size} we know that $b_{0}\geq \frac{1}{2q^{2}\Psi(q)}.$ Using this inequality together with the inequality $q_{n+1}\geq b_{0}q,$ which follows from \eqref{eq:recursive formulas}, we have $q_{n+1}\geq \frac{1}{2q\Psi(q)}$. Using this bound together with the fact $q\to q^{2}\Psi(q)$ is non-increasing, the following measure bound follows from \eqref{eq:empty word bound} 
\begin{equation}
\label{eq:bad measure bound}
\mathcal{L}(\cB(p/q,b_{0},\bb))\leq \frac{9}{4q^{2}\Psi(q)^{2}}\Psi\left(\frac{1}{2q\Psi(q)}\right).
\end{equation}The following bound for the measure of $S(p/q)$ follows immediately from the definition:
\begin{equation}
\label{eq:S(p/q) measure bound}
    \mathcal{L}(S(p/q))= \frac{\omega(q)}{q^{2}\Psi(q)(1-\omega(q))}\geq \frac{\omega(q)}{q^{2}\Psi(q)}.
\end{equation}
Thus our result would follow from \eqref{eq:bad measure bound} and \eqref{eq:S(p/q) measure bound} if the following inequality held
$$\frac{\omega(q)}{q^{2}\Psi(q)}\geq \frac{180}{4q^{2}\Psi(q)^{2}}\Psi\left(\frac{1}{2q\Psi(q)}\right).$$ After rearranging this inequality and cancelling some terms, we see that it is satisfied for all $q$ sufficiently large by our assumption. Thus our result follows.
\end{proof}

The following lemma builds a well-behaved subset of $S(p/q)$.
 
\begin{lem} \label{lem: bound on S(p/q)}
  Let $\Psi$ and $\omega$ satisfy \eqref{psi omega cond}. Assume that $$\omega(q)\geq \frac{45}{\Psi(q)}\Psi\left(\frac{1}{2q\Psi(q)}\right)$$ for all $q$ sufficiently large. Let $(p,q)\in \Z\times \N$ and $\frac{p}{q}=[0;a_{1},\ldots,a_{n}]$ for $n$ even.  Then if $q$ is sufficiently large, there exists a finite set of consecutive words $\{\bb_{l}\}_{l=1}^{L}$ such that 
\begin{equation} \label{words contained in S(p/q)}
    \bigcup_{l=1}^{L}I(\bb_{l})\subset S(p/q)\quad \textrm{ and }\quad \mathcal{L}\left(\bigcup_{l=1}^{L}I(\bb_{l})\right)\geq \frac{9\mathcal{L}(S(p/q))}{100}\, .
\end{equation}
Moreover, if each consecutive word has length $m\geq 2$, then
      \begin{equation} \label{b digit restrition}
    b_{i}\in \left\{ 1,2, \ldots,\left\lfloor \frac{1}{3q_{n+i}^{2}\Psi(q_{n+i})}\right\rfloor \right\} \quad 1\leq i \leq m-1\, ,
    \end{equation} where $\frac{p_{n+i}}{q_{n+i}}=[0;a_{1},\ldots, a_{n},b_{0},b_{1},\ldots,b_{i-1}]$ for $1\leq i\leq m-1.$ 
\end{lem}

\begin{proof}
Our proof is split into a number of cases that are listed below. Many of these cases follow by similar arguments, and so for the purpose of brevity, we will on occasion only give an outline of the argument.  \\

\noindent \textbf{Case 1. - There exists $b_{0}\in \N$ such that $I_{0}(b_{0})\subseteq S(p/q)$.} Let $M_{1}$ be the smallest natural number such that $I_{0}(M_{1})\subset S(p/q)$ and let $M_{2}$ be the largest natural number satisfying $I_{0}(M_{2})\subseteq S(p/q).$ $M_{1}$ and $M_{2}$ are well defined because of our assumption. Then we have $$\bigcup_{b=M_{1}}^{M_{2}}I_{0}(b)\subseteq S(p/q).$$ It follows from a simple argument using the definitions of $M_{1}$ and $M_{2}$ that we have the following measure bound $$\mathcal{L}\left(\bigcup_{b=M_{1}}^{M_{2}}I_{0}(b)\right)\geq \frac{\mathcal{L}(S(p/q))}{3}\, .$$ Taking our set of consecutive words to be $\{b\}_{b=M_{1}}^{M_{2}}$ completes our proof in this case.\\

\noindent \textbf{Case 2. - There exists no $b_{0}\in \N$ such that $I_{0}(b_{0})\subseteq S(p/q)$ and there exists $b_{0}\in \N$ such that $I_{0}(b_{0})\cap S(p/q)\neq\emptyset$ and $I_{0}(b_{0}+1)\cap S(p/q)\neq\emptyset$.} It follows from our assumptions that $S(p/q) \subset I_{0}(b_{0})\cup I_{0}(b_{0}+1)$, otherwise we would be in Case 1. Thus either 
\begin{equation}
    \label{eq:subcasea}
    \mathcal{L}(S(p/q) \cap I_{0}(b_{0}))\geq \frac{\mathcal{L}(S(p/q))}{2}
\end{equation}
or 
\begin{equation}
    \label{eq:subcaseb}
    \mathcal{L}(S(p/q) \cap I_{0}(b_{0}+1))\geq \frac{\mathcal{L}(S(p/q))}{2}\, .
\end{equation}
We treat each of these subcases individually.

\noindent \textit{Case 2a. \eqref{eq:subcasea} holds. } It follows from our assumptions that $S(p/q)$ contains the right end point of $I_{0}(b_{0})$. Thus $S(p/q)\cap I_{1}(b_{0};1)\neq\emptyset.$ If  $I_{1}(b_{0};1)\subset S(p/q)\cap I_{0}(b_{0})$ then we can take $\{b_{0}1\}$ to be our set of words, as in this case $$\frac{\mathcal{L}(S(p/q))}{2}\leq \mathcal{L}(S(p/q)\cap I_{0}(b_{0}))\leq \mathcal{L}( I_{0}(b_{0}))= 2\mathcal{L}( I_{1}(b_{0};1))\, ,$$ where in the final equality we used the fact that $\mathcal{L}(I_{1}(b_{0};1))=1/2$ for all $b_{0}\in \Z$. Thus \eqref{words contained in S(p/q)} and \eqref{b digit restrition} are satisfied. Now suppose $S(p/q)\cap I_{0}(b_{0})\subset I_{1}(b_{0};1).$ The fundamental intervals $\{I_{2}(b_{0};1,n)\}_{n=1}^{\infty}$ are all contained in $I_{1}(b_{0};1)$ and are increasing in the sense that the right end point of $I_{2}(b_{0};1,n)$ is the left end point of   $I_{2}(b_{0};1,n+1)$ for all $n\in \N$. Thus, since $S(p/q)$ contains the right end point of $I_{1}(b_{0};1)$ it intersects $I_{2}(b_{0};1,n)$ for all $n$ sufficiently large. Therefore, $S(p/q)\cap \cB(p/q,b_{0},1)\neq\emptyset.$ It follows now from our assumption $S(p/q)\cap I_{0}(b_{0})\subset I_{1}(b_{0};1)$, \eqref{eq:subcasea} and Lemma \ref{lem:S bigger than bad} that  
\begin{equation}
\label{eq:bad inclusion}
\cB(p/q,b_{0},1)\subseteq S(p/q)\cap I_{1}(b_{0};1)
\end{equation}and 
\begin{equation}
    \label{eq:10bound}
10\mathcal{L}(\cB(p/q,b_{0},1))\leq \mathcal{L}(S(p/q)\cap I_{1}(b_{0};1))\, .
\end{equation} Thus
\begin{equation}
    \label{eq:10bad bound}
10\mathcal{L}\left(I_{2}\left(b_{0};1,\left\lceil \frac{1}{3q_{n+2}^{2}\Psi(q_{n+2})}\right\rceil\right)\right)\leq \mathcal{L}(S(p/q)\cap I_{1}(b_{0};1))\, .
\end{equation}
 It follows now from \eqref{eq:bad inclusion}, \eqref{eq:10bad bound} and Lemma \ref{lem:Quasi-multiplicative} that
\begin{equation}
\label{eq:first inclusion}
I_{2}\left(b_{0};1,\left\lfloor \frac{1}{3q_{n+2}^{2}\Psi(q_{n+2})}\right\rfloor \right) \subseteq S(p/q)\, .
\end{equation}
Let $M\geq 1$ be as small as possible that 
$$\bigcup_{b_{2}=M}^{\left\lfloor \frac{1}{3q_{n+2}^{2}\Psi(q_{n+2})}\right\rfloor}I_{2}\left(b_{0};1,b_{2} \right) \subset S(p/q).$$ Note that $M$ is well defined by \eqref{eq:first inclusion}. If $M=1,$ then we would have $I_{1}(b_{0};1)\subset S(p/q)\cap I_{0}(b_{0}),$ and so we can apply the argument given at the beginning of this case. If $M\geq 2$ then we can take our set of words to be $\{b_{0}1b_{2}\}_{b_{2}=M}^{\left\lfloor \frac{1}{3q_{n+2}^{2}\Psi(q_{n+2})}\right\rfloor}$ since by Lemma \ref{lem:Quasi-multiplicative}, \eqref{eq:subcasea}, \eqref{eq:10bound}, and the minimality of $M$, we have 
\begin{align}
    5\mathcal{L}\left(\bigcup_{b_{2}=M}^{\left\lfloor \frac{1}{3q_{n+2}^{2}\Psi(q_{n+2})}\right\rfloor}I_{2}\left(b_{0};1,b_{2} \right)\right)&\geq \mathcal{L}\left(\bigcup_{b_{2}=M-1}^{\left\lfloor \frac{1}{3q_{n+2}^{2}\Psi(q_{n+2})}\right\rfloor}I_{2}\left(b_{0};1,b_{2} \right)\right)\nonumber\\
    &\geq \mathcal{L}\big((S(p/q)\cap I_{1}(b_{0};1))\setminus \cB(p/q,b_{0},1)\big)\nonumber\\
    &\geq \mathcal{L}(S(p/q)\cap I_{1}(b_{0};1))-\frac{\mathcal{L}(S(p/q)\cap I_{1}(b_{0};1))}{10}\nonumber\\
    &=\frac{9\mathcal{L}(S(p/q)\cap I_{1}(b_{0};1))}{10}\nonumber\\
    &\geq \frac{9\mathcal{L}(S(p/q))}{20}\, . \label{eq: technique repeated later}
\end{align}
Thus 
$$\mathcal{L}\left(\bigcup_{b_{2}=M}^{\left\lfloor \frac{1}{3q_{n+2}^{2}\Psi(q_{n+2})}\right\rfloor}I_{2}\left(b_{0};1,b_{2} \right)\right)\geq \frac{9\mathcal{L}((S(p/q))}{100}\, .$$ This completes our proof in this case.\\

\noindent \textit{Case 2b. \eqref{eq:subcaseb} holds.} The fundamental intervals $\{I_{1}(b_{0}+1;n)\}$ are all contained in $I_{0}(b_{0}+1)$ and are decreasing in the sense that the left end point of $I_{1}(b_{0}+1;n)$ is the right endpoint of $I_{1}(b_{0}+1;n+1)$ for all $n\in \N$. It follows from this observation and \eqref{eq:subcaseb} that $$I_{1}(b_{0}+1;n)\subset S(p/q)$$ for all $n$ sufficiently large. Using this observation together with similar arguments used in Case 2a, it can be shown that there exists $M\in \N$ such that $$\bigcup_{b_{1}=M}^{\left\lfloor \frac{1}{3q_{n+1}^{2}\Psi(q_{n+1})}\right\rfloor}I_{1}\left(b_{0}+1;b_{1} \right) \subset S(p/q).$$ and 
$$\mathcal{L}\left(\bigcup_{b_{1}=M}^{\left\lfloor \frac{1}{3q_{n+1}^{2}\Psi(q_{n+1})}\right\rfloor}I_{1}\left(b_{0}+1;b_{1} \right)\right) \geq \frac{9\mathcal{L}(S(p/q))}{100}.$$ Taking our set of words to be $\{(b_0+1)b_{1}\}_{b_{1}=M}^{\left\lfloor \frac{1}{3q_{n+1}^{2}\Psi(q_{n+1})}\right\rfloor}$ completes our proof.\\

\noindent \textbf{Case 3. - There exists $b_{0}\in \N$ such that $S(p/q)\subset I_{0}(b_{0}).$ }We split this case into further subcases.\\

\noindent \textit{Case 3a - $S(p/q)\cap \cB(p/q,b_{0})\neq\emptyset.$} Replicating the arguments given in Case 2a we can show that  $S(p/q)\cap \cB(p/q,b_{0})\neq\emptyset$ implies $$I_{1}\left (b_{0},\left\lfloor \frac{1}{3q_{n+1}^{2}\Psi(q_{n+1})}\right \rfloor\right)\subset S(p/q).$$ We then choose the smallest $M$ such that $$\bigcup_{b_{1}=M}^{\left\lfloor \frac{1}{3q_{n+1}^{2}\Psi(q_{n+1})}\right \rfloor}I_{1}(b_{0},b_{1})\subset S(p/q).$$ Taking our set of words to be $\{b_{0}b_{1}\}_{b_{1}=M}^{\left\lfloor \frac{1}{3q_{n+1}^{2}\Psi(q_{n+1})}\right \rfloor}$, we can replicate the argument used in Case 2a to see that our conclusion is satisfied.\\

\noindent \textit{Case 3b - $S(p/q)\cap \cB(p/q,b_{0})=\emptyset$ and there exists $b_{1}$ such that $I_{1}(b_{0};b_{1})\subset S(p/q).$} If we let $M_{1}$ and $M_{2}$ be respectively the minimum and maximum values of $m$ for which
$I_{1}(b_{0};m)\subset S(p/q),$ then by the same reasoning as in Case 1, with a slight modification taking into account the fact that consecutive $1$st level fundamental intervals are of different sizes (Lemma~\ref{lem:Quasi-multiplicative} bounds the change in length), it can be shown that $\{b_{0}b_{1}\}_{b_{1}=M_{1}}^{M_{2}}$ satisfies our desired conclusion. \\

\noindent \textit{Case 3c - $S(p/q)\cap \cB(p/q,b_{0})=\emptyset,$ there exists no $b_{1}$ such that $I_{1}(b_{0};b_{1})\subset S(p/q)$ and there exists $b_{1}$ such that $I_{1}(b_{0};b_{1})\cap S(p/q)\neq \emptyset$ and $I_{1}(b_{0};b_{1}+1)\cap S(p/q)\neq \emptyset.$ } In this case, it follows from our assumptions that $S(p/q)\subset I_{1}(b_{0};b_{1})\cup I_{1}(b_{0};b_{1}+1).$ Thus we must have either $$\frac{\mathcal{L}(S(p/q))}{2}\leq \mathcal{L}(S(p/q)\cap I_{1}(b_{0};b_{1}))\quad\text{ or }\quad \frac{\mathcal{L}(S(p/q))}{2}\leq \mathcal{L}(S(p/q)\cap I_{1}(b_{0};b_{1}+1)).$$
We now construct our set of words using the same arguments as used in Case 2a.\\

\noindent \textit{Case 3d - $S(p/q)\cap \cB(p/q,b_{0})=\emptyset$ and there exists $b_{1}$ such that $S(p/q)\subset I_{1}(b_{0};b_{1}).$ } Since $S(p/q)\cap \cB(p/q,b_{0})=\emptyset$ we know that $b_{1}$ satisfies \eqref{b digit restrition}.  Now at this stage four possible events can occur:
\begin{enumerate}[leftmargin=2\parindent]
    \item $S(p/q)\cap \cB(p/q,b_{0},b_{1})\neq\emptyset$.
    \item $S(p/q)\cap \cB(p/q,b_{0},b_{1})=\emptyset$ and there exists $b_{2}$ such that $I_{2}(b_{0};b_{1},b_{2})\subset S(p/q).$
    \item $S(p/q)\cap \cB(p/q,b_{0},b_{1})=\emptyset,$ there exists no $b_{2}$ such that $I_{2}(b_{0};b_{1},b_{2})\subset S(p/q)$ and there exists $b_{2}$ such that $I_{2}(b_{0};b_{1},b_{2})\cap S(p/q)\neq \emptyset$ and $I_{2}(b_{0};b_{1},b_{2}+1)\cap S(p/q)\neq \emptyset.$
    \item $S(p/q)\cap \cB(p/q,b_{0},b_{1})=\emptyset$ and there exists $b_{2}$ such that $S(p/q)\subset I_{2}(b_{0};b_{1},b_{2}).$ 
\end{enumerate}
If $(1)$ holds then the arguments outlined in Case 3a can be applied. Similarly, if $(2)$ or $(3)$ hold then the arguments in Case 3b and Case 3c can be applied respectively. If $(4)$ holds then since $S(p/q)\cap \cB(p/q,b_{0},b_{1})=\emptyset$ we know that $b_{2}$ satisfies \eqref{b digit restrition} and we are essentially back at the beginning of Case 3d. We can then define an analogous list of cases (1)-(4) that describe the behaviour of $S(p/q)$ within $I_{2}(b_{0};b_{1},b_{2}).$ We can then proceed by an analogous analysis. Since the maximum length of an $m$-th level fundamental interval tends to zero as $m$ increases, it is clear that $S(p/q)$ cannot be contained in a fundamental interval $I_{m}(b_{0};b_{1},\ldots b_{m})$ for arbitrarily large values of $m$. Thus this process must eventually yield a word $\bb=(b_{1},\ldots,b_{m})$ such that  \eqref{b digit restrition} holds for $1\leq i\leq m$ and one of the following three events occur:
\begin{itemize}[leftmargin=2\parindent]
    \item[(1')]$S(p/q)\cap \cB(p/q,b_{0},\bb)\neq\emptyset$.
    \item[(2')]$S(p/q)\cap \cB(p/q,b_{0},\bb)=\emptyset$ and there exists $b$ such that $I_{m+1}(b_{0};\bb b)\subset S(p/q).$
    \item[(3')]$S(p/q)\cap \cB(p/q,b_{0},\bb)=\emptyset,$ there exists no $b$ such that $I_{2}(b_{0};\bb b)\subset S(p/q)$ and there exists $b$ such that $I_{2}(b_{0};\bb b)\cap S(p/q)\neq \emptyset$ and $I_{2}(b_{0};\bb (b+1))\cap S(p/q)\neq \emptyset.$
\end{itemize} Applying the same arguments as those outlined in Cases 3a, 3b and 3c we can then construct our desired set of words and complete our proof.  
\end{proof}
Equipped with Lemma \ref{lem: belongs to u(p/q)} we can now prove Lemma \ref{lem: building U(p/q)}.
\begin{proof}[Proof of Lemma \ref{lem: building U(p/q)}]
    Let $(p,q)\in \Z\times \N$ and $q$ be large. Let $\{\bb_{l}\}_{l=1}^{L}$ be as in Lemma \ref{lem: bound on S(p/q)}. Then \eqref{eq: inclusion statement} follows from Lemma \ref{lem: belongs to u(p/q)}. Equation \eqref{b digit restrition2} follows from \eqref{b digit restrition}. It remains to prove the measure bound \eqref{eq: measure bound}. 
    It follows from an application of the mean value theorem applied to the map $x\to \frac{1}{x}$ and Lemma \ref{lem:b size} that when considering the left shift of the words $\{0\bb_{l}\}_{l=1}^{L}$ we have that
\begin{equation}
\label{eq:Quasi bernoulli bounda}
\mathcal{L}\left(\bigcup_{l=1}^{L}I(0;\bb_{l})\right)\geq \frac{q^{4}\Psi(q)^{2}}{9}\mathcal{L}\left(\bigcup_{l=1}^{L}I(\bb_{l})\right)
\geq \frac{q^{4}\Psi(q)^{2}}{9}\frac{9\mathcal{L}(S(p/q))}{100} 
\geq \frac{q^{2}\omega(q)\Psi(q)}{200}.
\end{equation}
Thus applying Lemma \ref{lem:Fundamental interval}, Lemma \ref{lem:Quasi-multiplicative} and \eqref{eq:Quasi bernoulli bounda} we have
\begin{align*}
\mathcal{L}\left(\bigcup_{l=1}^{L}I(0;a_{1},\ldots,a_{n},\bb_{l})\right)&\geq \frac{1}{16}\mathcal{L}(I(0;a_{1},\ldots,a_{n}))\mathcal{L}\left(\bigcup_{l=1}^{L}I(0;\bb_{l})\right)\\
&\geq \frac{1}{32q^{2}}\cdot \frac{q^{2}\omega(q)\Psi(q)}{200}\\
&=\frac{\omega(q)\Psi(q)}{6400}.
\end{align*} This completes our proof. 
\end{proof}

\subsection{Technical results for Theorems \ref{thm: non-empty2} and \ref{thm:main2}} \label{sec: Technical results with divisibility cond}
The purpose of this subsection is to prove the following two lemmas. Both of these lemmas assume that a function $\Psi$ satisfies the square divisibility condition. We recall that $\Psi$ satisfies the square divisibility condition if for every $q\in \N$ we have $\Psi(q)^{-1}=0 \mod q^{2}.$

\begin{lem} \label{lem: non-empty S(p/q) divisibility cond}
Let $\Psi$ and $\omega$ satisfy \eqref{psi omega cond}. Furthermore suppose that $\Psi$ satisfies the square divisibility condition and
    \begin{equation*}
        \omega(q)\geq \frac{45}{16\Psi(q)}\Psi\left( \frac{1}{8\Psi(q)} \right)
    \end{equation*}
    for all $q$ sufficiently large. Let $(p,q)\in \Z\times \N$ and $\frac{p}{q}=[0;a_{1},\ldots,a_{n}]$ for $n$ even where $a_{1},a_{n}\geq 2$. 
   Then if $q$ is sufficiently large, we have that 
    \begin{equation*}
        I\left(\frac{1}{q^{2}\Psi(q)}-1;1,a_{n}-1, a_{n-1},\ldots, a_{2},a_{1}-1,1, \left\lfloor \frac{1}{3q_{*}^{2}\Psi(q_{*})}\right\rfloor\right)\subseteq S(p/q) \, ,
    \end{equation*}
    where $\frac{p_{*}}{q_{*}}=[0;a_{1},\ldots, a_{n},\tfrac{1}{q^{2}\Psi(q)}-1,1,a_{n}-1,a_{n-1},\ldots, a_{2},a_{1}-1,1]$.
    Equivalently,
    \begin{align*}
        I\left(0;a_{1},\ldots, a_{n}, \tfrac{1}{q^{2}\Psi(q)}-1,1,a_{n}-1, a_{n-1}, \ldots,a_{2}, a_{1}-1,1, \left\lfloor \frac{1}{3q_{*}^{2}\Psi(q_{*})}\right\rfloor \right) \\
        \subset \left\{x: \Psi(q)(1-\omega(q))\leq x-\frac{p}{q}\leq \Psi(q) \right\}\, .
    \end{align*}
\end{lem}
Lemma \ref{lem: non-empty S(p/q) divisibility cond} is the main technical tool that allows us to prove the non-empty part of Theorem \ref{thm: non-empty2}. The result still holds without the assumption $a_{1},a_{n}\geq 2,$ however, we include this assumption as it simplifies our proofs and is sufficient for our needs.
\begin{lem} \label{lem: building U(p/q) with divisibility}
  Let $\Psi$ and $\omega$ satisfy \eqref{psi omega cond}. Furthermore suppose that $\Psi$ satisfies the square divisibility condition and
    \begin{equation*}
        \omega(q)\geq \frac{45}{16\Psi(q)}\Psi\left( \frac{1}{8\Psi(q)} \right)
    \end{equation*}
    for all $q$ sufficiently large. Let $(p,q)\in \Z\times \N$ and $\frac{p}{q}=[0;a_{1},\ldots,a_{n}]$ for $n$ even, where the partial quotients satisfy
    \begin{equation*}
        a_{i}\leq \frac{1}{6 q^{2}\Psi(q)}\, ,  \quad 1\leq i\leq n\, .
    \end{equation*}
     Then if $q$ is sufficiently large, there exists a finite set of consecutive words $\{\bb_{l}\}_{l=1}^{L}$ such that 
\begin{equation} \label{contained SQD}
\bigcup_{l=1}^{L}I(0;a_{1},\ldots,a_{n},\bb_{l})\subset \left\{x:\Psi(q)(1-\omega(q))\leq x-\frac{p}{q}< \Psi(q)\right\}    
\end{equation}
and 
\begin{equation} \label{measure SQD}
\mathcal{L}\left(\bigcup_{l=1}^{L}I(0;a_{1},\ldots,a_{n},\bb_{l}))\right)\geq \frac{\omega(q)\Psi(q)}{10^8}.
\end{equation}
Moreover, if each $\bb_{l}$ has length $m\geq 2$, then
      \begin{equation} \label{b digits SQD}
    b_{i}\in \left\{ 1,2, \ldots,\left\lfloor \frac{1}{3q_{n+i}^{2}\Psi(q_{n+i})}\right\rfloor \right\} \quad 1\leq i \leq m-1\, ,
    \end{equation}
    where $\frac{p_{n+i}}{q_{n+i}}=[0;a_{1},\ldots, a_{n},b_{0},b_{1},\ldots,b_{i-1}]$ for $1\leq i\leq m-1.$ 
\end{lem}
Lemma \ref{lem: building U(p/q) with divisibility} will allow us to prove Theorem \ref{thm:main2}. It is an analogue of Lemma \ref{lem: building U(p/q)} under the square divisibility assumption. We highlight the additional assumption appearing in Lemma \ref{lem: building U(p/q) with divisibility} that $a_{i}\leq \frac{1}{6 q^{2}\Psi(q)}$ for $ 1\leq i\leq n$. This condition imposes additional technicalities in our proof of Theorem~\ref{thm:main2} in Section~\ref{sec: dim statements proof} that are not present in the proof of Theorem~\ref{thm:main}.

When $\Psi$ satisfies the square divisibility condition we have a much better understanding of where $S(p/q)$ lies in relation to the partition of the unit interval given by the fundamental intervals. In this subsection we utilise this property. We begin with the following lemma which gives a useful expression for the left endpoint of $S(p/q)$.

\begin{lem} \label{lem: where is S(p/q)}
    Let $\Psi$ and $\omega$ satisfy \eqref{psi omega cond}. Furthermore suppose that $\Psi$ satisfies the square divisibility condition. Let $(p,q)\in \Z\times \N$ and $\frac{p}{q}=[0;a_{1},\ldots,a_{n}]$ with $n$ even. Then the left endpoint of $S(p/q)$ is equal to the rational number
    \begin{equation*}
        \left[ \frac{1}{q^{2}\Psi(q)}-1;1,a_{n}-1,a_{n-1}, \ldots, a_{2}, a_{1}-1,1\right].
    \end{equation*}
    Furthermore, there exists $\bc=(c_{0},c_{1},\ldots, c_{m})$ with $m$ even that satisfies the following properties
    \begin{enumerate}[(i)]
        \item $[c_{0};c_{1},\ldots, c_{m}]=\left[ \frac{1}{q^{2}\Psi(q)}-1;1,a_{n}-1,a_{n-1}, \ldots, a_{2}, a_{1}-1,1\right]$,
        \item $c_{0}=\frac{1}{q^{2}\Psi(q)}-1$,
        \item $0<c_{i}\leq \max_{1\leq i \leq n} a_{i} +1$ for every $1\leq i \leq m$.
    \end{enumerate}
    
\end{lem}

\begin{rem}
We remark that in the first expression for the left endpoint of $S(p/q)$ there are an even number of digits after the integer part $\frac{1}{q^{2}\Psi(q)}-1.$ The problem with this expression is that it may produce digits that take the value zero, i.e. $a_{n}-1$ and $a_{1}-1$ may both equal zero. This is why we have to introduce the word $\bc.$
\end{rem}

\begin{proof}[Proof of Lemma \ref{lem: where is S(p/q)}]
    Recall the formulas 
\begin{equation*}
    \frac{q_{n-1}}{q_{n}}=[0;a_{n},\ldots, a_{1}] \quad \text{ and } \quad -[0;b_{1},\ldots, b_{n}]=[-1;1,b_{1}-1,b_{2}, \ldots, b_{n}]\, .
\end{equation*}
Both are standard results in continued fraction theory, see for example \cite{PoorShall92}. It is then clear to see that since $\frac{1}{q^{2}\Psi(q)}\in \N,$ the left endpoint of $S(p/q)$ is
\begin{align*}
    \frac{1}{q^{2}\Psi(q)}-\frac{q_{n-1}}{q}&=\left[ \frac{1}{q^{2}\Psi(q)}-1;1,a_{n}-1,a_{n-1},\ldots, a_{2}, a_{1}\right]\\
    &=\left[ \frac{1}{q^{2}\Psi(q)}-1;1,a_{n}-1,a_{n-1},\ldots, a_{2}, a_{1}-1,1\right] .
\end{align*}
Define the word $\bc$ according to the rule 
\begin{equation*}
    \bc=\begin{cases}
        \left( \frac{1}{q^{2}\Psi(q)}-1,1,a_{n}-1,a_{n-1},\ldots, a_{2}, a_{1}-1,1\right) \quad & \text{ if } a_{1},a_{n}\geq 2\, ,\\
         \left( \frac{1}{q^{2}\Psi(q)}-1,1+a_{n-1},a_{n-2}\ldots, a_{2}, a_{1}-1,1\right) & \text{ if } a_{1}\geq 2 \, \text{ and } a_{n}=1\, ,\\
          \left( \frac{1}{q^{2}\Psi(q)}-1,1,a_{n}-1,a_{n-1},\ldots,a_{3}, a_{2}+1\right) & \text{ if } a_{1}=1 \, \text{ and } a_{n}\geq2\, , \\
          \left( \frac{1}{q^{2}\Psi(q)}-1,1+a_{n-1},a_{n-2},\ldots,a_{3}, a_{2}+1\right) & \text{ if } a_{1}=1 \text{ and } a_{n}=1\, .
    \end{cases}
\end{equation*}
Note that in all of the above cases $c_{0}=\frac{1}{q^{2}\Psi(q)}-1$ so $(ii)$ holds. It can easily be verified that each of the digits $c_{i}$ satisfy $(iii)$. By using the formula $[\ldots, b, 0, c, \ldots]=[\ldots, b+c,\ldots]$ it is easy to see that
\begin{equation*}
    [\bc]=\left[\frac{1}{q^{2}\Psi(q)}-1;1,a_{n}-1,a_{n-1},\ldots, a_{2},a_{1}-1,1\right]. 
\end{equation*}
So $(iii)$ holds. Finally, it follows our formula for $\bc$ that the number of digits appearing after $\frac{1}{q^{2}\Psi(q)}-1$ is $n+2$, $n$, or $n-2$. Since $n$ is even, this implies $m$ is even.
\end{proof}

We need the following modified version of Lemma~\ref{lem:S bigger than bad}.

\begin{lem} \label{lem:S bigger than bad 2}
    Let $\Psi$ and $\omega$ satisfy \eqref{psi omega cond}. Furthermore suppose that $\Psi$ satisfies the square divisibility condition and
    \begin{equation*}
        \omega(q)\geq \frac{45}{16\Psi(q)}\Psi\left( \frac{1}{8\Psi(q)} \right)
    \end{equation*}
    for all $q$ sufficiently large. Let $(p,q)\in \Z\times \N$ and $\frac{p}{q}=[0;a_{1},\ldots,a_{n}]$ for $n$ even. 
    Then if $q$ is sufficiently large, we have 
    \begin{equation*}
        \cL(S(p/q))\geq 20 \cL\left(\cB\left(p/q, \bc \right)\right)
    \end{equation*}
    where $\bc$ is as in Lemma \ref{lem: where is S(p/q)}.
\end{lem}

\begin{proof}
    By Lemma \ref{lem:Subset measure bound} we have that
\begin{equation} \label{eq: size of later bad}
    \cL\left(\cB\left(p/q,\bc\right)\right) \leq 9 q_{*}^{2}\Psi(q_{*}) \cL\left( I\left(\bc\right)\right)
\end{equation}
where $\frac{p_{*}}{q_{*}}=\left[0,a_{1}\ldots,a_{n},\bc\right].$ Using Lemma~\ref{lem:Quasi-multiplicative}, the fact that $1-\frac{q_{n-1}}{q_{n}}=\frac{q_{n}-q_{n-1}}{q_{n}}=[0;c_{1},\ldots, c_{m}]$ is in reduced form, and Lemma~\ref{lem: where is S(p/q)} $(ii)$, we have that
\begin{equation}
\label{eq:q_{2n+2} bound}
 q_{*}=q_{n+1+m}\left( (a_{1},\ldots,  a_{n})\left(\frac{1}{q^{2}\Psi(q)}-1\right)(c_{1},\ldots,c_{m}) \right) \geq \frac{1}{8\Psi(q)} \, .
\end{equation}

Combining \eqref{eq: size of later bad}, \eqref{eq:q_{2n+2} bound}, the non-increasing property of $q\mapsto q^{2}\Psi(q)$, and Lemma~\ref{lem:Fundamental interval}, we have
\begin{equation} \label{eq: size of bad divisibility case}
    \cL\left(\cB\left(p/q, \bc\right)\right) \leq \frac{9}{64q^{2}\Psi(q)^{2}}\Psi\left(\frac{1}{8\Psi(q)}\right).  
\end{equation}
The rest of this proof is identical to the latter part of the proof of Lemma~\ref{lem:S bigger than bad}. We use the simple lower bound $\cL(S(p/q))\geq \frac{\omega(q)}{q^{2}\Psi(q)}$ and then compare it with the upper bound for $\cL\left(\cB\left(p/q, \bc\right)\right)$ obtained above to see, via our lower bound on the decay rate of $\omega$,  that $\cL\left(S(p/q)\right) \geq 20 \cL\left(\cB\left(p/q, \bc\right)\right)$ as required.

\end{proof}

Equipped with Lemma \ref{lem:S bigger than bad 2} we can now prove Lemma \ref{lem: non-empty S(p/q) divisibility cond}.
\begin{proof}[Proof of Lemma \ref{lem: non-empty S(p/q) divisibility cond}]
For ease of notation write 
\begin{equation*}
    \bb=\left(\bc,\left\lfloor \frac{1}{3q_{*}^{2}\Psi(q_{*})}\right\rfloor\right)
\end{equation*}
where $\bc$ is given by $$\bc=\left(\frac{1}{q^{2}\Psi(q)}-1,1,a_{n}-1, a_{n-1},\ldots, a_{2},a_{1}-1,1\right).$$ By our assumption $a_{1},a_{n}\geq 2,$ all of the digits in $\bc$ are strictly positive integers. Note that the left endpoint of $I(\bb)$ is the right endpoint of $\cB\left(p/q, \bc\right)$, and $\cB(p/q,\bc)$ has the same left endpoint as $S(p/q)$ by Lemma~\ref{lem: where is S(p/q)}. By Lemma~\ref{lem:Quasi-multiplicative} we have the bound
\begin{equation*}
    \cL\left(I(\bb)\cup \cB\left(p/q, \bc\right)\right) \leq 5 \cL\left(\cB\left(p/q, \bc\right)\right),
\end{equation*}
and so by Lemma~\ref{lem:S bigger than bad 2}, $S(p/q)$ is sufficiently large, relative to $\cB\left(p/q, \bc\right)$, to guarantee that 
\begin{equation*}
    S(p/q) \supseteq I(\bb)\cup \cB\left(p/q,\bc\right)\, .
\end{equation*}
Hence $I(\bb)\subset S(p/q)$ as claimed.

\end{proof}

The following lemma is important because it tells us in a quantifiable way that $S(p/q)$ avoids fundamental intervals with big partial quotients. In this lemma, given $x\in \R$ and $A\subset \R,$ we let $d(x,A)$ denote the Euclidean distance between $x$ and $A$.

\begin{lem} \label{lem: far from bad}
    Let $\Psi$ and $\omega$ satisfy \eqref{psi omega cond}. Furthermore suppose that $\Psi$ satisfies the square divisibility condition and
    \begin{equation*}
        \omega(q)\geq \frac{45}{16\Psi(q)}\Psi\left( \frac{1}{8\Psi(q)} \right)
    \end{equation*}
    for all $q$ sufficiently large. Let $(p,q)\in \Z\times \N$ and $\frac{p}{q}=[0;a_{1},\ldots,a_{n}]$ for $n$ even, where the partial quotients satisfy
    \begin{equation} \label{cond: a_i bounded size divisibility statment 1}
        a_{i}\leq  \frac{1}{6 q^{2}\Psi(q)}\, , \quad 1\leq i\leq n\, .
    \end{equation}
    Then for $\bc=(c_{0},\ldots,c_{m})$ as in Lemma \ref{lem: where is S(p/q)}, if $q$ is sufficiently large and $0\leq j\leq m-1,$ then the following statements hold:
    \begin{enumerate}[(i)]
        \item If $j$ is odd, then 
    \begin{equation*}
d\left([c_{0};c_{1},\ldots,c_{j},c_{j+1}+1],\cB(p/q,\left(c_{0},\ldots, c_{j}\right)) \right)>\frac{1}{640}\cL(\cB(p/q,(c_{0},\ldots,c_{j})))\, ,
    \end{equation*}
    \item If $j$ is even and $c_{j+1}=1$, then
    \begin{equation*}
        d\left( [c_{0};c_{1},\ldots,c_{j+1},c_{j+2}+1],\cB(p/q,(c_{0},\ldots, c_{j},c_{j+1})) \right)>\frac{1}{40960}\cL(\cB(p/q,(c_{0},\ldots,c_{j})))\, .
    \end{equation*}
    \end{enumerate}
\end{lem}

\begin{proof}
 Assume for now that $j$ is odd, we will deal with case $(ii)$ later. To prove this statement we will focus on a sub-collection of the fundamental intervals lying between $[c_{0};c_{1},\ldots,c_{j},c_{j+1}+1]$ and $\cB(p/q,\left(c_{0},\ldots, c_{j}\right)$. The proof of our result will then rely upon two calculations. The first is a lower bound for the number of fundamental intervals (at level $j+1$) in our subcollection. 
 The second calculation yields a lower bound for the size of such a fundamental interval.\par 
    The left endpoint of $\cB(p/q,(c_{0},\ldots, c_{j}))$ is
    \begin{equation*}
        \left[c_{0};c_{1},\ldots, c_{j}, \left\lfloor \frac{1}{3q_{n+1+j}^{2}\Psi(q_{n+1+j})}\right\rfloor\right]\, ,
    \end{equation*}
    where 
    \begin{equation*}
        \frac{p_{n+1+j}}{q_{n+1+j}}=[0;a_{1},\ldots, a_{n},c_{0},c_{1},\ldots, c_{j}]\, .
    \end{equation*}
    We then have the bound 
    \begin{align} \label{eq: number of intervals}
        &\#\left\{I(c_{0};c_{1},\ldots, c_{j},b):c_{j}+2\leq b\leq \left\lfloor\frac{1}{3q_{n+1}^{2}\Psi(q_{n+1})}\right\rfloor\right\}\nonumber\\
        =&\left\lfloor \frac{1}{3q_{n+1}^{2}\Psi(q_{n+1})}\right\rfloor - c_{j+1}-1 \nonumber\\
        \geq&  \frac{1}{4q_{n+1}^{2}\Psi(q_{n+1})}-\frac{1}{6q^{2}\Psi(q)}-2 \quad\quad \text{ (by \eqref{cond: a_i bounded size divisibility statment 1} \& Lemma~\ref{lem: where is S(p/q)} $(iii)$)}\nonumber\\
        \geq &  \frac{1}{4q_{n+1}^{2}\Psi(q_{n+1})}-\frac{1}{5q^{2}\Psi(q)} \quad \quad \text{ (for $q$ sufficiently large)}\nonumber\\
        \geq & \frac{1}{20 q_{n+1}^{2}\Psi(q_{n+1})} .
    \end{align}
     The last inequality is a consequence of $q=q_{n}$ and the non-increasing property of $q\mapsto q^{2}\Psi(q)$, so $q^{2}\Psi(q)\geq q_{n+1}^{2}\Psi(q_{n+1})$. Note that there may be significantly more fundamental intervals between $[c_{0};c_{1},\ldots,c_{j},c_{j+1}+1]$ and $\cB(p/q,\left(c_{0},\ldots, c_{j}\right)$, however, the last partial quotient will be potentially larger, which will affect the next calculation. For every fundamental interval $I\left( c_{0};c_{1},\ldots,c_{j},b\right)$ with $b\leq \frac{1}{3q_{n+1}^{2}\Psi(q_{n+1})}$ we have, by Lemma~\ref{lem:Quasi-multiplicative}, that the length of the fundamental interval is bounded from below by
    \begin{align}\label{eq: length of an one interval}
        \cL\left( I\left( c_{0};c_{1},\ldots,c_{j},b\right)\right)&\geq \cL\left( I_{j+1}\left( c_{0};c_{1},\ldots,c_{j},\frac{1}{3q_{n+1}^{2}\Psi(q_{n+1})}\right)\right)\nonumber\\
        &\geq \frac{1}{32}\frac{1}{\left(\frac{1}{3q_{n+1}^{2}\Psi(q_{n+1})}\right)^{2}}\cL\left( I_{j}\left( c_{0};c_{1},\ldots,c_{j}\right)\right)
    \end{align}
    Multiplying \eqref{eq: number of intervals} by \eqref{eq: length of an one interval} we obtain that 
    \begin{equation} \label{eq: nearly there to proving minilemma}
        d\left([c_{0};c_{1},\ldots,c_{j},c_{j+1}+1],\cB(p/q,\left(c_{0},\ldots, c_{j}\right)) \right) >\tfrac{9}{640}q_{n+1}^{2}\Psi(q_{n+1})\cL\left( I_{j}\left( c_{0};c_{1},\ldots,c_{j}\right)\right).
    \end{equation}
    By \eqref{eq: bad size} and the non-increasing property of $q\mapsto q^{2}\Psi(q),$ we have that
\begin{align} \label{eq: Bad size again}
    \cL\left(\cB\left(p/q, (c_{0},c_{1},\ldots,c_{j})\right)\right)&\leq 9 q_{n+1+j}^{2}\Psi(q_{n+1+j}) \cL\left(I_{j}\left(c_{0};c_{1},\ldots,c_{j}\right)\right)\nonumber\\
    &\leq 9 q_{n+1}^{2}\Psi(q_{n+1}) \cL\left(I_{j}\left(c_{0};c_{1},\ldots,c_{j}\right)\right)\, ,
\end{align}
and so by \eqref{eq: nearly there to proving minilemma} we have
\begin{equation*}
   d\left([c_{0};c_{1},\ldots,c_{j},c_{j+1}+1],\cB(p/q,\left(c_{0},\ldots, c_{j}\right)) \right) \geq \frac{1}{640} \cL\left(\cB\left(p/q, (c_{0},c_{1},\ldots,c_{j})\right)\right)\, .
\end{equation*}
This completes the proof in case $(i)$. Now consider case $(ii)$. Since $j$ is even and $j\leq m-2$ the digit $c_{j+2}$ is well-defined. Consider
\begin{equation*}
     d\left( [c_{0};c_{1},\ldots,c_{j+1},c_{j+2}+1],\cB(p/q,(c_{0},\ldots, c_{j},c_{j+1})) \right)\, .
\end{equation*}
Now $j+1$ is odd so we can use \eqref{eq: number of intervals}, \eqref{eq: length of an one interval}, Lemma~\ref{lem:Quasi-multiplicative} combined with our assumption $c_{j+1}=1$, and \eqref{eq: Bad size again}, to obtain 
\begin{align*}
 &d\left( [c_{0};c_{1},\ldots,c_{j+1},c_{j+2}+1],\cB(p/q,(c_{0},\ldots, c_{j},c_{j+1})) \right)\\
 >&\frac{9}{640}q_{n+1}^{2}\Psi(q_{n+1})\cL\left( I_{j+1}\left( c_{0};c_{1},\ldots,c_{j+1}\right)\right)\\
    >&\frac{9}{40960}q_{n+1}^{2}\Psi(q_{n+1})\cL\left( I_{j}\left( c_{0};c_{1},\ldots,c_{j}\right)\right)\, \\
 >&\frac{1}{40960} \cL\left(\cB\left(p/q, (c_{0},c_{1},\ldots,c_{j})\right)\right)
\end{align*}
completing the proof in case $(ii)$.
\end{proof}

 The following lemma is our analogue of Lemma \ref{lem: bound on S(p/q)} in the context of Theorem \ref{thm:main2}. Its proof is similar, however there is one important distinction. In the proof of Lemma~\ref{lem: bound on S(p/q)} the first bad set that $S(p/q)$ can possibly intersect is $\cB(p/q,b_{0})$. Lemma~\ref{lem:S bigger than bad} then tells us that $S(p/q)$ is much larger than $\cB(p/q,b_{0})$, so if the two sets were to intersect only a small amount of $S(p/q)$ would be lost. In the context of Lemma \ref{lem: building U(p/q) with divisibility}, it is possible for $\cB(p/q,b_{0})$ to be significantly larger than $S(p/q)$, and so a large part of $S(p/q)$ can be lost if the two sets intersect. To combat this we use our square divisibility condition which by Lemma \ref{lem: where is S(p/q)} gives us more detailed information about the location of $S(p/q)$. Moreover, by Lemma~\ref{lem: far from bad} under mild conditions on the partial quotients of $p/q,$ we can ensure $S(p/q)$ is sufficiently far from a bad set.

\begin{lem} \label{lem: bound on S(p/q) with divisibility}
    Let $\Psi$ and $\omega$ satisfy \eqref{psi omega cond}. Assume that $\Psi$ satisfies the square divisibility condition and that 
    \begin{equation} \label{cond: extra decay requirement}
         \omega(q)\geq \frac{45}{16\Psi(q)}\Psi\left(\frac{1}{8\Psi(q)}\right)
    \end{equation}
    for all $q\in \N$ sufficiently large. Let $(p,q)\in \Z\times \N$ and $\frac{p}{q}=[0;a_{1},\ldots,a_{n}]$ for $n$ even, where the partial quotients satisfy
    \begin{equation} \label{cond: a_i bounded size divisibility statment}
        a_{i}\leq \frac{1}{6 q^{2}\Psi(q)}\, , \quad 1\leq i\leq n\, .
    \end{equation}
    Then if $q$ is sufficiently large, there exists a finite set of consecutive words $\{\bb_{l}\}_{l=1}^{L}$ such that 
\begin{equation} \label{eq: measure outcome with SqD}
    \bigcup_{l=1}^{L}I(\bb_{l})\subset S(p/q)\quad \textrm{ and }\quad \mathcal{L}\left(\bigcup_{l=1}^{L}I(\bb_{l})\right)\geq \frac{\mathcal{L}(S(p/q))}{10^{6}}.
\end{equation}
Moreover, if each $\bb_{l}$ has length $m\geq 2$, then
      \begin{equation} \label{b digit restrition divisibility statement}
    b_{i}\in \left\{ 1,2, \ldots,\left\lfloor \frac{1}{3q_{n+i}^{2}\Psi(q_{n+i})}\right\rfloor \right\} \quad 1\leq i \leq m-1\, ,
    \end{equation} where $\frac{p_{n+i}}{q_{n+i}}=[0;a_{1},\ldots, a_{n},b_{0},b_{1},\ldots,b_{i-1}]$ for $1\leq i\leq m-1.$
\end{lem}

\begin{proof}

The proof is similar to that of Lemma~\ref{lem: bound on S(p/q)} and involves a lengthy case analysis. For the purpose of brevity, when the argument for a case is similar to one that has appeared previously, then we will only provide an outline for why the result follows.\par 

\noindent \textbf{Case $1^{*}$. - There exists $b_{0}\in \N$ such that $I_{0}(b_{0})\subseteq S(p/q)$.} The proof is the same as that appearing in Lemma~\ref{lem: bound on S(p/q)}. That is, take the set of consecutive words to be $\{b\}_{b=M_{1}}^{M_{2}}$ where $M_{1}=\min\{K: I_{0}(K)\subseteq S(p/q)\}$ and $M_{2}=\max\{K: I_{0}(K)\subseteq S(p/q)\}.$ The same calculations as in  case 1 of the proof of Lemma~\ref{lem: bound on S(p/q)} show this to be a suitable collection of words. \\

\noindent \textbf{Case $2^{*}$. - There exists no $b_{0}\in \N$ such that $I_{0}(b_{0})\subset S(p/q)$ and there exists $b_{0}\in \N$ such that $I_{0}(b_{0})\cap S(p/q)\neq \emptyset$ and $I_{0}(b_{0}+1)\cap S(p/q)\neq \emptyset$}. By Lemma~\ref{lem: where is S(p/q)} $(ii)$ we have that $b_{0}=c_{0}=\frac{1}{q^{2}\Psi(q)}-1$. We split into two subcases. \\

\noindent \textbf{Case $2^{*}a$. - $\cL(S(p/q))\geq 20 \cL(\cB(p/q,c_{0}))$.} We follow the proof of case 2 appearing in Lemma~\ref{lem: bound on S(p/q)}. That is, if 

\begin{enumerate}[leftmargin=2\parindent]
    \item $\cL(S(p/q)\cap I_{0}(c_{0})) \geq \frac{\cL(S(p/q))}{2}$: If $I_{1}(c_{0};1)\subset S(p/q)$ then we take $\{(c_{0},1)\}$ to be our set of words. If $I_{1}(c_{0};1)$ is not contained in $S(p/q)$ then we can take our set of words to be 
    \begin{equation*}
    \{(c_{0};1,b)\}_{b=M}^{\left\lfloor \frac{1}{3q_{n+2}^{2}\Psi(q_{n+2})}\right\rfloor} \ \quad \text{ with }\quad M:=\min\left\{K: I_{2}(c_{0};1,K)\subset S(p/q)\right\}\,  .
    \end{equation*}
        \item $\cL(S(p/q)\cap I_{0}(c_{0}+1)) \geq \frac{\cL(S(p/q))}{2}$: Then we can take our set of words to be 
    \begin{equation*}
        \{(c_{0}+1,b)\}_{b=M}^{\left\lfloor \frac{1}{3q_{n+1}^{2}\Psi(q_{n+1})}\right\rfloor} \quad \text{ with } \quad M:=  \min \{K: I_{1}(c_{0}+1;K)\subset S(p/q)\}\, . 
    \end{equation*}
\end{enumerate}
Both sets of consecutive words are shown to satisfy \eqref{eq: measure outcome with SqD} and \eqref{b digit restrition divisibility statement} in the course of the proof of Lemma~\ref{lem: bound on S(p/q)}.\\

\noindent \textbf{Case $2^{*}b$. - $\cL(S(p/q))< 20 \cL(\cB(p/q,c_{0}))$.} By the assumption of case $2^{*}$ we have that $S(p/q)$ contains the right endpoint of $I_{0}(c_{0})$. Consider the following subcases;\\

\noindent \textit{Case $2^{*}bi$. - $c_{1}>1$.} Since $I_{1}(c_{0};1)\cap S(p/q)\neq\emptyset$ and $S(p/q)$ contains the right endpoint of $I_{0}(c_{0}),$ we have
\begin{equation*}
    \bigcup_{b=1}^{c_{1}-1}I_{1}(c_{0};b) \subseteq S(p/q)\, .
\end{equation*}
By \eqref{cond: a_i bounded size divisibility statment}, Lemma~\ref{lem: where is S(p/q)} $(iii)$, and the non-increasing property of $q\mapsto q^{2}\Psi(q)$, the set is disjoint from $\cB(p/q,c_{0})$. Moreover the collection contains the fundamental interval $I_{1}(c_{0};1)$ and it can be shown, via Lemma~\ref{lem:Quasi-multiplicative} and our assumption in case $2^{*}b$ that
\begin{align}
    \cL\left( \bigcup_{b=1}^{c_{1}-1}I_{1}(c_{0};b) \right) \geq \cL\left( I_{1}(c_{0};1) \right) 
    &\geq \frac{1}{32} \cL\left( I_{0}(c_{0}) \right)
\nonumber \\
&\geq \frac{1}{32}\cL\left( \cB(p/q,c_{0})\right) \nonumber \\
&\geq \frac{1}{640}\cL(S(p/q))\, .
\end{align}
Thus \eqref{eq: measure outcome with SqD} and \eqref{b digit restrition divisibility statement} are satisfied by the collection of words $\{(c_{0},b)\}_{b=1}^{c_{1}-1}$.\\

\noindent \textit{Case $2^{*}bii$. - $c_{1}=1$.} Since $I_{2}(c_{0};1,c_{2})\cap S(p/q)\neq \emptyset$ and $S(p/q)$ contains the right end point of $I_{0}(c_{0})$ we have 
\begin{equation*}
    \bigcup_{b=c_{2}+1}^{\left\lfloor \frac{1}{3q_{n+2}^{2}\Psi(q_{n+2})}\right\rfloor} I_{2}(c_{0};1,b) \subseteq S(p/q)\setminus \cB(p/q,(c_{0},1))\, .
\end{equation*}
 By Lemma~\ref{lem: far from bad}, with $j=0$ and $c_{1}=1$, and our assumption in case $2^{*}b$, we have that
 \begin{equation*}
     \cL\left(\bigcup_{b=c_{2}+1}^{\left\lfloor \frac{1}{3q_{n+2}^{2}\Psi(q_{n+2})}\right\rfloor} I_{2}(c_{0};1,b) \right) \geq \frac{1}{40960} \cL\left( \cB(p/q,c_{0})\right) \geq \frac{1}{819200} \cL\left(S(p/q)\right)\, .
 \end{equation*}
 Hence \eqref{eq: measure outcome with SqD} and \eqref{b digit restrition divisibility statement} are satisfied by the collection of words $\{(c_{0},1,b)\}_{b=c_{2}+1}^{\left\lfloor \frac{1}{3q_{n+2}^{2}\Psi(q_{n+2})}\right\rfloor}$.\\

\noindent \textbf{Case $3^{*}$. - There exists $b_{0}\in \N$ such that $I_{0}(b_{0})\supseteq S(p/q)$.} Note again, by Lemma~\ref{lem: where is S(p/q)} $(ii)$, that $b_{0}=c_{0}=\frac{1}{q^{2}\Psi(q)}-1$. Similarly to case 3 of Lemma~\ref{lem: bound on S(p/q)} we split into further subcases. \\

\noindent \textit{Case $3^{*}a$ - $S(p/q)\cap \cB(p/q, c_{0}) \neq \emptyset$.} This case is not possible. To see this note that the left endpoint of $S(p/q)$ lies inside $I_{1}(c_{0};c_{1})$ by Lemma~\ref{lem: where is S(p/q)}. By \eqref{cond: a_i bounded size divisibility statment} and Lemma~\ref{lem: where is S(p/q)} $(iii)$ it is clear that $\cB(p/q,c_{0})$ is disjoint from $I_{1}(c_{0};c_{1})$ and lies to its left. Since the left endpoint of $S(p/q)$ is contained in $I_{1}(c_{0};c_{1})$ we must have $S(p/q)\cap \cB(p/q,c_{0})=\emptyset$.\\

\noindent \textit{Case $3^{*}b$ - $S(p/q)\cap \cB(p/q, c_{0})= \emptyset$ and there exists $b_{1}$ such that $I_{1}(c_{0};b_{1})\subset S(p/q)$.} The argument is the same as in case $3b$ of Lemma~\ref{lem: bound on S(p/q)}. That is,  the collection of words
    $\{(c_{0},b)\}_{b=M_{1}}^{M_{2}}$ with $M_{1}=\min\{K: I_{1}(c_{0};K)\subseteq S(p/q)\}$ and $M_{2}=\max\{K: I_{1}(c_{0};K)\subseteq S(p/q)\}$,
\eqref{eq: measure outcome with SqD} and \eqref{b digit restrition divisibility statement}.\\

\noindent \textit{Case $3^{*}c$ - $S(p/q)\cap \cB(p/q, c_{0})= \emptyset,$  there exists no $b_{1}$ such that $I_{1}(c_{0};b_{1})\subset S(p/q)$ and there exists $b_{1}$ such that $I_{1}(c_{0};b_{1})\cap S(p/q)\neq \emptyset$ and $I_{1}(c_{0};b_{1}+1)\cap S(p/q)\neq \emptyset$.} By Lemma~\ref{lem: where is S(p/q)} we know in this case we must have $b_{1}+1=c_{1}$. Moreover, by \eqref{cond: a_i bounded size divisibility statment}, Lemma \ref{lem: where is S(p/q)}(iii), and the non-increasing assumption for $q\to q^{2}\Psi(q)$ the digit $c_{1}$ must satisfy \eqref{b digit restrition divisibility statement}. We split into further subcases again. \\

\noindent\textit{Case $3^{*}ci$ - $\cL(S(p/q))\geq  20 \cL(\cB(p/q,(c_{0}, c_{1}))$.} 
Consider the subcases 
\begin{enumerate}[leftmargin=2\parindent]
    \item $\cL(S(p/q)\cap I_{1}(c_{0};c_{1}-1))\geq \tfrac{\cL(S(p/q))}{2}$: Split into another two further subcases\\
    
    \begin{enumerate}[leftmargin=\parindent]
        \item $S(p/q)\cap I_{1}(c_{0};c_{1}-1)\supseteq I_{2}(c_{0};c_{1}-1,1)$. In this case we take our collection of words to be $\{(c_0,c_{1}-1,1)\}.$ It is then straightforward to check that \eqref{eq: measure outcome with SqD} and \eqref{b digit restrition divisibility statement} are satisfied by this collection.


\item $S(p/q) \cap I_{1}(c_{0};c_{1}-1)\subset I_{2}(c_{0};c_{1}-1,1)$. Then take the collection of words 
\begin{equation*}
    \{(c_{0},c_{1}-1,1,b)\}_{b=M}^{\left\lfloor \frac{1}{3q_{n+3}^{2}\Psi(q_{n+3})} \right\rfloor} \quad \text{ with } \quad M:=\min \left\{K: I_{3}(c_{0};c_{1}-1,1,K)\subseteq S(p/q) \right\}.
\end{equation*}
Note that $q_{n+2}(\ba c_{0}c_{1})=q_{n+3}(\ba c_{0}(c_{1}-1)1).$ Therefore the set of digits that can appear in the last place of a word describing a fundamental interval in $B(p/q,(c_{0},c_{1}))$ coincides with the set of digits that can appear in the last place of a word describing a fundamental interval in $B(p/q,(c_{0},c_{1}-1,1))$ (recall def. of $B(p/q,\ba)$ on pp. 20). Furthermore, $c_{1}\geq 2$ since $c_{1}=b_{1}+1$ and $b_{1}\in \N$, and so we can apply Lemma~\ref{lem: c-1,1 shorter than c} to see by pairwise comparison of the fundamental intervals appearing in $B(p/q,(c_{0},c_{1}))$ and $B(p/q,(c_{0},c_{1}-1,1))$ that
\begin{equation*}
    \cL\left(\cB(p/q,(c_{0},c_{1}))\right)>\cL\left(\cB(p/q,(c_{0},c_{1}-1,1))\right).
\end{equation*}
Therefore $\cL(S(p/q))\geq 20 \cL\left(\cB(p/q,(c_{0},c_{1}-1,1))\right)$. We can then use similar calculations to those appearing in case 2a of Lemma~\ref{lem: bound on S(p/q)} to verify that our collection of words satisfies \eqref{eq: measure outcome with SqD} and \eqref{b digit restrition divisibility statement}.
    \end{enumerate}
    

\item $\cL(S(p/q) \cap I_{1}(c_{0};c_{1}))\geq \tfrac{\cL(S(p/q))}{2}$: Then take the collection of consecutive words 
\begin{equation*}
    \{(c_{0},c_{1},b)\}_{b=M}^{\left\lfloor\frac{1}{3q_{n+2}^{2}\Psi(q_{n+2})}\right\rfloor}\quad \text{ with } \quad M:=\min\{K: I_{2}(c_{0},c_{1},K)\subseteq S(p/q)\}\, .
\end{equation*}
Using the assumption $\cL(S(p/q)\geq  20 \cL(\cB(p/q,(c_{0}, c_{1}))$ we see that if $S(p/q)$ intersects $\cB(p/q,(c_{0},c_{1}))$ the intersection can only remove a small amount. Repeating the argument appearing in case $3c$ of the proof of Lemma~\ref{lem: bound on S(p/q)} we can deduce that the collection of consecutive words satisfies \eqref{eq: measure outcome with SqD} and \eqref{b digit restrition divisibility statement}.\\
\end{enumerate}

\noindent \textit{Case $3^{*}cii$ - $\cL(S(p/q))<  20 \cL(\cB(p/q,(c_{0}, c_{1}))$.}
The argument is similar to that used in case $2^{*}bii$ above. Consider the intersection of $S(p/q)$ with $I_{1}(c_{0};c_{1})$. Since $S(p/q)\cap I_{1}(c_{0};c_{1}-1)\neq \emptyset$ we have that $S(p/q)\cap \cB(p/q, (c_{0},c_{1}))\neq \emptyset$ as $\cB(p/q, (c_{0},c_{1}))$ appears on the right inside $I_{1}(c_{0};c_{1})$. Therefore
\begin{equation*}
    \bigcup_{b=c_{2}+1}^{\left\lfloor \frac{1}{3q_{n+2}^{2}\Psi(q_{n+2})}\right\rfloor}I_{2}(c_{0};c_{1},b) \subset S(p/q)\setminus \cB(p/q,(c_{0},c_{1}))\, ,
\end{equation*}
By Lemma~\ref{lem: far from bad} with $j=1$ odd, and our case $3^{*}cii$ assumption, we have that
\begin{equation*}
    \cL\left( \bigcup_{b=c_{2}+1}^{\left\lfloor \frac{1}{3q_{n+2}^{2}\Psi(q_{n+2})}\right\rfloor}I_{2}(c_{0};c_{1},b)\right) >\frac{1}{640} \cL(\cB(p/q,(c_{0},c_{1}))) >\frac{1}{12800}\cL(S(p/q))\, .
\end{equation*}
Thus \eqref{eq: measure outcome with SqD} and \eqref{b digit restrition divisibility statement} are satisfied by the collection of words $\{(c_{0},c_{1},b)\}_{b=c_{2}+1}^{\left\lfloor \frac{1}{3q_{n+2}^{2}\Psi(q_{n+2})}\right\rfloor}$.\\

\noindent\textit{Case $3^{*}d$ - $S(p/q)\cap \cB(p/q, c_{0})= \emptyset$ and there exists $b_{1}$ such that $S(p/q)\subset I_{1}(c_{0};b_{1})$.} By Lemma~\ref{lem: where is S(p/q)} we know $b_{1}=c_{1}$, which satisfies \eqref{b digit restrition divisibility statement}. Just as in the proof of Lemma~\ref{lem: bound on S(p/q)} consider the following four possible events
\begin{enumerate}[leftmargin=2\parindent]
    \item $S(p/q)\cap \cB(p/q, (c_{0},c_{1})) \neq \emptyset$. We use the argument of either case $3^{*}ci$ or $3^{*}cii$, depending on the size of $S(p/q)$ relative to $\cB(p/q,(c_{0},c_{1}))$, to see the set of words
    \begin{equation*}
        \{(c_{0}, c_{1},b)\}_{b=M}^{\left\lfloor \frac{1}{3q_{n+2}^{2}\Psi(q_{n+2})}\right\rfloor} \quad \text{ with } M:=\min \{K: I_{2}(c_{0};c_{1},K)\subset S(p/q)\}\, ,
    \end{equation*}
    satisfies \eqref{eq: measure outcome with SqD} and \eqref{b digit restrition divisibility statement}.
    \item \textit{$S(p/q)\cap \cB(p/q,(c_{0},c_{1}))=\emptyset$ and there exists $b_{2}$ such that $I_{2}(c_{0};c_{1},b_{2})\subset S(p/q)$.} We use the same argument as case $3^{*}b$ above. That is, take the collection of consecutive words $\{(c_{0},c_{1},b)\}_{b=M_{1}}^{M_{2}} $ where $M_{1}=\min\{K: I_{2}(c_{0};c_{1},K)\subseteq S(p/q)\}$ and $M_{2}=\max\{K: I_{2}(c_{0};c_{1},K)\subseteq S(p/q)\}$
    which satisfies \eqref{eq: measure outcome with SqD} and \eqref{b digit restrition divisibility statement}.
    \item \textit{$S(p/q)\cap \cB(p/q,(c_{0},c_{1}))=\emptyset$, there exists no $b_{2}$ such that $I_{2}(c_{0};c_{1},b_{2})\subset S(p/q)$ and there exists $b_{2}$ such that $I_{2}(c_{0};c_{1},b_{2})\cap S(p/q)\neq \emptyset$ and $I_{2}(c_{0};c_{1},b_{2}+1)\cap S(p/q)\neq \emptyset$.} By Lemma~\ref{lem: where is S(p/q)}$(iii)$ we have $b_{2}=c_{2}$ which satisfies \eqref{b digit restrition divisibility statement}. We then use similar arguments to case $2^{*}$ to obtain a suitable collections of words.  
    \item \textit{$S(p/q)\cap \cB(p/q,(c_{0},c_{1}))=\emptyset$ and there exists $b_{2}$ such that $S(p/q) \subset I_{2}(c_{0};c_{1},b_{2})$.} By Lemma~\ref{lem: where is S(p/q)}$(iii)$ $b_{2}=c_{2}$ and satisfies \eqref{b digit restrition divisibility statement}. If one of the following happens then we can duplicate the arguments as used in cases $3^{*}a-c$:
    \begin{itemize}[leftmargin=\parindent]
        \item[(1')] $S(p/q)\cap \cB(p/q,(c_{0},c_{1},c_{2}))\neq \emptyset$,
        \item[(2')] $S(p/q)\cap \cB(p/q,(c_{0},c_{1},c_{2}))=\emptyset$ and there exists $b_{3}\in \N$ such that $S(p/q)\supset I_{3}(c_{0};c_{1},c_{2},b_{3})$,
         \item[(3')] $S(p/q)\cap \cB(p/q,(c_{0},c_{1},c_{2}))=\emptyset$, there exists no $b_{3}\in \N$ such that $S(p/q)\supset I_{3}(c_{0};c_{1},c_{2},b_{3})$ and there exists $b_{3}\in \N$ such that $S(p/q)\cap I_{3}(c_{0};c_{1},c_{2},b_{3})\neq \emptyset$ and $S(p/q)\cap I_{3}(c_{0};c_{1},c_{2},b_{3}-1)\neq \emptyset$. 
    \end{itemize}
    This covers all cases except the possibility that there exists $b_{3}\in \N$ such that $S(p/q)\subset I_{3}(c_{0};c_{1},c_{2},b_{3}).$ In this case we must have $b_{3}=c_{3}$. This puts us essentially back to the beginning of $3^{*}d(4)$ except now we also know that $b_{3}=c_{3}$ which satisfies \eqref{b digit restrition divisibility statement}, by \eqref{cond: a_i bounded size divisibility statment} and Lemma~\ref{lem: where is S(p/q)}$(iii)$. We can repeat this process, with the prefix word $(c_{0},c_{1},\ldots, c_{i})$ being replaced by $(c_{0},c_{1},\ldots, c_{i},c_{i+1})$ at each step where (1')-(3') do not occur. This step also verifies $c_{i+1}$ satisfies \eqref{b digit restrition divisibility statement}. The process must eventually terminate. To see this assume that for $1\leq i\leq m-1$ we have $S(p/q)\subset I_{i+1}(c_{0};c_{1},\ldots, c_{i},b_{i+1})$ for some $b_{i+1}\in \N$. Consider the case where $i=m-1$ then $b_{m}=c_{m}$ and $S(p/q)\subset I_{m}(c_{0};c_{1},\ldots, c_{m})$. The digits $c_{1},\ldots,c_{m}$ each satisfy \eqref{b digit restrition divisibility statement} by \eqref{cond: a_i bounded size divisibility statment} and Lemma~\ref{lem: where is S(p/q)}$(iii).$ By  Lemma~\ref{lem: where is S(p/q)} and Lemma~\ref{lem:Fundamental interval} we know that the left endpoints of $S(p/q)$ and $I_{m}(c_{0};c_{1},\ldots, c_{m})$ agree. Furthermore, $m$ is even so $\cB(p/q,(c_{0},\ldots, c_{m}))$ appears on the left of $I_{m}(c_{0};c_{1},\ldots, c_{m})$ thus $S(p/q)\cap\cB(p/q,(c_{0},\ldots, c_{m})) \neq \emptyset$ and so we are in case (1'). At this point we use the bound $\cL(S(p/q))\geq 20 \cL(\cB(p/q,(c_{0},\ldots,c_{m})))$ which follows from Lemma~\ref{lem:S bigger than bad 2}. We take our collection of words to be
    \begin{equation*}
        \{(c_{0},c_{1},\ldots, c_{m},b)\}_{b=M}^{\left\lfloor \frac{1}{3q_{n+m+1}^{2}\Psi(q_{n+m+1})}\right\rfloor} \quad \text{ with } \quad M:=\min \{K: I_{m+1}(c_{0};c_{1},\ldots, c_{m},K)\subseteq S(p/q)\}\, .
    \end{equation*}
    This collection satisfies \eqref{b digit restrition divisibility statement} by the choice of $c_{i}$ at each step, and \eqref{eq: measure outcome with SqD} can be shown to hold by similar arguments to those appearing in cases $3^{*}ci(2)$ and $3^{*}d$. Here we use our condition $\cL(S(p/q))\geq 20 \cL(\cB(p/q,\bc))$.


\end{enumerate}

\end{proof}

\begin{proof}[Proof of Lemma~\ref{lem: building U(p/q) with divisibility}]
The proof is analogous to the proof of Lemma~\ref{lem: building U(p/q)} with Lemma~\ref{lem: bound on S(p/q) with divisibility} replacing Lemma~\ref{lem: bound on S(p/q)}. We take our collection of words to be the set $\{\bb_{l}\}_{l=1}^{L}$ provided by Lemma \ref{lem: bound on S(p/q) with divisibility}. Note that Lemma~\ref{lem: belongs to u(p/q)} and Lemma~\ref{lem:b size} are still applicable in the context of Lemma \ref{lem: building U(p/q) with divisibility}. Combining Lemma~\ref{lem: belongs to u(p/q)} with Lemma~\ref{lem: bound on S(p/q) with divisibility} implies that our set of words satisfies \eqref{contained SQD}. \eqref{b digits SQD} follows from \eqref{b digit restrition divisibility statement}, and \eqref{measure SQD} follows from the same calculations as in the proof of Lemma~\ref{lem: building U(p/q)} with \eqref{eq: measure outcome with SqD} replacing \eqref{words contained in S(p/q)}.
\end{proof}


\section{Proof of Theorems~\ref{thm: non-empty}--\ref{thm: non-empty3}} \label{sec: proof non-empty}
We begin with the proof of Theorem~\ref{thm: non-empty}. We will then briefly sketch the differences in the argument required to prove the second part of Theorem~\ref{thm: non-empty2}. Recall that the first part of Theorem~\ref{thm: non-empty2} is an immediate consequence of Theorem \ref{thm:empty}. We conclude this section with a proof of Theorem~\ref{thm: non-empty3}.
\subsection{Proof of Theorem~\ref{thm: non-empty}}
\begin{proof}
Let $\Psi$ and $\omega$ satisfy the assumptions of our theorem. To prove our result we will construct an injective function $f:\{1,2\}^{\N}\to E(\Psi,\omega).$ Let us start by picking $N_{0}\in \N$ to be even and sufficiently large that if $\frac{p_{N_{0}}}{q_{N_{0}}}=[0;1^{N_{0}}]$ then $q_{N_{0}}$ is sufficiently large that Lemma \ref{lem: building U(p/q)} can be applied. To simplify our notation we let $\ba=(1^{N_{0}}).$ Let us now pick $(\delta_{i})_{i=1}^{\infty}\in \{1,2\}^{\N}$ and set out to define $f((\delta_i))$. We will construct $f$ via an inductive argument.\\

\noindent \textbf{Step 1.} Applying Lemma \ref{lem: building U(p/q)} to $p_{N_{0}}/q_{N_{0}}$ we can choose a finite word $\bb_{1}=(b_{1,0},\ldots,b_{1,m_{1}-1})$ in a way that depends only upon $p_{N_{0}}/q_{N_{0}},$ $\Psi$ and $\omega$ such that $$I(0;\ba\bb_{1})\subset \left\{x:\Psi(q_{N_{0}})(1-\omega(q_{N_{0}}))\leq x-\frac{p_{N_{0}}}{q_{N_{0}}}\leq  \Psi(q_{N_{0}})\right\},$$ and if $m_{1}\geq 2$ then 
$$b_{1,i}\in \left\{ 1,2, \ldots,\left\lfloor \frac{1}{3q_{N_{0}+i}^{2}\Psi(q_{N_{0}+i})}\right\rfloor \right\} \quad 1\leq i \leq m_{1}-1,$$
where $\frac{p_{N_{0}+i}}{q_{N_{0}+i}}=[0;\ba,b_{1,0},b_{1,1},\ldots,b_{1,i-1}]$ for $1\leq i\leq m_{1}-1.$ Let us now also pick $l_{1}\in \{1,2\}$ such that $N_{0}+m_{1}+l_{1}$ is even. \\

\noindent \textbf{Inductive step.} Let $k\geq 1$ and suppose that we have constructed $k$ words $\bb_{1},\ldots, \bb_{k}$ of length $m_{1},\ldots, m_{k}$ respectively and $l_{1},\ldots,l_{k}\in \N$ such that the following properties hold:
\begin{enumerate}[leftmargin=2\parindent]
    \item For each $1\leq j\leq k$ the sum $N_{j}:=N_{0}+\sum_{p=1}^{j}(m_{p}+l_{p})$ is even.
    \item For each $1\leq j\leq k$ we have $$I(0;\ba \bb_{1}\delta_{1}^{l_{1}}\bb_{2}\delta_{2}^{l_{2}}\ldots \delta_{j-1}^{l_{j-1}}\bb_{j})\subset \left\{x:\Psi(q_{N_{j-1}})(1-\omega(q_{N_{j-1}}))\leq x-\frac{p_{N_{j-1}}}{q_{N_{j-1}}}\leq  \Psi(q_{N_{j-1}})\right\}.$$
\item For each $1\leq j\leq k$ if $m_{j}\geq 2$ then
$$b_{j,i}\in \left\{ 1,2, \ldots,\left\lfloor \frac{1}{3q_{N_{j-1}+i}^{2}\Psi(q_{N_{j-1}+i})}\right\rfloor \right\} \quad 1\leq i \leq m_{j}-1,$$
where $\frac{p_{N_{j-1}+i}}{q_{N_{j-1}+i}}=[0;\ba \bb_{1}\delta_{1}^{l_{1}}\bb_{2}\delta_{2}^{l_{2}}\ldots,\bb_{j-1}\delta_{j-1}^{l_{j-1}}b_{j,0},b_{j,1},\ldots,b_{j,i-1}].$ 
\end{enumerate}
We now apply Lemma \ref{lem: building U(p/q)} to $\frac{p_{N_{k}}}{q_{N_{k}}}=[0;\ba\bb_{1}\delta_{1}^{l_{1}}\ldots \bb_{k}\delta_{k}^{l_{k}}]$ and choose a word $\bb_{k+1}=(b_{k+1,0},\ldots, b_{k+1,m_{k+1}-1})$ in a way depending only upon $\frac{p_{N_{k}}}{q_{N_{k}}},$ $\Psi$ and $\omega$ so that
$$I(0;\ba \bb_{1}\delta_{1}^{l_{1}}\bb_{2}\delta_{2}^{l_{2}}\ldots \bb_{k}\delta_{k}^{l_{k}}\bb_{k+1})\subset \left\{x:\Psi(q_{N_{k}})(1-\omega(q_{N_{k}}))\leq x-\frac{p_{N_{k}}}{q_{N_{k}}}\leq \Psi(q_{N_{k}})\right\},$$ and if $m_{k+1}\geq 2$ then $$b_{k+1,i}\in \left\{ 1,2, \ldots,\left\lfloor \frac{1}{3q_{N_{k}+i}^{2}\Psi(q_{N_{k}+i})}\right\rfloor \right\} \quad 1\leq i \leq m_{k+1}-1,$$
where $\frac{p_{N_{k}+i}}{q_{N_{k}+i}}=[0;\ba \bb_{1}\delta_{1}^{l_{1}}\bb_{2}\delta_{2}^{l_{2}}\ldots,\bb_{k}\delta_{k}^{l_{k}}b_{k+1,0},b_{k+1,1},\ldots,b_{k+1,i-1}].$ We can apply Lemma \ref{lem: building U(p/q)} because $N_{k}$ is even. We now choose $l_{k+1}\in \{1,2\}$ so that $N_{k+1}=N_{0}+\sum_{p=1}^{k+1}(m_{p}+l_{p})$ is even.  We observe that our new collection of words $\bb_{1},\ldots, \bb_{k+1}$ with lengths $ m_{1},\ldots, m_{k+1}$ and the collection of integers $l_{1},\ldots,l_{k+1}$ satisfy properties $(1)-(3)$ listed above with $k$ replaced with $k+1.$ Thus we can repeat this step indefinitely. 

Repeatedly applying our inductive step yields an infinite sequence of words $(\bb_{j})_{j=1}^{\infty}$ with corresponding lengths $(m_{j})_{j=1}^{\infty}$ and the sequence of integers $(l_{j})_{j=1}^{\infty}.$ We define $f((\delta_{i}))$ to be 
$$f((\delta_{i}))=\bigcap_{j=1}^{\infty}I_{N_{j}}(0;\ba\bb_{1}\delta_{1}^{l_{1}}\ldots \bb_{j}\delta_{j}^{l_{j}}),$$ or equivalently
$f((\delta_i))$ is the irrational number satisfying $f((\delta_{i}))=[0;\ba \bb_{1}\delta_{1}^{l_{1}}\bb_{2}\delta_{2}^{l_{2}}\ldots].$ It remains to show that $f$ is injective and maps into $E(\Psi,\omega).$\\

\noindent \textbf{$f$ is injective.} Let $(\delta_{i}),(\delta_{i}')\in \{1,2\}^{\N}$ be distinct sequences and let $j\in \N$ be the minimal integer such that $\delta_{j}\neq \delta_{j}'.$ Without loss of generality we assume $\delta_{j}=1$ and $\delta_{j}'=2$. Then it follows from our construction of $f$ that there exists $\bb_{1}\ldots,\bb_{j}$ and $l_{1},\ldots, l_{j}$ such that
$$f((\delta_i))\in I(0;\ba \bb_{1}\delta_{1}^{l_{1}},\ldots, \bb_{j}, 1^{l_{j}})$$ and 
$$f((\delta_i'))\in I(0;\ba \bb_{1}\delta_{1}^{l_{1}},\ldots, \bb_{j}, 2^{l_{j}}).$$ Since the continued fraction expansion of an irrational number is unique, it follows that $f((\delta_i))\neq f((\delta_i'))$ and the injectivity of $f$ follows.\\

\noindent \textbf{$f$ maps into $E(\Psi,\omega)$.} 
 By construction, we know that $$0\leq f((\delta_i))-\frac{p_{N_{j}}}{q_{N_{j}}}\leq \Psi(q_{N_{j}})$$ for all $j.$ It remains to show that for any $(p,q)\in \Z\times \N$ with $q$ sufficiently large we have
 \begin{equation}
 \label{eq:uncountable WTS}
 \Psi(q)(1-\omega(q))\leq \left|x-\frac{p}{q}\right|.
 \end{equation}
 Since $\lim_{q\to\infty}q^{2}\Psi(q)=0,$ by Lemma \ref{lem:Legendre}, to prove \eqref{eq:uncountable WTS} holds for $q$ sufficiently large, it suffices to show that 
\begin{equation}
 \label{eq:uncountable WTS2}
\Psi(q_{n})(1-\omega(q_{n}))\leq \left|x-\frac{p_{n}}{q_{n}}\right|.
\end{equation}for all $n$ sufficiently large. If $n=N_{j}$ for some $j$ then $$\Psi(q_{N_{j}})(1-\omega(q_{N_{j}}))\leq x-\frac{p_{N_{j}}}{q_{N_{j}}} $$ by our construction. So \eqref{eq:uncountable WTS2} holds when $n=N_{j}.$ Suppose $n\notin \{N_{j}\}$. We let $(a_l)$ denote the sequence of partial quotients for $f((\delta_i))$. Then by our construction either $a_{n+1}\in \{1,2\}$ in which case 
 \begin{equation}
     \label{eq:Almost finishedA}
     \left|x-\frac{p_{n}}{q_{n}}\right |\geq \frac{1}{6q_{n}^{2}} 
 \end{equation}
  by Lemma \ref{lem:approximation quality}, or $$a_{n+1}\in \left\{1,2,\ldots ,\left \lfloor \frac{1}{3q_{n}^{2}\Psi(q_{n})}\right \rfloor \right\}$$ where $\frac{p_{n}}{q_{n}}=[0;a_{1},\ldots,a_{n}]$ and 
  \begin{equation}
      \label{eq:Almost finishedB}
      \left|x-\frac{p_{n}}{q_{n}}\right |\geq \Psi(q)
  \end{equation}
  holds by Lemma \ref{lem:approximation quality}. Since $\lim_{q\to\infty}q^{2}\Psi(q)=0$, it follows from \eqref{eq:Almost finishedA} and \eqref{eq:Almost finishedB} that \eqref{eq:uncountable WTS2} holds for all $n$ sufficiently large such that $n\notin \{N_{j}\}$. In summary, we have shown that \eqref{eq:uncountable WTS2} holds for all $n$ sufficiently large. Therefore $f$ maps into $E(\Psi,\omega)$ and our proof is complete.
\end{proof}


\subsection{ Proof of Theorem~\ref{thm: non-empty2}}
Recall that the first part of Theorem \ref{thm: non-empty2} is a special case of Theorem~\ref{thm:empty}. The second statement is proved below. The argument presented is similar in spirit to the proof of Theorem~\ref{thm: non-empty}. The main difference being that we apply Lemma \ref{lem: non-empty S(p/q) divisibility cond} in place of Lemma \ref{lem: building U(p/q)}.
 \begin{proof}[Proof of the second statement of Theorem~\ref{thm: non-empty2}.]

Let $\Psi$ and $\omega$ satisfy the assumptions of the second statement of Theorem \ref{thm: non-empty2}. To prove our result we will construct an injective function $f:\{2,3\}^{\N}\to E(\Psi,\omega)$\footnote{The significance of taking $\{2,3\}^{\N}$ instead of $\{1,2\}^{\N},$ as in the proof of Theorem \ref{thm: non-empty}, is to ensure we satisfy the $a_{1},a_{n}\geq 2$ assumption appearing in Lemma \ref{lem: non-empty S(p/q) divisibility cond}.}. Let us start by picking $N_{0}\in \N$ to be even and sufficiently large that if $\frac{p_{N_{0}}}{q_{N_{0}}}=[0;2^{N_{0}}]$ then $q_{N_{0}}$ is sufficiently large that Lemma \ref{lem: non-empty S(p/q) divisibility cond} can be applied. We let $\ba_{0}=(2^{N_{0}}).$ Let us now pick $(\delta_{i})_{i=1}^{\infty}\in \{2,3\}^{\N}$ and set out to define $f((\delta_i))$. We will construct $f((\delta_i))$ via an inductive argument.\\

\noindent \textbf{Step 1.} We apply Lemma \ref{lem: non-empty S(p/q) divisibility cond} to $p_{N_{0}}/q_{N_{0}}.$ Therefore 
\begin{align*}
&I\left(0;\ba_{0}, \frac{1}{q_{N_{0}}^{2}\Psi(q_{N_{0}})}-1,1,1,2^{ N_{0}-2},1,1 ,\left\lfloor \frac{1}{3q_{*,1}^{2}\Psi(q_{*,1})}\right\rfloor\right) \\
& \hspace{2cm}\subset  \left\{x:\Psi(q_{N_{0}})(1-\omega(q_{N_{0}}))\leq x-\frac{p_{N_{0}}}{q_{N_{0}}}\leq \Psi(q_{N_{0}})\right\}
\end{align*}
where $$\frac{p_{*,1}}{q_{*,1}}=\left[0;\ba_{0}, \frac{1}{q_{N_{0}}^{2}\Psi(q_{N_{0}})}-1,1,1,2^{ N_{0}-2},1,1 \right].$$ We may assume that we chose $N_{0}$ to be sufficiently large that 
\begin{equation}
\label{eq:positive integers}
\frac{1}{q_{N_{0}}^{2}\Psi(q_{N_{0}})}-1\geq 1\quad \text{ and }\quad \left\lfloor \frac{1}{3q_{*,0}^{2}\Psi(q_{*,0})}\right\rfloor\geq 1.
\end{equation}
We let 
$$\ba_{1}=\left(\ba_{0}, \frac{1}{q_{N_{0}}^{2}\Psi(q_{N_{0}})}-1,1,1,2^{N_{0}-2},1,1 ,\left\lfloor \frac{1}{3q_{*,0}^{2}\Psi(q_{*,0})}\right\rfloor, \delta_{1}^{l_{1}}\right),$$
$N_{1}$ denote the length of $\ba_{1},$ and $$\frac{p_{N_{1}}}{q_{N_{1}}}=[0;\ba_{1}].$$
Here $l_{1}\in \N$ is chosen so that $N_{1}$ is even and so that for all $q\geq q_{N_{1}}$ we have $$\max_{1\leq i\leq N_{1}}a_{1,i}\leq \left \lfloor \frac{1}{3q^{2}\Psi(q)}\right \rfloor$$ where $\ba_{1}=(a_{1,1},\ldots a_{1,N_{1}}).$ This is permissible as $\lim_{q\to\infty}q^{2}\Psi(q)=0$ and $\delta_{1}\in \{2,3\}$.  We remark that the first digit and last digit of $\ba_{1}$ are greater than or equal to two by construction and all other digits in $\ba_{1}$ are strictly positive integers by \eqref{eq:positive integers}. \\

\noindent \textbf{Inductive step.} Let $k\geq 1$ and suppose that we have constructed $k+1$ words $\{\ba_{j}=(a_{j,1},\ldots a_{j,N_{j}})\}_{j=0}^{k}$ such that $N_{j}$ is even for each $0\leq j\leq k$ and  $l_{1},\ldots,l_{k}\in \N.$ For each $0\leq j\leq k$ let $\frac{p_{N_{j}}}{q_{N_{j}}}=[0;\ba_{j}].$ Suppose that the following properties are satisfied:
\begin{enumerate}[leftmargin=2\parindent]
    \item For each $0\leq j< k$ we have $$\ba_{j+1}=\left(\ba_{j},\frac{1}{q_{N_{j}}^{2}\Psi(q_{N_{j}})}-1,1,a_{j,N_{j}}-1,a_{j,N_{j}-1}\ldots, a_{j,2},a_{j,1}-1,1,\left\lfloor \frac{1}{3q_{*,j}^{2}\Psi(q_{*,j+1})}\right\rfloor, \delta_{j}^{l_{j+1}}\right)$$where  $$\frac{p_{*,j+1}}{q_{*,j+1}}=\left[0;\ba_{j},\frac{1}{q_{N_{j}}^{2}\Psi(q_{N_{j}})}-1,1,a_{j,N_{j}}-1,a_{j,N_{j}-1}\ldots, a_{j,2},a_{j,1}-1,1\right]$$
    \item For each $0\leq j< k$ we have
    $$I(0;\ba_{j+1})\subset \left\{x:\Psi(q_{N_{j}})(1-\omega(q_{N_{j}}))\leq x-\frac{p_{N_{j}}}{q_{N_{j}}}\leq \Psi(q_{N_{j}})\right\}.$$
    \item For each $1\leq j\leq k,$ for any $q\geq q_{N_{j}}$ we have  
    $$\max_{1\leq i\leq N_{j}}a_{j,i}\leq \left \lfloor \frac{1}{3q^{2}\Psi(q)}\right \rfloor$$
    \item For each $0\leq j\leq k$ the first and last digit of $\ba_{j}$ is greater than or equal to two, and all other digits are strictly positive integers.
\end{enumerate}
We now show how to define $\ba_{k+1}$ and $l_{k+1}$ so that the above properties are satisfied with $k+1$ in place of $k$. We apply Lemma \ref{lem: non-empty S(p/q) divisibility cond} to $\frac{p_{N_{k}}}{q_{N_{k}}}$ so that
\begin{align*}
&I\left(0;\ba_{k},\frac{1}{q_{N_{k}}^{2}\Psi(q_{N_{k}})}-1,1,a_{k,N_{k}}-1,a_{k,N_{k}-1}\ldots, a_{k,2},a_{k,1}-1,1,\left\lfloor \frac{1}{3q_{*,k+1}^{2}\Psi(q_{*,k+1})}\right\rfloor\right)\\
\subset &\left\{x:\Psi(q_{N_{k}})(1-\omega(q_{N_{j}}))\leq x-\frac{p_{N_{k}}}{q_{N_{k}}}\leq \Psi(q_{N_{k}})\right\}
\end{align*}
where 
$\frac{p_{*,k+1}}{q_{*,k+1}}=\left[0;\ba_{k},\frac{1}{q_{N_{k}}^{2}\Psi(q_{N_{k}})}-1,1,a_{k,N_{k}}-1,a_{k,N_{k}-1}\ldots, a_{k,2},a_{k,1}-1,1\right]$. Let $\ba_{k+1}$ be given by $$\left(\ba_{k},\frac{1}{q_{N_{k}}^{2}\Psi(q_{N_{k}})}-1,1,a_{k,N_{k}}-1,a_{k,N_{k}-1}\ldots, a_{k,2},a_{k,1}-1,1,\left\lfloor \frac{1}{3q_{*,k+1}^{2}\Psi(q_{*,k+1})}\right\rfloor,\delta_{k+1}^{l_{k+1}}\right)$$ where $l_{k+1}$ has been chosen to be sufficiently large that $$\max_{1\leq i\leq N_{k+1}}a_{k+1,i}\leq \left \lfloor \frac{1}{3q^{2}\Psi(q)}\right \rfloor$$ for any $q\geq q_{N_{k+1}}$ where $\frac{p_{N_{k+1}}}{q_{N_{k+1}}}=[0;\ba_{k+1}]$ and $\ba_{k+1}=(a_{k+1,1},\ldots, a_{k+1,N_{k+1}}).$ This is permissible as $\lim_{q\to\infty}q^{2}\Psi(q)=0$ and $\delta_{k+1}\in \{2,3\}.$ It is clear that for this choice of $\ba_{k+1}$ and $l_{k+1}$ properties $(1), (2)$ and $(3)$ all hold with $k$ replaced by $k+1$. To see that the fourth property holds, observe that the first digit of $\ba_{k+1}$ is the first digit of $\ba_{k},$ so by assumption it is greater than or equal to $2$. The last digit $\delta_{k+1}$ belongs to $\{2,3\},$ so the last digit also has the desired property. It remains to check that all other digits in $\ba_{k+1}$ are strictly positive integers. If a digit belongs to $\{1,2,3\}$ then this is obviously true. If it equals one of $\ba_{k}$'s digits, then by our assumption it is a strictly positive integer. The only digits that do not necessarily fall in to these categories are $$\frac{1}{q_{N_{k}}^{2}\Psi(q_{N_{k}})}-1,\quad  a_{k,N_{k}}-1\quad \text{ and } \quad \left \lfloor \frac{1}{3q_{*,k+1}^{2}\Psi(q_{*,k+1})}\right \rfloor.$$ Since $\lim_{q\to\infty}q^{2}\Psi(q)=0$ we can assume that our initial choice of $N_{0}$ was sufficiently large that $$\frac{1}{q_{N_{k}}^{2}\Psi(q_{N_{k}})}-1\geq 1 \quad \text{ and }\quad \left \lfloor \frac{1}{3q_{*,k}^{2}\Psi(q_{*,k})}\right \rfloor\geq 1.$$ The remaining digit, $a_{k,N_{k}}-1$, is the last term of $\ba_{k}$ with one subtracted. By our assumption the last term of $\ba_{k}$ is at least two Therefore $a_{k,N_{k}}-1\geq 1.$ We may now conclude that property (4) is satisfied with $k$ replaced by $k+1$.\par
\vspace{1.2ex}
Repeatedly applying our inductive step yields an infinite sequence of words $(\ba_{j})_{j=0}^{\infty}$ and an infinite sequence of integers $(l_{j})_{j=1}^{\infty}$ such that properties $(1)-(4)$ hold for any $k\geq 1$. We define $f((\delta_{i}))$ to be $$f((\delta_i))=\bigcap_{j=0}^{\infty}I(0;\ba_{j}).$$ It follows from property $(1)$ that the intervals appearing in this intersection are nested and consequently the intersection is non-empty. Moreover, by our construction this intersection must be a single point. Equivalently, $f((\delta_{i}))$ is the point whose sequence of partial quotients is equal to the componentwise limit of the sequence $(0\ba_{k} 0^{\infty})_{k=0}^{\infty}$. Here we are using property (4) which guarantees that for each $k$ each digit in $\ba_{k}$ is a strictly positive integer. This guarantees that $(0\ba_{k} 0^{\infty})_{k=0}^{\infty}$ will converge to the continued fraction expansion of some real number.

It remains to show that $f$ is injective and maps into $E(\Psi,\omega).$ Injectivity follows by an analogous argument to that given in the proof of Theorem \ref{thm: non-empty}. As such we omit the argument and focus on proving that $f$ maps into $E(\Psi,\omega)$.\par
\vspace{1ex}
\noindent \textbf{$f$ maps into $E(\Psi,\omega).$}
By property $(2)$ we know that $$\Psi(q_{N_{j}})(1-\omega(q_{N_{j}})\leq f((\delta_{i}))-\frac{p_{N_{j}}}{q_{N_{j}}}\leq \Psi(q_{N_{j}})$$ for all $j\in \N$. Therefore, to show that $f((\delta_{i}))\in E(\Psi,\omega),$ it remains to show that there are at most finitely many $(p,q)\in \Z\times \N$ satisfying 
\begin{equation}
    \label{eq:Nearly finished WTS}
    \left|x-\frac{p}{q}\right|<\Psi(q)(1-\omega(q)).
\end{equation}
Since $\lim_{q\to \infty}q^{2}\Psi(q)=0$ we know that $\Psi(q)\leq \frac{1}{2q^{2}}$ for all $q$ sufficiently large. Therefore, by Lemma \ref{lem:Legendre}, to show that \eqref{eq:Nearly finished WTS} has finitely many solutions it suffices to show that for all $n$ sufficiently large we have
\begin{equation}
    \label{eq:Nearly finished WTSA}
    \left|x-\frac{p_{n}}{q_{n}}\right|\geq \Psi(q_{n})(1-\omega(q_{n})). 
\end{equation}  By property $(2)$ \eqref{eq:Nearly finished WTSA} is satisfied whenever $n\in \{N_{j}\}_{j=0}^{\infty}.$ It remains to consider the remaining case when $n\notin \{N_{j}\}_{j=0}^{\infty}.$ Let us fix such an $n$ and let $k$ be such that $N_{k}<n<N_{k+1}.$ Upon inspection three cases naturally arise from our construction. Either $a_{n+1}\in \{1,2,3\}$, $a_{n+1}\in \{a_{k,1}-1,a_{k,2},\ldots, a_{k,N_{k}-1},a_{k,N_{k}}-1\}$ or $q_{n}=q_{*,k+1}$. We treat each case below individually.\par
\vspace{1ex}
\noindent \textbf{Case 1 - $a_{n+1}\in \{1,2,3\}$.} In this case it follows from Lemma \ref{lem:approximation quality} that $$\left|x-\frac{p_{n}}{q_{n}}\right |\geq \frac{1}{9q_{n}^{2}}.$$ Since $\lim_{q\to\infty}q^{2}\Psi(q)=0$ there are at most finitely many values of $n$ where $a_{n+1}\in \{1,2,3\}$ satisfying \eqref{eq:Nearly finished WTSA}.\par
\vspace{1ex}
\noindent \textbf{Case 2 - $a_{n+1}\in \{a_{k,1}-1,a_{k,2},\ldots, a_{k,N_{k}-1},a_{k,N_{k}}-1\}.$} In this case we have the following bound for $a_{n+1}$ that follows from property (3): $$a_{n+1}\leq \left\lfloor \frac{1}{3q_{n}^{2}\Psi(q_{n})}\right\rfloor.$$ It now follows from an application of Lemma \ref{lem:approximation quality} that such an $n$ cannot contribute any solutions to \eqref{eq:Nearly finished WTSA}.\par
\vspace{1ex}
\noindent \textbf{Case 3 - $q_{n}=q_{*,k+1}.$} In this case, we know by our construction that $$a_{n+1}=\left\lfloor \frac{1}{3q_{n}^{2}\Psi(q_{n})}\right\rfloor.$$ Again by an application of Lemma \ref{lem:approximation quality} no such $n$ can contribute any solutions to \eqref{eq:Nearly finished WTSA}.\par
\vspace{1.2ex}
In summary, we have shown that there are finitely many values of $n\notin \{N_{j}\}_{j=0}^{\infty}$ that satisfy \eqref{eq:Nearly finished WTSA}. Combined with our earlier observation, we may conclude that \eqref{eq:Nearly finished WTSA} has finitely many solutions and $f((\delta_i))\in E(\Psi,\omega)$. This completes our proof.
\end{proof}

\subsection{Proof of Theorem~\ref{thm: non-empty3}}

\begin{proof}[Proof of Theorem \ref{thm: non-empty3}]
Let $\omega:(0,\infty)\to(0,\infty)$ and $\tau>2$ be arbitrary. Moreover let $\psi:(0,\infty)\to (0,\infty)$ be given by $\psi(q)=\frac{1}{q^{2}\left\lceil q^{\tau-2}\right\rceil}$. Clearly $\psi$ satisfies the square divisibility condition and $\lim_{q\to\infty}\frac{-\log\psi(q)}{\log q}=\tau$. We will construct the desired $\Psi$ by perturbing $\psi$. These perturbations will ensure $E(\Psi,\omega)\neq \emptyset$ but will be sufficiently small to leave the lower limit at infinity unchanged. \par 
Let $N_{0}\in\N$ and set $\frac{p_{N_{0}}}{q_{N_{0}}}=[0;1^{N_{0}}].$ We can choose $N_{0}$ sufficiently large so that 
\begin{equation}
\label{eq:integer first digits}
q^{2}\psi(q)\leq \tfrac{1}{2}
\end{equation} for all $q\geq q_{N_{0}}$. Define the function $\Psi$ at $q_{N_{0}}$ so that
\begin{equation*}
    \frac{1}{q_{N_{0}}^{2}\Psi(q_{N_{0}})}-\frac{q_{N_{0}-1}}{q_{N_{0}}}:=\left[\frac{1}{q_{N_{0}}^{2}\psi(q_{N_{0}})}-1;1,1,1,\ldots  \right]\, .
\end{equation*}
That is, writing $\varphi=[0;2,1,1\ldots]$, we set
\begin{equation}
\label{eq:first formula for Psi}
    \Psi(q_{N_{0}})= \frac{1}{q_{N_{0}}^{2}\left(\frac{1}{q_{N_{0}}^{2}\psi(q_{N_{0}})}-\varphi + \frac{q_{N_{0}-1}}{q_{N_{0}}}\right)}\, .
\end{equation}
Here we are using the formula $-[0;b_{1},b_{2},\ldots,]=[-1;1,b_{1}-1,b_{2},b_{3},\ldots ].$ Which holds for any sequence of positive integers $(b_i)_{i=1}^{\infty}.$ For a proof see \cite{PoorShall92}. It is a consequence of \eqref{eq:integer first digits} and the square divisibility condition that  $$\frac{1}{q_{N_{0}}^{2}\psi(q_{N_{0}})}-1\geq 1$$ and is an integer. It follows from \eqref{eq:first formula for Psi} that if $N_{0}$ is chosen sufficiently large then
\begin{equation}
\label{eq:first double bound}
    \frac{1}{2}\psi(q_{N_{0}})\leq \Psi(q_{N_{0}}) \leq 2\psi(q_{N_{0}})\, .
\end{equation}
We now want to pick a fundamental interval contained inside $S(\Psi,\omega, p_{N_{0}}/q_{N_{0}}).$
We write it in this format to emphasis $S(\Psi,\omega, p_{N_{0}}/q_{N_{0}})$ is the set appearing in Proposition~\ref{prop: equivalence}\ref{S(p/q) set} for $\Psi$, not $\psi$. 
Choose $R_{1}\in \N$ to be an even number and large enough (depending on $\omega,$ $\psi$ and $N_{0}$) so that
\begin{equation}
\label{eq:double measure bound}
 \sum_{i=1}^{2} \cL\left(I_{R_{1}}\left(\frac{1}{q_{N_{0}}^{2}\psi(q_{N_{0}})}-1;1^{R_{1}-1},i\right)\right)  <\frac{\omega(q_{N_{0}})}{q_{N_{0}}^{2}\Psi(q_{N_{0}})}\, .
\end{equation}
Equation \eqref{eq:double measure bound} must eventually hold since the terms on the left tend to zero as $R_{1}$ increases and the term on the right is positive and does not depend upon $R_{1}$. We remark that the left endpoint of $S(\Psi,\omega, p_{N_{0}}/q_{N_{0}})$ is contained in $$I_{R_{1}}\left(\frac{1}{q_{N_{0}}^{2}\psi(q_{N_{0}})}-1;1^{R_{1}}\right),$$ and to the immediate right of this fundamental interval we have $$I_{R_{1}}\left(\frac{1}{q_{N_{0}}^{2}\psi(q_{N_{0}})}-1;1^{R_{1}-1},2\right)$$ since $R_{1}$ is even.
It follows from these observations, \eqref{eq:double measure bound} and the lower bound $$\mathcal{L}(S(\Psi,\omega, p_{N_{0}}/q_{N_{0}}))>\frac{\omega(q_{N_{0}})}{q_{N_{0}}^{2}\Psi(q_{N_{0}})}$$  that 
\begin{equation}
\label{eq:inclusion1}
    I_{R_{1}}\left(\frac{1}{q_{N_{0}}^{2}\psi(q_{N_{0}})}-1;1,\ldots,1,2\right) \subset S(\Psi,\omega,p_{N_{0}}/q_{N_{0}})\, .
\end{equation}
Define
\begin{equation*}
    \bc_{1}:=\left( 1^{N_{0}}, \frac{1}{q_{N_{0}}^{2}\psi(q_{N_{0}})}-1, 1^{R_{1}-1},2\, \right)
\end{equation*}
Set $N_{1}=N_{0}+R_{1}+1$, define $\frac{p_{N_{1}}}{q_{N_{1}}}:=\left[ 0;\bc_{1}\right]$, and choose $R_{1}$ large enough so that 
\begin{equation}
\label{eq:guarantee monotoncity equation1}
2\psi(q_{N_{1}})<\frac{1}{2}\psi(q_{N_{0}}).
\end{equation} We now repeat the above step with  $\frac{p_{N_{0}}}{q_{N_{0}}}$ replaced by $\frac{p_{N_{1}}}{q_{N_{1}}}.$ That is, define $\Psi$ at $q_{N_{1}}$ by
\begin{equation*}
    \Psi(q_{N_{1}}):=\frac{1}{q_{N_{1}}^{2}\left(\frac{1}{q_{N_{1}}^{2}\psi(q_{N_{1}})}-\varphi+\frac{q_{N_{1}-1}}{q_{N_{1}}}\right)}\, .
\end{equation*}
Note that the inequalities
\begin{equation}
\label{eq:second double bound}
    \frac{1}{2}\psi(q_{N_{1}})\leq \Psi(q_{N_{1}})\leq 2\psi(q_{N_{1}})\, 
\end{equation}
can be guaranteed by our initial choice of $N_{0}$. Again, we can choose some $R_{2}\in \N$ a large even integer so that 
\begin{equation}
\label{eq:inclusion2}
    I_{R_{2}}\left(\frac{1}{q_{N_{1}}^{2}\psi(q_{N_{1}})}-1;1^{R_{2}-1},2\right) \subset S(\Psi,\omega,p_{N_{1}}/q_{N_{1}})\, .
\end{equation}
It is a consequence of \eqref{eq:integer first digits} and that $\psi$ satisfies the square divisibility condition that $$\frac{1}{q_{N_{1}}^{2}\psi(q_{N_{1}})}-1\geq 1$$ and the left hand side of this inequality is an integer. Define
\begin{equation*}
    \bc_{2}:=\left(\bc_{1}, \frac{1}{q_{N_{1}}^{2}\psi(q_{N_{1}})}-1, 1^{R_{2}-1} ,2\, \right).
\end{equation*}
Set $N_{2}=N_{1}+R_{2}+1$ and $\frac{p_{N_{2}}}{q_{N_{2}}}=[0;\bc_{2}]$. Again we may assume that $R_{2}$ is chosen sufficiently large so that 
\begin{equation}
    \label{eq:guarantee monotoncity equation2}
   2\psi(q_{N_{2}}) <\frac{1}{2}\psi(q_{N_{1}}).
\end{equation}
 Repeating the above arguments indefinitely yields two sequences of integers $(N_{i})_{i=0}^{\infty}$, $(R_{i})_{i=1}^{\infty},$ a sequence of rationals $(p_{N_{i}}/q_{N_{i}})_{i=0}^{\infty}$, and a sequence of words $(\bc_{i})_{i=1}^{\infty}$ whose digits are strictly positive integers. Moreover $\frac{p_{N_{i}}}{q_{N_{i}}}=[0;\bc_{i}]$ and $\bc_{i}$ has length $N_{i}$ for all $i\geq 1.$ We define $\Psi$ at $q_{N_{i}}$ according to the rule
 \begin{equation*}
    \Psi(q_{N_{i}}):=\frac{1}{q_{N_{i}}^{2}\left(\frac{1}{q_{N_{i}}^{2}\psi(q_{N_{i}})}-\varphi+\frac{q_{N_{i}-1}}{q_{N_{i}}}\right)}\,.
 \end{equation*}
The following properties are a consequence of our construction:
 \begin{enumerate}[leftmargin=2\parindent] 
 \item For all $i\geq 1$ we have
 $$\bc_{i+1}=\left(\bc_{i},\frac{1}{q_{N_{i}}^{2}\Psi(q_{N_{i}})}-1,1^{R_{i+1}-1},2\right).$$ 
  \item For all $i\geq 0$ we have $$ I_{R_{i+1}}\left(\frac{1}{q_{N_{i}}^{2}\psi(q_{N_{i}})}-1;1^{R_{i+1}-1},2\right) \subset S(\Psi,\omega,p_{N_{i}}/q_{N_{i}}).$$
     \item Equations \eqref{eq:first double bound} and \eqref{eq:second double bound} hold for all $i$, i.e. for all $i\geq 0$ we have
     $$\frac{1}{2}\psi(q_{N_{i}})\leq \Psi(q_{N_{i}})\leq 2\psi(q_{N_{i}})$$
     \item Equations \eqref{eq:guarantee monotoncity equation1} and \eqref{eq:guarantee monotoncity equation2} holds for all $i\geq 0$, i.e. for all $i\geq 0$ we have
     $$2\psi(q_{N_{i+1}}) <\frac{1}{2}\psi(q_{N_{i}}).$$
 \end{enumerate}
By property (1) we know that  
$$x=\lim_{i\to\infty} [0;\bc_{i}]:=[0;a_{1}(x),a_{2}(x),\ldots]$$ 
is well defined and its sequence of partial quotients is the component wise limit of the sequence $(0\bc_{i}0^{\infty})_{i=1}^{\infty}$. 
Moreover this $x$ has the following properties:
\begin{enumerate}[leftmargin=2\parindent]
    \item[(A)] for every $N_{i}$ we have $$[a_{N_{i}+1}(x);a_{N_{i}+2}(x),\ldots] \in S(\Psi,\omega, p_{N_{i}}(x)/q_{N_{i}}(x)),$$ 
    \item[(B)] for every $j\notin \{N_{i}+1\}_{i\in \N\cup\{0\}}$ we have $a_{j}(x)\leq 2$.
\end{enumerate}
Property $(A)$ is a consequence of Property $(2)$ above. Property $(B)$ is a consequence of the formula stated in property $(1)$.

Define $\Psi$ to be the non-increasing function given by
\begin{equation*}
    \Psi(q):=\begin{cases} \Psi(q_{N_{i}}) \quad \, &q=q_{N_{i}} \\
    \min\{\Psi(q_{N_{i}}), \max\{\psi(q),\Psi(q_{N_{i+1}})\}\} &q_{N_{i}}< q < q_{N_{i+1}}\, . 
    \end{cases}
\end{equation*}
It is a consequence of properties $(3)$ and $(4)$ that $\Psi$ is non-increasing and $\lim_{q\to\infty}\frac{-\log \Psi(q)}{\log q}=\tau$. It follows from the equality $\lim_{q\to\infty}\frac{-\log \Psi(q)}{\log q}=\tau$ that $q^{2}\Psi(q)<1/2$ for all sufficiently large $q$. Hence we can apply Proposition~\ref{prop: equivalence} together with properties $(A)$ and $(B)$ to conclude that $x\in E(\Psi,\omega)$ so the set is non-empty.
\end{proof}

\section{Proof of Theorem~\ref{thm:main} and Theorem~\ref{thm:main2}} \label{sec: dim statements proof}
We begin our proofs with some discussion on the lower order of the function $q\mapsto \Psi(q)\omega(q).$ As we will see, two divided by the lower order coincides with the natural upper bound for the Hausdorff dimension of $E(\Psi,\omega)$ provided by \eqref{eq:Hausdorff upper bound}. The benefit of the lower order is that when doing dimension calculations it is easier to work with than the natural upper bound.

Given $\Psi:(0,\infty)\to (0,\infty)$ and $\omega:(0,\infty)\to (0,1)$ we define the lower order of $(\Psi,\omega)$ to be $$\lambda(\Psi,\omega)=\liminf_{q\to\infty}\frac{-\log \Psi(q)\omega(q)}{\log q}.$$ The following simple proposition connects the lower order of $(\Psi,\omega)$ with the upper bound appearing in \eqref{eq:Hausdorff upper bound}. The proof is standard but we include it for completeness.

\begin{prop}
\label{prop:lower order and dimension}
Let $\Psi:(0,\infty)\to(0,\infty)$ and $\omega:(0,\infty)\to (0,1)$ be such that $q\to q^{2}\Psi(q)\omega(q)$ is non-increasing.  Then $$\inf\left\{s\geq 0: \sum_{q=1}^{\infty}q(\Psi(q)\omega(q))^{s}<\infty\right\}=\frac{2}{\lambda(\Psi,\omega)}.$$
\end{prop}

\begin{proof}
 Let $\lambda<\lambda(\Psi,\omega)$ be arbitrary. Then for all $q$ sufficiently large we have $\Psi(q)\omega(q)<q^{-\lambda}.$ Taking $s>2/\lambda,$ it now follows that   $$\sum_{q=1}^{\infty}q(\Psi(q)\omega(q))^{s}\ll \sum_{q=1}^{\infty}q^{1-s\lambda}<\infty.$$ Since $s$ and $\lambda$ were arbitrary we have 
 \begin{equation}
 \label{eq:inequalityA}
 \inf\left\{s\geq 0: \sum_{q=1}^{\infty}q(\Psi(q)\omega(q))^{s}<\infty\right\}\leq \frac{2}{\lambda(\Psi,\omega)}.
 \end{equation}
Let us now establish the reverse inequality. Let us assume that $\lambda(\Psi,\omega)<\infty.$ Otherwise our desired equality follows from \eqref{eq:inequalityA}, as in this case both sides equal zero. Let $\lambda>\lambda(\Psi,\omega)$ be arbitrary. Then for infinitely many $q\in \R$ we have $\Psi(q)\omega(q)>q^{-\lambda}$. Let $s<2/\lambda$ be arbitrary. We observe the following:
\begin{align}
\label{eq:dyadic splitting}
\sum_{q=1}^{\infty}q(\Psi(q)\omega(q))^{s}=\sum_{n=1}^{\infty}\sum_{2^{n}<q\leq 2^{n+1}}q(\Psi(q)\omega(q))^{s}
&=\sum_{n=1}^{\infty}\sum_{2^{n}<q\leq 2^{n+1}}q^{1-2s}(q^{2}\Psi(q)\omega(q))^{s}\nonumber\\
&\gg \sum_{n=1}^{\infty}2^{n(1-2s)}\sum_{2^{n}<q\leq 2^{n+1}}(q^{2}\Psi(q)\omega(q))^{s}.
\end{align} Suppose that $q'\in \R$ is such that $\Psi(q')\omega(q')>(q')^{-\lambda}$ and $n$ is such that $2^{n+1}<q'\leq 2^{n+2},$ then since $q\to q^{2}\Psi(q)\omega(q)$ is decreasing, for such an $n$ we have 
\begin{equation}
   \label{eq:bound for infinite n}
   \sum_{2^{n}<q\leq 2^{n+1}}(q^{2}\Psi(q)\omega(q))^{s}\geq \sum_{2^{n}<q\leq 2^{n+1}}((q')^{2}\Psi(q')\omega(q'))^{s}\gg  2^{n(1+2s-s\lambda)}.
\end{equation}
\begin{equation}
\label{eq:infinite n lower bound}
   2^{n(1-2s)}\sum_{2^{n}<q\leq 2^{n+1}}(q^{2}\Psi(q)\omega(q))^{s}\geq 2^{n(1-2s)}\cdot 2^{n(1+2s-s\lambda)}=2^{n(2-\lambda s)}\geq 1.
\end{equation} In the final inequality we used that $2-\lambda s>0$. This inequality is a consequence of our choice of $s$. Combining \eqref{eq:dyadic splitting} with the fact \eqref{eq:infinite n lower bound} holds for infinitely many $n$, we see that $\sum_{q=1}^{\infty}q(\Psi(q)\omega(q))^{s}=\infty.$ Since $s$ and $\lambda$ were arbitrary, it follows that
\begin{equation}
\label{eq:inequalityB}
\frac{2}{\lambda(\Psi,\omega)}\leq \inf\left\{s\geq 0: \sum_{q=1}^{\infty}q(\Psi(q)\omega(q))^{s}<\infty\right\}.
\end{equation}
Combining \eqref{eq:inequalityA} and \eqref{eq:inequalityB} our result now follows.
\end{proof}

Observe that the conclusion of Proposition \ref{prop:lower order and dimension} can be made under the assumptions of Theorems \ref{thm:main} and \ref{thm:main2}. Hence to prove these statements it remains to show that
\begin{equation}
    \label{eq:WTS lower bound}
    \dim_{H}(E(\Psi,\omega))\geq \frac{2}{\lambda(\Psi,\omega)}
\end{equation}
within each of the relevant settings. Our proof is an adaptation of an argument of Bugeaud \cite{Bug1}. We will prove \eqref{eq:WTS lower bound} by constructing a sufficiently large Cantor set 
\begin{equation*}
    \cC=\bigcap_{k=1}^{\infty} E_{k} \subseteq E(\Psi,\omega)
\end{equation*}
and then obtain \eqref{eq:WTS lower bound} by applying Proposition~\ref{prop:dimension lower bound} to $\cC$. \par 
The proofs of Theorems \ref{thm:main} and \ref{thm:main2} follow the same strategy. For the purpose of brevity, we only prove Theorem~\ref{thm:main2} in full. The proof of Theorem \ref{thm:main2} contains additional technicalities that are unnecessary in the proof of Theorem~\ref{thm:main}. For the rest of this section we fix $\Psi$ and $\omega$ satisfying the assumptions of Theorem \ref{thm:main2}. We will highlight where differences arise in our proof when compared to the proof of Theorem~\ref{thm:main}. This will often simply involve using Lemma~\ref{lem: building U(p/q) with divisibility} in place of Lemma~\ref{lem: building U(p/q)}.
\subsection{Construction of a Cantor subset by induction}
In our construction of $\cC$ we adopt the following notational convention. Let $p/q\in \Q\cap (0,1)$ be such that $q$ satisfies the relevant partial quotient bound and is large enough so that Lemma~\ref{lem: building U(p/q) with divisibility} can be applied. We also assume that $q$ is sufficiently large so that Lemma \ref{lem:b size} can be applied. Given words $\{\bb_{l}\}_{l=1}^{L}$ satisfying the conclusion of Lemma \ref{lem: building U(p/q) with divisibility} for $p/q,$ we let $$U\left(\frac{p}{q}\right):=\bigcup_{l=1}^{L}I(0;\ba\bb_{l})$$ where $\ba=(a_{1},\ldots,a_{n})$ is the unique word of even length such that $p/q=[0;a_{1},\ldots,a_{n}]$\footnote{In our proof of Theorem~\ref{thm:main} the parameter $q$ is large enough so that Lemma~\ref{lem: building U(p/q)} and Lemma~\ref{lem:b size} can be applied. We then use the words $\{\bb_{l}\}_{l=1}^{L}$ satisfying the conclusion of Lemma \ref{lem: building U(p/q)}.}. Such a word exists by Lemma \ref{lem:rationalcfexpansion}. We will build $\cC$ inductively by starting with the base case and then repeatedly applying Lemma~\ref{inductive lemma} below. 

\subsubsection{Base level construction} We start by defining the set $E_{1}$. Consider the fundamental interval $I_{N_{0}}(0;2^{N_{0}})$ corresponding to the digit $2$ repeated $N_{0}$ times.  
We can assume that $N_{0}\in \N$ has been chosen sufficiently large so that for all $q\geq q_{N_{0}}(2^{N_{0}})$: 
\begin{equation}
    \label{eq:Psi quadratic bound}
    \Psi(q)\leq \frac{1}{1000q^{2}},
\end{equation} 
and 
\begin{equation}
    \label{eq:product bound}
    \prod_{j=1}^{\infty}\left(1-21\cdot 2^{j-1}q^{2}\Psi(2^{(j-1)/2}q)\right)\geq \frac{1}{2}\, .
\end{equation}
The former assumption can be shown to be possible by the condition $\lim_{q\to \infty}q^{2}\Psi(q)=0$, while the latter assumption follows from the condition $\sum_{q=1}^{\infty}q\Psi(q)<\infty$, Lemma \ref{lem:Geometric convergence}, and well-known properties of infinite series and infinite products.\par
Define the subset of the fundamental interval 
\begin{align*}
    J_{0}:=&\left\{ x \in I_{N_{0}}(0;\textbf{2}^{N_{0}}): a_{N_{0}+j}(x) \leq \frac{1}{7 \cdot 2^{j-1}q_{N_{0}}(2^{N_{0}})^{2}\Psi(2^{(j-1)/{2}}q_{N_{0}}(2^{N_{0}}))},\,  \quad j\geq 1\right\}\\
    \cup&\bigcup_{m=1}^{\infty}\left\{[0;2^{N_{0}},a_{1},\ldots, a_{m}]:a_{j}\leq \frac{1}{7 \cdot 2^{j-1}q_{N_{0}}(2^{N_{0}})^{2}\Psi(2^{(j-1)/{2}}q_{N_{0}}(2^{N_{0}}))},\,\quad  1\leq j\leq m \right\}.
\end{align*}
Observe that all irrational points in $J_{0}$ satisfy Proposition~\ref{prop: equivalence} \ref{E-bad}. This is due to the non-increasing property of $q\mapsto q^{2}\Psi(q)$ and noting that 
\begin{equation}
\label{eq:quotient exponential growth}
2^{(j-1)/2}q_{N_{0}}(2^{N_{0}})\leq q_{N_{0}+j}(2^{N_{0}}\ba)
\end{equation} for any word $\ba\in \N^{j}$. This is a consequence of \eqref{eq:recursive formulas}. Therefore by Lemma \ref{lem:Subset measure bound}, \eqref{eq:product bound} and \eqref{eq:quotient exponential growth} we have
    \begin{align} \label{size of J0}
        \mathcal{L}(J_{0})
        &\geq  \prod_{j=1}^{\infty}\left(1-\frac{3}{\frac{1}{7 \cdot 2^{j-1}q_{N_{0}}(2^{N_{0}})^{2}\Psi(2^{(j-1)/{2}}q_{N_{0}}(2^{N_{0}}))}} \right)\mathcal{L}\left(I_{N_{0}}(0;2^{N_{0}})\right)\nonumber\\
        &\geq \prod_{j=1}^{\infty}\left(1-21\cdot 2^{j-1}q_{N_{0}}(2^{N_{0}})^{2}\Psi\left(2^{(j-1)/2}q_{N_{0}}(2^{N_{0}})\right) \right)\mathcal{L}\left(I_{N_{0}}(0;2^{N_{0}})\right)\nonumber\\
        &= \frac{\mathcal{L}\left(I_{N_{0}}(0;2^{N_{0}})\right)}{2}\, .
    \end{align}
Hence $J_{0}$ retains a positive proportion of the measure of $I_{N_{0}}(0;2^{N_{0}})$.\par 
Pick a large integer $Q_{1} \in \N$ such that
\begin{equation}\label{eq:Q_1 size}
    Q_{1} > \max\left\{ -\frac{1000}{\mathcal{L}\left(I_{N_{0}}(0;2^{N_{0}})\right)}  \log\left(\mathcal{L}\left(I_{N_{0}}(0;2^{N_{0}})\right)\right) \right\} \, .
\end{equation}
Note that this lower bound for $Q_{1}$ depends only upon $N_{0}$. We will impose other lower bounds on the size of $Q_{1}$ later.
By Dirichlet's Theorem in Diophantine approximation, for any $x\in I_{N_{0}}(0;2^{N_{0}})$ there exists $(p,q)\in\mathbb{Z}\times \N$ satisfying
\begin{equation*}
    \left|x-\frac{p}{q}\right|<\frac{1}{qQ_{1}}\, , \quad 1\leq q\leq Q_{1}\, .
\end{equation*}
Let
\begin{equation*}
    A=\left\{x\in I_{N_{0}}(0;2^{N_{0}}): \left|x-\frac{p}{q}\right|<\frac{1}{qQ_{1}} \, , \quad 1\leq q\leq \frac{Q_{1}}{10}\right\}.
\end{equation*}
By \eqref{eq:Q_1 size} and a standard measure calculation we have that
\begin{align} \label{eq: A size base case}
    \mathcal{L}(A)\leq \sum_{1\leq q\leq \frac{Q_{1}}{10}} \frac{2}{qQ_{1}}(q\cL(I_{N_{0}}(0;2^{N_{0}}))+1)
    &\leq \frac{2}{10}\mathcal{L}\left(I_{N_{0}}(0;2^{N_{0}})\right)+\frac{2}{Q_{1}}\sum_{1\leq q\leq \frac{Q_{1}}{10}}q^{-1}\nonumber\\
    &\leq \frac{1}{4}\mathcal{L}\left(I_{N_{0}}(0;2^{N_{0}})\right).
\end{align}
Let $\tilde{T}=\{\frac{r_{j}}{s_{j}}\}_{1\leq j\leq \tilde{t}}$ be a maximal $(Q_{1})^{-2}$-separated family of rational points $\frac{r}{s}\in \Q\cap J_{0}$ with $\frac{Q_{1}}{10}\leq s\leq Q_{1}$. We claim that 
    \begin{equation}
    \label{eq:separated covering base case}
      J_{0}\setminus A \subseteq \bigcup_{j=1}^{\tilde{t}}\left[\frac{r_{j}}{s_{j}}-\frac{11}{Q_{1}^{2}},\frac{r_{j}}{s_{j}}+\frac{11}{Q_{1}^{2}}\right] .
    \end{equation} 
    This follows because if $x\in J_{0}\setminus A$, then by Dirichlet's Theorem there exists $(u,v)\in \Z\times \N$ with $\frac{Q_{1}}{10}\leq v\leq Q_{1}$ such that
    \begin{equation*}
        \left|x-\frac{u}{v}\right|<\frac{1}{vQ_{1}}\leq \frac{1}{v^{2}}\, .
    \end{equation*}
    By Lemma \ref{lem:Grace} this $(u,v)\in \Z\times \N$ must satisfy
     \begin{equation}
     \label{eq:l equation base case}
        \frac{u}{v}\in\left\{[0;a_{1}(x),\ldots,a_{l}(x)],[0;a_{1}(x),\ldots, a_{l}(x),1],[0;a_{1}(x),\ldots, a_{l}(x),a_{l+1}(x)-1]\right\}  
    \end{equation}
    for some $l\in \N$. Note that $l$ can be made arbitrarily large by choosing $Q_{1}$ suitably large in a way that does not depend upon $x$. Here we are using the fact that the size of the partial quotients is bounded at each level for elements of $J_{0}$. Hence we can choose $Q_{1}$ large enough so that $l>N_{0}$ thus guaranteeing $\frac{u}{v}\in J_{0}$.
    Consequently 
    \begin{equation}
    \label{eq:Almost covered base case}
    J_{0}\setminus A\subseteq \bigcup_{\frac{u}{v}\in J_{0}: \frac{Q_{1}}{10}\leq v\leq Q_{1}}\left[\frac{u}{v}-\frac{1}{vQ_{1}},\frac{u}{v}+\frac{1}{vQ_{1}}\right]\, .
    \end{equation}
     Equation \eqref{eq:separated covering base case} now follows from \eqref{eq:Almost covered base case} and an application of the triangle inequality. Namely, if $\frac{u}{v}\not\in \tilde{T}$ then by the maximality of $\tilde{T}$ there exists some $\frac{r}{s}\in \tilde{T}$ such that $|r/s-u/v|\leq Q_{1}^{-2}.$ Then, by the triangle inequality 
     \begin{align*}
         \left|x-\frac{r}{s}\right|&\leq \left| x-\frac{u}{v}\right| + \left|\frac{r}{s}-\frac{u}{v}\right|
         \leq \frac{1}{vQ_{1}}+\frac{1}{Q_{1}^{2}}
          \leq \frac{10}{Q_{1}^{2}}+\frac{1}{Q_{1}^{2}}\, .
     \end{align*}
     giving us \eqref{eq:separated covering base case}. We now use our measure bounds for $J_{0}$ and $A$ to determine a lower bound on the cardinality of $\tilde{T}$. By \eqref{size of J0}, \eqref{eq: A size base case}, and \eqref{eq:separated covering base case} we have that
\begin{align*}
\tilde{t} \cdot \frac{22}{Q_{1}^{2}} \geq 
   \mathcal{L}\left(\bigcup_{j=1}^{\tilde{t}}\left[ \frac{r_{j}}{s_{j}}-\frac{11}{Q_{1}^{2}}, \frac{r_{j}}{s_{j}}+\frac{11}{Q_{1}^{2}}\right]\right) \geq \cL(J_{0}\setminus A) \geq \frac{1}{4} \mathcal{L}\left(I_{N_{0}}(0;2^{N_{0}})\right)
\end{align*}
Thus we can rearrange to obtain
\begin{equation*}
    \tilde{t} \geq  \frac{ Q_{1}^{2}}{88}\mathcal{L}\left(I_{N_{0}}(0;2^{N_{0}})\right)\, .
\end{equation*}
Lastly note that since each $\frac{r}{s}\in \tilde{T}$ belongs to $J_{0}$ we can write $\frac{r}{s}=[0;a_{1},\ldots, a_{n}]$ for some $n=N_{0}+j$ for some $j\in \N$ where the partial quotients $a_{N_{0}+1},\ldots, a_{n}$ satisfy the upper bounds imposed by the definition of $J_{0}$. We would like to apply Lemma \ref{lem: building U(p/q) with divisibility} to $\frac{r}{s},$ however it is possible that $n$ is odd. If $n$ is odd we write $\tfrac{r}{s}=[0;a_{1},\ldots, a_{n-1}, a_{n}-1,1]$ if $a_{n}\geq 2$ and  $\tfrac{r}{s}=[0;a_{1},\ldots, a_{n-1}+1]$ if $a_{n}=1.$ Since $a_{i}=2$ for $1\leq i \leq N_{0}$, $q\to q^{2}\Psi(q)$ is non-increasing, and \eqref{eq:quotient exponential growth} we have
\begin{equation*}
    a_{N_{0}+i}+1\leq \frac{1}{7 \cdot 2^{j-1}q_{N_{0}}(2^{N_{0}})^{2}\Psi(2^{(j-1)/{2}}q_{N_{0}}(2^{N_{0}}))} +1 < \frac{1}{6 s^{2}\Psi(s)}\, , \qquad 1\leq i \leq j\, . 
\end{equation*}
Thus, regardless of the parity of $n$ we can find a continued fraction expansion of $\frac{r}{s}$ of even length so that the bound $$\max_{1\leq i\leq n} a_{i} <\frac{1}{6s^{2}\Psi(s)}$$ is satisfied. Therefore Lemma~\ref{lem: building U(p/q) with divisibility} can be applied to each $\frac{r}{s}\in \tilde{T}$. We define
\begin{equation*}
    E_{1}:=\bigcup_{j=1}^{\tilde{t}}U\left(\frac{r_{j}}{s_{j}}\right).
\end{equation*}
In this final line we have assumed that $Q_{1}$ is sufficiently large so that Lemma \ref{lem: building U(p/q) with divisibility} can be applied to each element in $\tilde{T}$.

\subsubsection{Inductive step} The following lemma allows us to apply an inductive argument.

\begin{lem} \label{inductive lemma}
 Let $Q\geq Q_{1}$. Then for all $Q'\geq Q$ sufficiently large in a way that only depends upon $Q,$ $\Psi$ and $\omega$ we have the following:  Suppose $\frac{p}{q}$ is such that 
 \begin{enumerate}[1)]
     \item $\frac{Q}{10}\leq q\leq Q$.
     \item For any $x\in U(\frac{p}{q})$ there are no solutions to
    \begin{equation*}
        \left|x-\frac{r}{s}\right|<(1-\omega(s))\Psi(s) \quad Q_{1}\leq s \leq Q\, .
    \end{equation*}
    \item Writing $\frac{p}{q}=[0;c_{1}, \ldots, c_{n}]$ for $n$ even, we have 
    \begin{equation*}
        c_{i} \leq \frac{1}{6q^{2}\Psi(q)} \, , \quad 1\leq i \leq n\, .
    \end{equation*}
 \end{enumerate}
    Then we can pick a collection of rational points $T=\{\frac{r_{j}}{s_{j}}\}_{1\leq j\leq t}$ with
    \begin{equation*}
        t=\left\lfloor \frac{(Q'^{2})\Psi(Q)\omega(Q)}{10^{10}}\right \rfloor
    \end{equation*}
    such that
    \begin{enumerate}
        \item[a)] $\frac{r_{j}}{s_{j}}\in U(\frac{p}{q})$ and $\frac{Q'}{10}\leq s_{j}\leq Q'$ for each $1\leq j\leq t$, 
        \item[b)] $|\frac{r_{j}}{s_{j}}-\frac{r_{i}}{s_{i}}|\geq (Q')^{-2}$ for every $1\leq j<i\leq t$,
        \item[c)] Writing $\frac{r_{j}}{s_{j}}=[0;a_{1},\ldots, a_{l}]$ for some $l$ even we have that
    \begin{equation*}
        a_{i}\leq \frac{1}{6 s_{j}^{2}\Psi(s_{j})} \, , \quad 1\leq i\leq l\, .
    \end{equation*}
        \item[d)] $U(\frac{r_{j}}{s_{j}})\subset U(\frac{p}{q})$ for every $1\leq j\leq t$,
        \item[e)] For every $x\in U(\frac{r_{j}}{s_{j}})$
        \begin{enumerate}
            \item[i)] $$\left|x-\frac{r}{s}\right|<\Psi(s),\quad Q< s\leq Q'$$ has a solution $(r,s)\in \Z\times \N.$ 
            \item[ii)] $$\left|x-\frac{r}{s}\right|< (1-\omega(s))\Psi(s),\quad Q_1\leq s\leq Q'$$ has no solutions $(r,s)\in\mathbb{Z}\times \N$. 
        \end{enumerate}
    \end{enumerate}
\end{lem}

\begin{rem}
    Note the sets $U(\frac{r_{j}}{s_{j}})$ appearing in $d),e)$ are well-defined since Lemma~\ref{lem: building U(p/q) with divisibility} is applicable to $\frac{r_{j}}{s_{j}}\in T$ by $c)$. 
\end{rem}
\begin{rem}
Lemma \ref{inductive lemma} can be adapted to prove Theorem~\ref{thm:main}. Under the assumptions of Theorem~\ref{thm:main} its statement can be simplified. In particular, to apply Lemma~\ref{lem: building U(p/q)} (rather than Lemma~\ref{lem: building U(p/q) with divisibility}) we do not need assumption 3). Moreover, we do not require the resulting set of rationals $T$ to satisfy c). Furthermore, the constant appearing in the denominator in the formula for $t$ can be improved, but this has no effect on the later dimension calculations. 
\end{rem}

\begin{proof}
Let us begin by fixing $Q\geq Q_{1}$. We will now show that for all $Q'$ sufficiently large in a way that depends upon $Q,$ $\Psi$ and $\omega$ the desired conclusion holds. Let $\frac{p}{q}$ be as in the statement of our lemma. Write $\frac{p}{q}=[0;\bc]$ where $\bc=(c_{1},\ldots,c_{n})$ with $n$ even (this is possible by Lemma \ref{lem:rationalcfexpansion}). By assumption 3) we can apply Lemma \ref{lem: building U(p/q) with divisibility} to $p/q$\footnote{In the proof of Theorem \ref{thm:main} we would apply Lemma~\ref{lem: building U(p/q)} at this point.} and let $\{\bb_{l}\}_{l=1}^{L}$ denote the corresponding set of words. We let $m$ denote their common length. For each $l$ we let $\frac{p(\bc\bb_{l})}{q(\bc\bb_{l})}=[0;\bc\bb_{l}]$ and
    \begin{align*}
    J(\bb_{l}):=&\left\{ x\in I(0;\bc\bb_{l}): a_{n+m+j}(x)\leq \frac{1}{7\cdot 2^{j-1}q(\bc\bb_{l})^{2}\Psi\left(2^{(j-1)/2}q(\bc\bb_{l})\right)} , \quad j\geq 1\right\}\\
    \cup&\bigcup_{m=1}^{\infty}\left\{[0;\bc\bb_{l},a_{1},\ldots, a_{m}]:a_{j}\leq \frac{1}{7 \cdot 2^{j-1}q_{N_{0}}(\bc\bb_{l})^{2}\Psi(2^{(j-1)/{2}}q_{N_{0}}(\bc\bb_{l})},\,\quad  1\leq j\leq m \right\}.
    \end{align*}
    and $J=\bigcup_{l=1}^{L} J(\bb_{l}).$\par 
    By Lemma \ref{lem:Subset measure bound}, \eqref{eq:product bound}, and the definition of $U\left(\frac{p}{q}\right)$ we have
    \begin{align} \label{size of J}
        \mathcal{L}(J)&=\sum_{l=1}^{L}\mathcal{L}(J(\bb_{l})) \geq \sum_{l=1}^{L} \prod_{j=1}^{\infty}\left(1-\frac{3}{\frac{1}{7\cdot 2^{j-1}q(\bc\bb_{l})^{2}\Psi(2^{(j-1)/2}q(\bc\bb_{l}))}} \right)\mathcal{L}\left(I(0;\bc\bb_{l})\right)\nonumber\\
        &\geq \frac{1}{2}\sum_{l=1}^{L} \mathcal{L}\left(I(0;\bc\bb_{l})\right)\nonumber\\
        &= \frac{1}{2}\mathcal{L}\left(U\left(\frac{p}{q}\right)\right).
    \end{align}
Let us assume that $Q'$ has been chosen sufficiently large so that we have:
   \begin{align}
       \label{eq:Q bound}
       &Q'\geq \max_{p'/q'\in \mathbb{Q}\cap (0,1):Q/10\leq q'\leq Q}1000 \left|U\left(\frac{p'}{q'}\right)\right|^{-1} \log\left(\left|U\left(\frac{p'}{q'}\right)\right|^{-1}\right)\, ,\\
    \label{eq: Q bound 2}
       &\max_{\stackrel{p'/q'\in \mathbb{Q}\cap (0,1)}{Q/10\leq q'\leq Q}}\max_{\bb_{l}\text{ provided by Lemma }\ref{lem: building U(p/q) with divisibility}}\frac{1}{2}q(\bc\bb_{l})^{2}\Psi\left(q(\bc\bb_{l})\right)>\left(\frac{Q'}{10}\right)^{2}\Psi\left(\frac{Q'}{10}\right)\, , \\
   \label{eq: Q bound 3}
       &Q'>\max_{\stackrel{p'/q'\in \mathbb{Q}\cap (0,1)}{Q/10\leq q'\leq Q}}\max_{\bb_{l}\text{ provided by Lemma }\ref{lem: building U(p/q) with divisibility}}10 q(\bc\bb_{l})\, , \\
           \label{eq: Q bound 4}
        &Q^{2}\Psi(Q)>12\left(\frac{Q'}{10}\right)^{2}\Psi\left(\frac{Q'}{10}\right)   
   \end{align}
   Note that \eqref{eq:Q bound}, \eqref{eq: Q bound 3} and \eqref{eq: Q bound 4} are possible since one side is fixed (depending on $Q$, $\Psi$ and $\omega$), and that \eqref{eq: Q bound 2} is possible since $\lim_{q\to\infty} q^{2}\Psi(q)=0$. By applying Dirichlet's Theorem for $Q'$ and an arbitrary choice of $x\in U(\frac{p}{q}),$ we see there exists a solution $(u,v)\in \Z\times \N$ satisfying
    \begin{equation*}
        \left|x-\frac{u}{v}\right|<\frac{1}{vQ'} \, , \quad 1\leq v\leq Q'\, .
    \end{equation*}
    Letting
    \begin{equation*}
        A=\left\{x\in U\left(\frac{p}{q}\right): \left|x-\frac{u}{v}\right|<\frac{1}{vQ'} \, , \, \, 1\leq v\leq \frac{Q'}{10}\right\}
    \end{equation*}
    we see as before that
    \begin{align} 
    \label{size of A_k}
        \mathcal{L}(A)\leq \sum_{1\leq v\leq \frac{Q'}{10}} \frac{2}{vQ'}\left(v\mathcal{L}\left(U\left(\frac{p}{q}\right)\right)+1\right)& \leq \mathcal{L}\left(U\left(\frac{p}{q}\right)\right)\frac{2}{10}+\frac{2}{Q'}\sum_{1\leq v\leq \frac{Q'}{10}} \frac{1}{v}\nonumber \\
        &\leq \frac{1}{4}\mathcal{L}\left(U\left(\frac{p}{q}\right)\right),
    \end{align}
    where in the final line we have used \eqref{eq:Q bound}. Let $\tilde{T}=\{\frac{r_{j}}{s_{j}}\}_{1\leq j\leq \tilde{t}}$ be a maximal $(Q')^{-2}$-separated family of rational points $\frac{r}{s}\in \Q\cap J$ with $\frac{Q'}{10}\leq s\leq Q'$. We claim that 
    \begin{equation}
    \label{eq:separated covering}
      J\setminus A \subseteq \bigcup_{j=1}^{\tilde{t}}\left[\frac{r_{j}}{s_{j}}-\frac{11}{(Q')^{2}},\frac{r_{j}}{s_{j}}+\frac{11}{(Q')^{2}}\right] .
    \end{equation} This follows because if $x\in J\setminus A,$ then by Dirichlet's Theorem there exists $(u,v)\in \Z\times \N$ with $\frac{Q'}{10}\leq v\leq Q'$ such that
    \begin{equation*}
        \left|x-\frac{u}{v}\right|<\frac{1}{vQ'}\leq \frac{1}{v^{2}}\, .
    \end{equation*}
    By Lemma \ref{lem:Grace} this $(u,v)\in \Z\times \N$ must satisfy
     \begin{equation}
     \label{eq:l equation}
        \frac{u}{v}\in\left\{[0;a_{1}(x),\ldots,a_{l}(x)],[0;a_{1}(x),\ldots, a_{l}(x),1],[0;a_{1}(x),\ldots, a_{l}(x),a_{l+1}(x)-1]\right\}  
    \end{equation}for some $l\in \N$. Note that \eqref{eq: Q bound 3} implies that $v\geq \frac{Q'}{10}> q(\bc\bb_{l})$ for any $\bc$ and $\bb_{l}$. Therefore $l$ must be strictly greater than the length of $\bc\bb_{l}$ for any $\bc$ and $\bb_{l}$. Using this fact we see that \eqref{eq:l equation} implies that  $\frac{u}{v}\in J$.
    Consequently 
    \begin{equation}
    \label{eq:Almost covered}
    J\setminus A\subseteq \bigcup_{\frac{u}{v}\in J: Q'/10\leq v\leq Q'}\left[\frac{u}{v}-\frac{1}{vQ'},\frac{u}{v}+\frac{1}{vQ'}\right]
    \end{equation}
     \eqref{eq:separated covering} now follows from \eqref{eq:Almost covered} and an application of the triangle inequality.
    
    Hence, by \eqref{size of J}, \eqref{size of A_k} and \eqref{eq:separated covering}, we have that
    \begin{equation*}
        \mathcal{L}\left( \bigcup_{j=1}^{\tilde{t}}\left[\frac{r_{j}}{s_{j}}-\frac{11}{(Q')^{2}},\frac{r_{j}}{s_{j}}+\frac{11}{(Q')^{2}}\right] \right) \geq \frac{1}{4}\mathcal{L}\left(U\left(\frac{p}{q}\right)\right),
    \end{equation*}
    and so we have the lower bound
    \begin{equation} \label{eq: step where different lemma used 1}
        \tilde{t}\geq \frac{(Q')^{2}\mathcal{L}\left(U\left(\frac{p}{q}\right)\right)}{88}.
    \end{equation}
    By Lemma \ref{lem: building U(p/q) with divisibility} and using the fact $q\to \Psi(q)\omega(q)$ is non-increasing, the above implies the following lower bound\footnote{It is here where a change in constant appears in the proof of Theorem~\ref{thm:main}, since Lemmas~\ref{lem: building U(p/q)} and ~\ref{lem: building U(p/q) with divisibility} give different bounds on the size of $U(p/q)$. Namely, the constant in the denominator would be $563200$ rather than $88\times 10^{8}$. This has a knock-on effect in later constants.}
    \begin{equation*}
        \tilde{t}\geq \frac{(Q')^{2}\Psi(Q)\omega(Q)}{88\times 10^{8}}\, .
    \end{equation*}
  We now remove the rightmost rational from $\tilde{T}.$ Thus we are left with at least $$\frac{(Q')^{2}\Psi(Q)\omega(Q)}{88\times 10^{8}}-1$$ many rationals. We may assume that $Q'$ has been chosen sufficiently large that from those rationals that remain, there exists a subset of cardinality $$\left\lfloor \frac{(Q')^{2}\Psi(Q)\omega(Q)}{10^{10}}\right\rfloor.$$ We let this set be our set $T$. This collection satisfies the required properties:
    \begin{itemize}[leftmargin=2\parindent]
        \item Conditions $a)$ and $b)$ are satisfied immediately because $T$ is a subset of $\tilde{T}$.
        \item To see that condition $c)$ holds let $\frac{r_{j}}{s_{j}}\in T$ and $(a_{1},\ldots,a_{l})$ be such that $\frac{r_{j}}{s_{j}}=[0;a_{1},\ldots, a_{l}]$ and $(a_{1},\ldots,a_{l})$ satisfies the restrictions imposed by the set $J$, i.e. $(a_{1},\ldots,a_{l})$ begins with $\bc\bb_{l}$ for some $\bb_{l}$ and satisfies the partial quotient upper bound. The integer $l$ is not necessarily even. We will explain how to overcome this technicality at the end of this discussion. We see that $c)$ holds by the following case analysis:
    \begin{itemize}[leftmargin=\parindent]
        \item For $1\leq i \leq n$ condition $c)$ holds by our assumption 3) combined with the non-increasing property of $q\mapsto q^{2}\Psi(q)$.
        \item For $i=n+1$ Lemma \ref{lem:b size} implies that $a_{n+1}\leq \frac{2}{q^{2}\Psi(q)}$. Using this inequality and the non-increasing property of $q\mapsto q^{2}\Psi(q)$ and \eqref{eq: Q bound 4} we have that
        \begin{equation*}
            a_{n+1}\leq \frac{2}{q^{2}\Psi(q)} \leq \frac{2}{Q^{2}\Psi(Q)}< \frac{1}{6\left(\frac{Q'}{10}\right)^{2}\Psi\left( \frac{Q'}{10}\right)}\leq \frac{1}{6s^{2}\Psi(s)}\, .
        \end{equation*}
        \item Let $\bb_{l}$ be such that $\frac{r_{j}}{s_{j}}\in J(\bb_{l})$. Then for $n+2\leq i\leq n+m$ condition $c)$ holds by the non-increasing property of $q\mapsto q^{2}\Psi(q),$ Lemma~\ref{lem: building U(p/q) with divisibility} and \eqref{eq: Q bound 2}. In particular, for each $1\leq i \leq m$ we have  $a_{n+i}=b_{i-1}$ and so
        \begin{equation*}
            a_{n+i}\leq \left\lfloor \frac{1}{3q_{n+i-1}^{2}\Psi(q_{n+i-1})}\right\rfloor \leq \frac{1}{3q(\bc\bb_{l})^{2}\Psi(q(\bc\bb_{l}))} \leq \frac{1}{6\left(\frac{Q'}{10}\right)^{2}\Psi\left( \frac{Q'}{10}\right)}\leq \frac{1}{6 s^{2}\Psi(s)}\, .
        \end{equation*}
        \item For $\bb_{l}$ as above and $n+m+1\leq i \leq l,$ we have by construction of the set $J$ and the non-increasing condition on $q\mapsto q^{2}\Psi(q)$, that the following holds assuming we chose our original $Q_{1}$ to be sufficiently large
        \begin{equation*}
            a_{i}+1\leq \frac{1}{7\cdot 2^{i-n-m-1}q(\bc\bb_{l})^{2}\Psi(2^{(i-n-m-1)/2}q(\bc\bb_{l}))}+1 <\frac{1}{6s^{2}\Psi(s)}\,.
        \end{equation*}
        Here we have used the fact that $s\geq 2^{(i-n-m-1)/2}q(\bc\bb_{l}),$ which is a consequence of \eqref{eq:recursive formulas}. It follows that $$\max_{1\leq i\leq l}a_{i}< \frac{1}{6s^{2}\Psi(s)}.$$ So $c)$ is satisfied by $(a_{1},\ldots, a_{l})$ which is not necessarily of even length. Because of the additional $+1$ term appearing in the above equation, if we use the substitutions $[0;a_{1}\ldots,a_{l}]=[0;a_{1},\ldots, a_{l}-1,1]$ if $a_{l}\geq 2$ and $[0;a_{1}\ldots,a_{l}]=[0;a_{1},\ldots, a_{l-1}+1]$ if $a_{l}=1$ to ensure our word is of even length, then $c)$ is satisfied by the resulting even length word. This means that Lemma~\ref{lem: building U(p/q) with divisibility} is applicable to each $\frac{r_{j}}{s_{j}}\in T$ and so the sets $U(\frac{r_{j}}{s_{j}})$ are well-defined.
    \end{itemize}
     \item Let $\frac{r_j}{s_j}\in T.$ Since $s_{j}\geq Q_{1}$ and $s_{j}\geq \frac{Q'}{10},$ it follows from \eqref{eq:Psi quadratic bound} that $\Psi(s_{j})\leq (Q')^{-2}$. Recall that we removed the rightmost rational from $\tilde{T}$ to define $T$. Since $\tilde{T}$ is a $(Q')^{-2}$ separated subset of $U\left(\frac{p}{q}\right)$ and $\Psi(s_{j})\leq (Q')^{-2}$, it follows therefore that $$U\left(\frac{r_{j}}{s_{j}}\right)\subseteq \left[\frac{r_{j}}{s_j},\frac{r_{j}}{s_{j}}+\Psi(s_{j})\right]\subseteq\left[\frac{r_{j}}{s_j},\frac{r_{j}}{s_{j}}+(Q')^{-2}\right]\subseteq U\left(\frac{p}{q}\right). $$ So condition $d)$ holds.
        \item Let $\frac{r_j}{s_j}\in T$ and $x\in U\left(\frac{r_j}{s_j}\right)$. It  follows from Lemma \ref{lem: building U(p/q) with divisibility} that $|x-\frac{r_j}{s_j}|\leq \Psi(s_j)$\footnote{In the proof of Theorem \ref{thm:main} this would follow from Lemma \ref{lem: building U(p/q)}.}. We can choose $Q'$ large enough so that $\frac{Q'}{10}>Q.$ Since $\frac{Q'}{10}\leq s_{j}\leq Q'$ we have $Q<s_{j}\leq Q'$ and so condition $e)i)$ holds.
        \item Let $\frac{r_j}{s_j}\in T$ and  $x\in U\left(\frac{r_{j}}{s_{j}}\right)$. By \eqref{eq:Psi quadratic bound} and Lemma \ref{lem:Legendre}, if $(u,v)\in \Z\times \N$ with $v\geq Q_{1}$ satisfies $$\left|x-\frac{u}{v}\right|< (1-\omega(v))\Psi(v)$$ then $\frac{u}{v}$ is a convergent of $x$. Thus, to verify condition $e)ii)$, it suffices to consider convergents $\frac{p_{k}}{q_{k}}$ to $x$ satisfying $Q_{1}\leq q_{k}\leq Q'$. Therefore, to show that condition $e)ii)$ holds, it suffices to show that there exists no $k\in\N$ such that
        \begin{equation}
        \label{eq:no solutions}
            \left|x-\frac{p_{k}}{q_k}\right|< (1-\omega(q_{k}))\Psi(q_{k})\quad \textrm{ and }\quad Q_{1}\leq q_{k}\leq Q'.
        \end{equation}        
        Let $(a_{k})_{k=1}^{\infty}$ be sequence of partial quotients for $x$. Since $x\in U\left(\frac{p}{q}\right)$ we have $$\frac{p}{q}=[0;a_{1},\ldots, a_{n}].$$ Moreover, since $x\in U\left(\frac{r_{j}}{s_j}\right)$ there exists $n_{1}>n$ such that  $$\frac{r_{j}}{s_{j}}=[0;a_{1},\ldots, a_{n_{1}}].$$ Since $(q_k)_{k=1}^{\infty}$ is an increasing sequence,  $q=q_{n}\leq Q,$ and $x\in U\left(\frac{p}{q}\right),$ it follows from statement $2)$ in our assumptions that \eqref{eq:no solutions} has no solutions for $k\leq n$. If $n<k\leq n+m-1,$ then since $[0;a_{1},\ldots, a_{n_{1}}]$ is the continued fraction expansion of $\frac{r_{j}}{s_{j}}$ and $\frac{r_{j}}{s_{j}}\in J,$ it follows that $a_{k+1}$ is a digit coming from a word belonging to $\{\bb_{l}\}_{l=1}^{L},$ in which case 
        \begin{equation}
            \label{eq:a_{l+1} bound}
            a_{k+1}\leq \left \lfloor \frac{1}{3q_{k}^{2}\Psi(q_{k})}\right \rfloor
        \end{equation}
         by Lemma \ref{lem: building U(p/q) with divisibility}\footnote{Or by Lemma~\ref{lem: building U(p/q)} in the proof of Theorem \ref{thm:main}}. Combining Lemma \ref{lem:approximation quality} and \eqref{eq:a_{l+1} bound} we have $$\left|x-\frac{p_{k}}{q_{k}}\right|\geq \frac{1}{3a_{k+1}q_k^{2}}\geq \Psi(q_{k}).$$ So there are no solutions to \eqref{eq:no solutions} for $n<k\leq n+m-1$. If $n+m\leq k<n_{1}$ then since $[0;a_{1},\ldots, a_{n_{1}}]$ is the continued fraction expansion of $\frac{r_{j}}{s_{j}}$ and $\frac{r_{j}}{s_{j}}\in J,$ there exists $\bb_{l}$ such that we have 
        $$a_{k+1}\leq \frac{1}{7\cdot 2^{k-m-n}q(\bc\bb_{l})^{2}\Psi(2^{(k-m-n)/2}q(\bc\bb_{l}))}.$$ 
        Appealing to \eqref{eq:recursive formulas} we can show that $q_{k}\geq 2^{(k-m-n)/2}q(\bc\bb_{l})$. Using our assumption that the function $q\mapsto q^{2}\Psi(q)$ is non-increasing, the following holds 
    \begin{equation}
    \label{eq:q power 2 bound}
        q_{k}^{2}(x)\Psi(q_{k}(x)) \leq 2^{k-m-n}q(\bc\bb_{l})^{2}\Psi(2^{(k-m-n)/2}q(\bc\bb_{l})).
    \end{equation}
    Hence, by Lemma \ref{lem:approximation quality} and \eqref{eq:q power 2 bound} we have
    \begin{align*}
    \label{eq: different lemma used 2}
        \left|x-\frac{p_{k}(x)}{q_{k}(x)}\right|>\frac{1}{3a_{k+1}(x)q_{k}(x)^{2}}& \geq \frac{1}{3\frac{1}{7\cdot 2^{k-m-n}q(\bc\bb_{l})^{2}\Psi(2^{(k-m-n)/2}q(\bc\bb_{l}))}q_{k}(x)^{2}} \\
        &\geq  \Psi(q_{k}(x)),
    \end{align*} 
    So \eqref{eq:no solutions} has no solutions for $n+m\leq k<n_{1}$. Taking $k=n_{1},$ then since $x\in U\left(\frac{r_{j}}{s_{j}}\right)$ it follows from the definition of $U\left(\frac{r_{j}}{s_{j}}\right)$ that $$\left|x-\frac{p_{n_{1}}}{q_{n_{1}}}\right|=\left|x-\frac{r_{j}}{s_{j}}\right|\geq  (1-\omega(s_{j}))\Psi(s_{j}).$$ So \eqref{eq:no solutions} has no solution when $k=n_{1}$. Finally, if $k> n_{1}$ then by applying Lemma \ref{lem:b size} to $\frac{r_{j}}{s_{j}}$\footnote{which is satisfied under the assumptions of both Theorem~\ref{thm:main} and Theorem~\ref{thm:main2}} written with an even number of partial quotients we have that
    \begin{equation*}
        a_{n_{1}+1}\geq \frac{1}{2q_{n_{1}}^{2}\Psi(q_{n_{1}})}\, . 
    \end{equation*}
    So, combined with \eqref{eq:recursive formulas}, \eqref{eq:Psi quadratic bound} and the bound $\frac{Q'}{10}\leq s_{j}$ we have $$q_{k}\geq a_{n_{1}+1}q_{n_{1}}\geq \frac{q_{n_{1}}}{2q_{n_{1}}^{2}\Psi(q_{n_{1}})}=\frac{s_{j}}{2s_{j}^{2}\Psi(s_{j})}\geq 500s_{j}>Q'.$$  Thus any convergent $\frac{p_{k}}{q_{k}}$ of $x\in U(\frac{r_{j}}{s_{j}})$ with $k>n_{1}$ cannot contribute a solution to \eqref{eq:no solutions}. Summarising, we have shown that \eqref{eq:no solutions} has no solutions for all $k\in \N.$ This proves condition $e)ii)$ holds.
    \end{itemize}
\end{proof}

Let $E_{1}$ and $Q_{1}$ be as above. Now let $(Q_{k})_{k=1}^{\infty}$ be a rapidly increasing sequence of integers satisfying the following properties:
\begin{itemize}[leftmargin=2\parindent]
    \item For each $k\in \N$, $Q_{k+1}$ is sufficiently large that the conclusion of Lemma \ref{inductive lemma} holds when we have taken $Q=Q_{k}$ in the statement of this lemma.
    \item We have 
    \begin{equation}
        \label{eq:View lower order}
        \lim_{k\to\infty}\frac{-\log \Psi(Q_{k})\omega(Q_{k})}{\log Q_{k}} =\lambda(\Psi,\omega).
    \end{equation}
  \item We have 
  \begin{equation}
  \label{eq:logs disappear}
     \lim_{k\to\infty}\frac{\log \Psi(Q_{k})\omega(Q_{k})}{\log \Psi(Q_{k+1})\omega(Q_{k+1})}=0.
  \end{equation}
  
    \end{itemize} Repeatedly applying Lemma \ref{inductive lemma} yields a sequence of sets $(E_k)_{k=1}^{\infty}$ satisfying the following properties:

\begin{itemize}[leftmargin=2\parindent]
    \item For each $E_{k}$ there exists a set of rationals $S_{k}=\{\frac{r_{j,k}}{s_{j,k}}\}_{j\in J_{k}}$ such that $$E_{k}=\bigcup_{j\in J_{k}}U\left(\frac{r_{j,k}}{s_{j,k}}\right)$$ is a disjoint union and $\frac{Q_{k}}{10}\leq s_{j,k}\leq Q_{k}$ for all $j\in J_{k}.$
    \item For all $k\geq 1$ we have
    $E_{k+1}\subset E_{k}$.
    \item For all $x\in E_{k}$ there exists at least $k$ distinct choices of $(r,s)\in \Z\times \N$ satisfying $$\left| x-\frac{r}{s}\right|<\Psi(s),\quad Q_{1}\leq s\leq Q_{k}.$$
    \item For all $x\in E_{k}$ there are no $(r,s)\in \Z\times \N$ satisfying $$\left| x-\frac{r}{s}\right|<(1-\omega(s)\Psi(s),\quad Q_{1}\leq s\leq Q_{k}.$$
    \end{itemize}
\subsubsection{Hausdorff dimension of $\cC$} Taking $$\cC=\bigcap_{k=1}^{\infty}E_{k},$$ then it is clear from the above that $\cC\subset E(\Psi,\omega).$ It also follows from Lemma \ref{inductive lemma} that we can assume that our sets $\{E_k\}$ satisfy the following properties:    
    \begin{itemize}[leftmargin=2\parindent]
    \item For each $k\in \N$ and $\frac{r_{j,k}}{s_{j,k}}\in S_{k},$ there exists $$\left\lfloor \frac{(Q_{k+1})^{2}\Psi(Q_{k})\omega(Q_{k})}{10^{10}}\right\rfloor$$ choices of $j'\in J_{k+1}$ such that $$U\left(\frac{r_{j',k+1}}{s_{j',k+1}} \right)\subset U\left(\frac{r_{j,k}}{s_{j,k}} \right).$$
    \item For each $k\in \N$ and $\frac{r_{j,k}}{s_{j,k}}\in S_{k},$ we have $$\mathcal{L}\left(U\left(\frac{r_{j,k}}{s_{j,k}}\right)\right)\geq \frac{\Psi(Q_{k})\omega(Q_{k})}{10^{8}}.$$ This follows from Lemma \ref{lem: building U(p/q) with divisibility} \footnote{Or from Lemma~\ref{lem: building U(p/q)} in the proof of Theorem \ref{thm:main} with a slightly better constant} and the fact $q\mapsto \Psi(q)\omega(q)$ is a decreasing function.
    \item For each $k\in \N$ and $\frac{r_{j,k}}{s_{j,k}},\frac{r_{j',k}}{s_{j',k}}\in S_{k},$ there is a gap of size at least $\frac{1}{2Q_{k}^{2}}$ between $U\left(\frac{r_{j,k}}{s_{j,k}}\right)$ and $U\left(\frac{r_{j',k}}{s_{j',k}}\right).$ This follows from \eqref{eq:Psi quadratic bound} and the inequality $\left|\frac{r_{j,k}}{s_{j,k}}-\frac{r_{j',k}}{s_{j',k}}\right|\geq \frac{1}{Q_{k}^{2}}$ which is a consequence of Lemma \ref{inductive lemma}.
\end{itemize}
Using the properties listed above and Proposition \ref{prop:dimension lower bound} we have
\begin{align*}
\dim_{H}(E(\Psi,\omega))\geq \dim_{H}(\cC)&\geq \liminf_{k\to\infty}\frac{\log \lfloor \frac{Q_{k}^{2}\Psi(Q_{k-1})\omega(Q_{k-1})}{10^{10}}\rfloor }{-\log \frac{\Psi(Q_{k})\omega(Q_{k})}{10^{8}} }\\
&\geq \liminf_{k\to\infty}\frac{\log Q_{k}^{2}\Psi(Q_{k-1})\omega(Q_{k-1})}{-\log \Psi(Q_{k})\omega(Q_{k})}\\
&=\liminf_{k\to\infty}\frac{\log Q_{k}^{2}}{-\log \Psi(Q_{k})\omega(Q_{k})}\\
&=\frac{2}{\lambda(\Psi,\omega)}.
\end{align*}
In the penultimate line we used \eqref{eq:logs disappear}. In the final line we used \eqref{eq:View lower order}. Summarising, we have shown that \eqref{eq:WTS lower bound} holds, and so the proof of Theorem \ref{thm:main2} is complete. Note that the changes in constants that appear in the proof of Theorem~\ref{thm:main} make no difference to the final dimension calculation. Thus this theorem follows similarly.

\subsection*{Acknowledgments}
The first author was supported by an EPSRC New Investigators Award (EP/W003880/1). The second author is supported by a Leverhulme Trust Early Career Research Fellowship ECF-2024-401.

\end{document}